\documentclass[12pt]{amsart}
\usepackage{amsmath,amsthm,amssymb}

\usepackage{dsfont}
\usepackage{mathrsfs} 
\usepackage[all,cmtip]{xy}
\usepackage{patchcmd}
\input{diagxy}
\usepackage{hyperref}
\usepackage{bm}
\usepackage{mathtools}
\usepackage{enumerate}
\usepackage{stmaryrd}
\usepackage[usenames,dvipsnames,svgnames,table]{xcolor}
\usepackage[normalem]{ulem}
\usepackage[top=2.3cm, bottom=2.3cm, left=2cm, right=2cm]{geometry}
\allowdisplaybreaks

\newcommand{\memph}[1]{{\emph{#1}}}

\makeatletter
\patchcommand\@starttoc{\begin{quote}}{\end{quote}}
\def\@tocline#1#2#3#4#5#6#7{\relax
  \ifnum #1>\c@tocdepth 
  \else
    \par \addpenalty\@secpenalty\addvspace{#2}%
    \begingroup \hyphenpenalty\@M
    \@ifempty{#4}{%
      \@tempdima\csname r@tocindent\number#1\endcsname\relax
    }{%
      \@tempdima#4\relax
    }%
    \parindent\z@ \leftskip#3\relax \advance\leftskip\@tempdima\relax
    \rightskip\@pnumwidth plus4em \parfillskip-\@pnumwidth
    #5\leavevmode\hskip-\@tempdima
      \ifcase #1
       \or\or \hskip 1em \or \hskip 2em \else \hskip 3em \fi%
      #6\nobreak\relax
    \dotfill\hbox to\@pnumwidth{\@tocpagenum{#7}}\par
    \nobreak
    \endgroup
  \fi}
\makeatother

 \theoremstyle{plain}
 \newtheorem{thm}{Theorem}[section]
 \newtheorem{cor}[thm]{Corollary}
 \newtheorem{lem}[thm]{Lemma}
 
 \newtheorem{prop}[thm]{Proposition}

\theoremstyle{definition}
 \newtheorem{defn}[thm]{Definition}
 \newtheorem{test}[thm]{Test}

\theoremstyle{remark}
 \newtheorem{rem}[thm]{Remark}

 \newtheorem{ter}[thm]{Terminology}
 \newtheorem{nota}[thm]{Notation}
 \newtheorem{conv}[thm]{Convention}
 \newtheorem{exam}[thm]{Example}

 \numberwithin{equation}{section}

\theoremstyle{plain}

\newtheorem*{thm*}{Theorem}

\newcommand{\VF}{\textup{VF}}
\newcommand{\ACVF}{\textup{ACVF}}
\newcommand{\RV}{\textup{RV}}
\newcommand{\DC}{\textup{DC}}
\newcommand{\MM}{\mdl M}

\newcommand{\OO}{\mdl O}

 \DeclareMathOperator{\ran}{ran}
 \DeclareMathOperator{\dom}{dom}
 
 \DeclareMathOperator{\id}{id}

 \DeclareMathOperator{\lh}{lh}

 \DeclareMathOperator{\aut}{Aut}

 \DeclareMathOperator{\supp}{supp}
 
 \DeclareMathOperator{\acl}{acl}
 \DeclareMathOperator{\dcl}{dcl}
 \DeclareMathOperator{\pr}{pr}
 \DeclareMathOperator{\alg}{ac}

 \DeclareMathOperator{\mgl}{GL}

\DeclareMathOperator{\jcb}{Jcb}
\DeclareMathOperator{\K}{\mathds{k}}

\DeclareMathOperator{\res}{res}  

\def\XXint#1#2#3{{\setbox0=\hbox{$#1{#2#3}{\int}$}
\vcenter{\hbox{$#2#3$}}\kern-.5\wd0}}

\newcommand{\Z}{\mathds{Z}}
\newcommand{\A}{\mathds{A}}
\newcommand{\C}{\mathds{C}}

\newcommand{\G}{\mathds{G}}

\newcommand{\Q}{\mathds{Q}}
\newcommand{\N}{\mathds{N}}

\newcommand{\R}{\mathds{R}}

\newcommand{\omin}{$o$\nobreakdash}

\newcommand{\cmin}{$C$\nobreakdash}

\newcommand{\T}{$T$\nobreakdash}

\newcommand{\dand}{\quad \text{and} \quad}
\newcommand{\tand}{~\text{and}~}

\newcommand{\gB}{\mathfrak{B}}

\newcommand{\ga}{\mathfrak{a}}
\newcommand{\gb}{\mathfrak{b}}
\newcommand{\gc}{\mathfrak{c}}

\newcommand{\gh}{\mathfrak{h}}

\newcommand{\go}{\mathfrak{o}}
\newcommand{\gp}{\mathfrak{p}}
\newcommand{\gq}{\mathfrak{q}}

\newcommand{\0}{\emptyset}

\DeclareMathAlphabet{\mathpzc}{OT1}{pzc}{m}{it}

\newcommand{\dbra}[1]{
  \,[\mkern-6.8mu[\, #1 \,]\mkern-6.8mu]}
\newcommand{\dpar}[1]{
   (\mkern-4mu (#1 )\mkern-4mu )}

 \DeclarePairedDelimiter\abs{\lvert}{\rvert}

 \newcommand{\lan}[3]{\mathcal{L}_{#1 \textup{#2} #3}}

\newcommand{\mdl}[1]{\mathcal{#1}}  
\newcommand{\bb}[1]{\mathbb{#1}}

\newcommand{\limplies}{\rightarrow}

\newcommand{\rest}{\upharpoonright}

\newcommand{\fun}{\longrightarrow}
\newcommand{\efun}{\longmapsto}
\newcommand{\sub}{\subseteq}

\newcommand{\mi}{\smallsetminus}

\newcommand{\la}{\langle}
\newcommand{\ra}{\rangle}

\providecommand\given{}
\newcommand\SetSymbol[1][]{%
\nonscript \: #1 \vert
\allowbreak
\nonscript\:
\mathopen{}}
\DeclarePairedDelimiterX\set[1]\{\}{%
\renewcommand\given{\SetSymbol[\delimsize]}
#1
}

\newbox\gnBoxA
\newdimen\gnCornerHgt
\setbox\gnBoxA=\hbox{$\ulcorner$}
\global\gnCornerHgt=\ht\gnBoxA
\newdimen\gnArgHgt

\def\code #1{%
        \setbox\gnBoxA=\hbox{$#1$}%
        \gnArgHgt=\ht\gnBoxA%
        \ifnum \gnArgHgt<\gnCornerHgt
                \gnArgHgt=0pt%
        \else
                \advance \gnArgHgt by -\gnCornerHgt%
        \fi
        \raise\gnArgHgt\hbox{$\ulcorner$} \box\gnBoxA %
                \raise\gnArgHgt\hbox{$\urcorner$}}

\newcommand{\mVF}{{\mu}{\VF}}
\newcommand{\mgVF}{{\mu} {\VF}}
\newcommand{\mRV}{{\mu}{\RV}}
\newcommand{\mgRV}{{\mu}{\RV}}
\newcommand{\mG}{{\mu}{\Gamma}}
\newcommand{\RES}{\textup{RES}}
\newcommand{\mgRES}{{\mu}{\RES}}
\newcommand{\mRES}{{\mu}{\RES}}
\newcommand{\isp}{\textup{I}_{\textup{sp}}}
\newcommand{\dsp}{\textup{D}_{\textup{sp}}}

\newcommand{\mgisp}{\mu \isp}
\newcommand{\mispdb}{\mu \dsp}

\newcommand{\mgE}{{\mu}{\bb E}}
\newcommand{\mgL}{{\mu}{\bb L}}
\newcommand{\mL}{{\mu}{\bb L}}

\DeclareMathOperator{\rv}{rv}

\DeclareMathOperator{\csn}{csn}
\DeclareMathOperator{\rcsn}{\overline {csn}}
\DeclareMathOperator{\vv}{val}

\DeclareMathOperator{\gsk}{\mathbf{K}_+}
\DeclareMathOperator{\ggk}{\mathbf{K}}

\DeclareMathOperator{\sggk}{!\mathbf{K}}

\DeclareMathOperator{\fn}{FN}
\DeclareMathOperator{\fib}{fib}
\DeclareMathOperator{\fin}{fin}

\DeclareMathOperator{\rad}{rad}

\DeclareMathOperator{\vrv}{vrv}
\DeclareMathOperator{\RVH}{RVH}

\DeclareMathOperator{\can}{\mathbf{c}}

\DeclareMathOperator{\gal}{Gal}
\DeclareMathOperator{\ito}{int}

\DeclareMathOperator{\db}{d \! b}
\DeclareMathOperator{\bdd}{b \! d}
\newcommand{\var}{\textup{Var}}

\DeclareMathOperator{\RVV}{RV_\infty}

\newcommand{\ddx}{\tfrac{d}{d x}}
\DeclareMathOperator{\bfk}{\bf k}

\DeclareMathOperator{\puC}{\tilde \C}

\newcommand{\loc}{\textup{loc}}
\DeclareMathOperator{\tor}{Tor}

\DeclareMathOperator{\Vol}{Vol}
\newcommand{\Var}{\textup{Var}}

\DeclareMathOperator{\tbk}{tbk}

\newcommand{\gmv}{\ggk^{\hat \mu}\var_{\C}}

\newcommand{\ccirc}{\circ \circ}

\DeclareMathOperator{\Cl}{Cl}

\begin{document}

\title[Bounded integral and motivic Milnor fiber]{Bounded integral and motivic Milnor fiber}

\author[Arthur Forey]{Arthur Forey}
\address[A. Forey]{D-MATH, ETH Z\"urich, R\"amistrasse 101, CH-8092 Z\"urich, Switzerland}
\email{arthur.forey@math.ethz.ch}

\author[Yimu Yin]{Yimu Yin}
\address[Y. Yin]{Pasadena, California}
\email{yimu.yin@hotmail.com}

\thanks{Many thanks to Ehud Hrushovski and Fran\c cois Loeser for encouragements and enlightening discussions during the preparation of this manuscript. We also thanks Johannes Nicaise for suggesting to work on the Fubini theorem. Thanks also to Raf Cluckers and Martin Hils for inspiring discussions and comments. A.F.\ is supported by a DFG-SNF lead agency program grant
  (grant 200020L\_175755). }


\begin{abstract}
We construct a new Hrushovski-Kazhdan style motivic integral, the so-call bounded integral, that interpolates the integral with volume forms and that without volume forms. This is done within the framework of model theory of algebraically closed valued fields of equicharacteristic zero. As the main application, we rectify and refine some results of Hrushovski and Loeser on how to construct motivic Milnor fiber without resolution of singularities.
\end{abstract}

\maketitle

\tableofcontents

\pagebreak

\section{Introduction}
Recent years have seen significant development in applying the Hrushovski-Kazhdan integration theory to the  study of a variety of topics in geometry and topology, especially those around Denef-Loeser's motivic Milnor fiber. Such an approach  was first articulated in  \cite{hru:loe:lef} for the purpose of  finding a more conceptual (read ``resolution-free'') construction of the complex  motivic Milnor fiber, among other things; unfortunately, it contains a substantial technical oversight. The main goal of this paper is to present a new Hrushovski-Kazhdan style integral,  the so-call \memph{bounded integral}, that interpolates (or refines) the integral with volume forms and that without volume forms, and thereby fully rectify the results in  \cite{hru:loe:lef}.
The overhaul taken up here constitutes a comprehensive expansion of the theoretical framework of the Hrushovski-Kazhdan style integration that should prove amenable to new applications in the future.

We now explain our construction in more detail.
In this introduction, for concreteness, we work in  the field $\puC = \bigcup_{m \in \Z^+} \C \dpar{ t^{1 / m} }$ of complex Puiseux series. This field is the algebraic closure of the field $\C \dpar t$ of complex Laurent series. A typical element  takes the form $x = \sum_{n \in \Z} a_n t^{n/m}$ for some $m \in \Z^+$ such that its support $\supp(x) = \set{n/m \in \Q \given a_n \neq 0 }$ is well-ordered, in other words, there is a $q \in \Q$ such that $a_n = 0$ for all $n/m < q$. We think of $\bfk \coloneqq \C$ as a subfield of $\puC$ via the embedding $a \efun at^0$. The map $\vv: \tilde \C^\times \fun \Q$ given by $x \efun \min\supp(x)$ is indeed a valuation, and its valuation ring $\OO$ consists of those series $x$ with $\min\supp(x) \geq 0$ and its maximal ideal $\MM$ of those series $x$ with $\min\supp(x) > 0$. Its residue field $\K$ admits a section onto $\bfk$ and hence is isomorphic to $\C$. For a series $x = \sum_{n \in \Z} a_n t^{n/m} \in \puC$ with $\vv(x) = p/m$, let $\rv(x) = a_p t^{p/m}$, which is called the \memph{leading term} of $x$.

It is well-known that $(\puC, \OO)$ is an algebraically closed valued field of equicharacteristic zero. Thus, we may  think of it as an $\lan{}{RV}{}$-model of the first-order theory $\ACVF(0,0)$ of such fields, with the parameter space (the ``ground field'') $\bb S = \C \dpar t$.  Here  the two sorts $\VF$, $\RV$ of the language $\lan{}{RV}{}$ are  interpreted, respectively, as $\puC$, $\tilde \C^\times / 1 + \MM$ (or, equivalently, the set of leading terms) and the cross-sort function $\rv : \VF^\times \fun \RV$ as the quotient map (or the leading term map described above). The obvious epimorphism from $\RV$ onto the value group $\Gamma = \Q$, also referred to as the $\Gamma$-sort, with the kernel $\K^\times$, is denoted by $\vrv$. See \cite[\S~2]{Yin:special:trans} for the precise definitions. The Galois group $\gal(\puC / \C \dpar t)$, which  is also the group of automorphisms over $\C \dpar t$ of $\puC$ as a model of $\ACVF(0,0)$, may be identified with the procyclic group $\hat \mu$ of roots of unity (the limit of the inverse system of groups $\mu_n$ of $n$th roots of unity).

The category $\VF_*$ essentially consists of the definable subsets of $\VF^n$, $n \geq 0$, as objects and the definable bijections between them as morphisms. The category $\RV[k]$ essentially consists of the definable subsets of $\RV^k$ as objects and the definable bijections between them as morphisms. The category  $\RV[*]$  is the coproduct of $\RV[k]$, $k \geq 0$, and hence is equipped with a gradation by ambient dimensions. The actual definitions will be recalled in \S~\ref{sec:HK:main}.

To such categories one associates Grothendieck rings $\ggk\VF_*$ and $\ggk\RV[*]$, roughly constructed as the free group generated by isomorphism classes of objects, subject to the usual scissor relation $[A \mi B] + [B] = [A]$ for any objects $A \sub B$, and the product being given by fiber product. See  Definition~\ref{def:Kgroup} for details.

Pulling back from $\RV$ to $\VF$ along the map $\rv: \VF^\times \fun \RV$ yields a ring homomorphism between the   Grothendieck rings
\begin{equation}\label{lifting:hom}
\bb L : \ggk\RV[*] \fun \ggk\VF_*
\end{equation}
One of the  main results of \cite{hrushovski:kazhdan:integration:vf} states that $\bb L$ is surjective with the kernel $(\bm P-1)$, where $\bm P$ stands for the element $[\{1\}] - [\rv(\MM \mi 0)]$  in $\ggk \RV[1]$ (so the principal ideal $(\bm P - 1)$ is not homogenous), and hence, inverting, we get a canonical isomorphism
\[
\int : \ggk\VF_* \fun \ggk\RV[*]/(\bm P-1).
\]

The structure of $\ggk \RV[*]$ can be significantly elucidated. To wit, it is isomorphic to a tensor product of two other Grothendieck rings $\ggk \RES[*]$ and $\ggk \Gamma[*]$, where  $\RES[*]$ is  the category of twisted constructible sets in the residue field $\K$  and $\Gamma[*]$ is the category of definable sets in the value group $\Gamma$ (as an \omin-minimal group), both are graded by ambient dimensions. The objects of $\RES[*]$ are twisted because the short exact sequence
\[
1 \fun \K^\times \fun \RV \fun \Gamma \fun 0
\]
does not admit a natural splitting, and $\ggk \Gamma[*]$ is not the Grothendieck ring of \omin-minimal groups because not all definable bijections are admitted as morphisms. Anyway, we have two retractions from $\ggk \RV[*]$ onto a quotient $\sggk \RES$ of $\ggk \RES$ (the gradation is forgotten) that vanish on $(\bm P-1)$, reflecting the fact that there are two Euler characteristics in the $\Gamma$-sort; these are denoted by $\bb E_{b}$, $\bb E_{g}$.

The objects of $\VF_*$, $\RV[*]$ may be equipped with $\Gamma$-volume forms, and the resulting categories are denoted by $\mgVF[*]$, $\mgRV[*]$. There is again a  canonical isomorphism
\[
\int^\mu : \ggk\mgVF[*] \fun \ggk\mgRV[*]/(\bm P),
\]
where the principal ideal $(\bm P)$ is now homogenous. The ring $\ggk\mgRV[*]$ can  be expressed as a tensor product of  $\ggk \mgRES[*]$ and $\ggk \mG[*]$, which induces two retractions $\mgE_ {b}$, $\mgE_{g}$ from $\ggk\mgRV[*]$ onto $\sggk \RES[*]$ that vanish on $(\bm P)$; see the right column of the diagram (\ref{diag-diamond-interpol}).

We think of $\int$ as the universal additive invariant (or generalized Euler characteristic) of definable sets and $\int^\mu$ as a motivic integral. They are not directly related, though: in  $\ggk\VF_*$, two objects are identified if there is a definable bijection between them, whereas in $\ggk\mgVF[*]$, two objects are identified if there is a measure-preserving bijection between two dense subobjects.

The main result of this paper is  that the two isomorphisms $\int$, $\int^\mu$ can be interpolated by a third one, the bounded integral (Theorem \ref{main:prop}):
\[
\int^\diamond : \ggk \mgVF^\diamond[*] \fun \ggk  \mgRV^{\db}[*] / (\bm P_\Gamma).
\]
The category $\mgVF^\diamond[*]$ consists of those \memph{proper invariant} objects of $\mgVF[*]$, which roughly means a bounded definable set that is a union of open boxes of valuative radius $\gamma$, $\gamma \in \Gamma$. The category $\mgRV^{\db}[*]$ consists of those \memph{doubly bounded} objects of $\mgRV[*]$, in the sense that its image in the value group is doubly bounded. See \S~\ref{sec:inv:cov} for the precise definitions. Note that $\mgVF^{\diamond}[*]$  is also graded since, as in classical measure theory, gradation by ambient dimensions is a necessity in the presence of volume forms (a curve has no volume if considered as a subset of a surface). The ideal $(\bm P_\Gamma)$ is homogenous but no longer principal; it is generated by the elements $\bm P_\gamma \in \ggk \mgRV^{\db}[1]$, one for each $\gamma \in \Z^+$, defined as follows. Let $\MM_\gamma$ be the open  disc around $0$ with valuative radius $\gamma$ and $t_\gamma \in \vrv^{-1}(\gamma)$ a definable element. Then $\bm P_\gamma$ is the element
\[
[\rv(\MM \mi \MM_{\gamma})] + [\{t_\gamma\}] - [\{1\}]
\]
with the constant volume form $0$, and it does not depend on the choice of $t_\gamma$. We can express $\ggk  \mgRV^{\db}[*]$ as a tensor product of $\ggk \mRES[*]$ and $\ggk \mG^{\db}[*]$. Since the objects of $\mG^{\db}[*]$ are doubly bounded, the two Euler characteristics coincide and consequently there is only one retraction $\bb E^{\diamond}$ from $\ggk  \mgRV^{\db}[*]$ onto $\sggk \RES$ that vanishes on $(\bm P_\Gamma)$.

The situation is summarized in the  commutative diagram (\ref{diag-diamond-interpol}). More precisely speaking, the bounded integral $\int^\diamond$ is a common refinement of $\int$ and $\int^\mu$: it is finer than $\int$ because it carries volume forms, and it is finer than $\int^\mu$ because it only admits those measure-preserving morphisms that are defined everywhere (as opposed to on dense subobjects).
\begin{figure}[htb]
\begin{equation}\label{diag-diamond-interpol}
\bfig
  \hSquares(0,0)/<-`->`->`->`->`<-`->/<400>[{\ggk \VF_*}`{\ggk \mgVF^\diamond[*]}`{\ggk \mgVF[*]}`{\ggk \RV[*] / (\bm P - 1)}`{\ggk \mgRV^{\db}[*] / (\bm P_\Gamma)}`{\ggk \mgRV[*] / (\bm P)}; ``\int`\int^{\diamond}`\int^\mu``]
\efig
\end{equation}
\end{figure}


%
%

The general strategy for the construction of the bounded integral $\int^\diamond$ follows roughly that of the original Hrushovski-Kazhdan integral. One first defines a lifting homomorphism from $\ggk  \mgRV^{\db}[*]$ to $\ggk \mgVF^\diamond[*]$,  shows that it is surjective, then studies its kernel.
On the technical side, in order to stay within the realm of proper invariant sets at each step of the construction, one is led to the notion of a \memph{proper covariant} bijection, first formulated in \cite{hru:loe:lef}. It goes as follows (see Definition~\ref{prop:pseu} for more details). For $\alpha \in \Gamma$, let $\pi_\alpha : \VF\fun \VF/\MM_\alpha$ be the projection map;  if $\alpha\in \Gamma^n$ then $\pi_{\alpha}$ denotes the product of the maps $\pi_{\alpha_i}$. A definable bijection $f$ from $A\sub \VF^n$ onto  $B \sub \VF^m$ is said to be proper covariant if there are $\alpha \in \Gamma^n$, $\beta \in \Gamma^n$, and a function $f_\downarrow : \pi_{\alpha}(A) \fun \pi_{\beta}(B)$ such that $A$ is $\alpha$-invariant, $B$ is $\beta$-invariant, and  $\pi_\beta \circ f = f_\downarrow \circ \pi_\alpha$. In order to show the existence of such proper covariant bijections of a very special kind, we establish in \S~\ref{sec:cont:fib:prop} some general properties of definable continuous functions, reminiscent of those of definable continuous functions in \omin-minimal theories (see, for instance,  \cite[Chapter~6]{dries:1998}).

Next, we turn to the construction of the motivic Milnor fiber attached to a nonconstant polynomial function $f : (\C^d, 0) \fun (\C, 0)$,  where $0$ is assumed to be a singular point, that is, $\nabla f (0) = 0$.

By an (algebraic) variety over a field $\bfk$, we mean a reduced separated $\bfk$-scheme of finite type. We denote by $\Var_{\bfk}$ the category of  varieties over $\bfk$.
Let $\ggk \var_{\C}$ be the Grothendieck ring of complex algebraic varieties and $\ggk^{\hat \mu} \var_{\C}$ the corresponding Grothendieck ring with good $\hat \mu$-actions. Here, in addition to the factorization requirement for the $\hat \mu$-actions, the qualifier ``good'' also means that a $\C$-vector space endowed with a linear $\hat \mu$-action is identified with the same  $\C$-vector space but endowed with the trivial $\hat \mu$-action.

Let $\mathscr L$ be the space of formal arcs on $\C^d$ at $0$. So each element in $\mathscr L$ is of the form $\gamma(t) = (\gamma_1(t), \ldots, \gamma_d(t))$, where $\gamma_i(t)$ is a complex formal power series and $\gamma_i(0) = 0$. For each $m \in \Z^+$, let $\mathscr L_m$ be the space of such arcs modulo $t^{m+1}$ (also referred to as ``truncated arcs''). Consider the following locally closed subset of $\mathscr L_m$:
\begin{equation*}
  \mdl X_{f,m}  = \set{\gamma(t) \in \mathscr L_m \given f(\gamma(t)) = t^m \mod t^{m+1} }.
\end{equation*}
It may be viewed in a natural way as the set of closed points of an algebraic variety over $\C$ and, as such, carries a natural $\mu_{m}$-action. The \memph{motivic zeta function} attached to $f$ is then the generating series whose coefficients are in effect the ``$\hat \mu$-equivariant motivic volumes'' of the sets of truncated arcs above:
\[
 Z_f(T) = \sum_{m \in \Z^+} [\mdl X_{f,m}] [\A]^{-nd}T^m \in \gmv[[\A]^{-1}] \dbra{T}.
\]
It is  shown in \cite{denefloeser:arc, den:loe:2002} that $Z_f(T)$ is rational and belongs to $\gmv[[\A]^{-1}][T]_\dag$, where the subscript ``$\dag$'' indicates that the ring is localized by  a certain multiplicative family of elements. The \memph{motivic Milnor fiber}
\[
\mathscr S_{f} = - \lim_{T \limplies \infty} Z_{f}(T)
\]
is then extracted from this rational expression via a formal process of sending the variable $T$ to infinity (this process is also summarized in \cite[\S~8.4]{hru:loe:lef}). Of course, to justify calling $\mathscr S_{f}$ a ``Milnor so-and-so'' one needs to show, at the very least, that invariants of the \memph{topological Milnor fiber} $F_{f}$ attached to $f$ can be recovered from it.  This is indeed the case for, say, the Euler characteristic and the Hodge characteristic.

Originally,  the proofs that $Z_f(T)$ is rational and  the Euler (or Hodge) characteristics coincide both rely on resolution of singularities.  More recently, in \cite{hru:loe:lef}, these results are established by way of a more conceptual construction. To begin with, $Z_{f}(T)$ is expressed as
\begin{equation*}\label{intro:zeta:mot}
\sum_{m \in \Z^+} H_m([\mdl X_f]) T^m,
\end{equation*}
where the coefficients $H_m([\mdl X_f])$ are certain integrals of definable sets that take values in the ring $\gmv [[\A]^{-1}]$,  and the so-called \memph{nonarchimedean Milnor fiber} of $f$ is the definable set
\[
\mdl X_f = \set{x \in \MM^d \given \rv(f(x)) = \rv(t) }.
\]
Formulated in this way, the rationality of $Z_{f}(T)$ essentially follows from  certain computation rules of (convergent) geometric series. That the Euler characteristics of $\mathscr S_{f}$ and $F_{f}$ coincide follows from the fact that we can express both the Euler characteristic of each coefficient of $Z_{f}(T)$ and that of $F_{f}$ in terms of traces of the monodromy action on the cohomological groups of $F_{f}$, where the first expression relies on  the resolution-free proofs of the A'Campo-Denef-Loeser formula (this is the main point of \cite{hru:loe:lef}) and quasi-unipotence of local monodromy (see \cite[Remark~8.5.5]{hru:loe:lef}). However, technical difficulties, arising from the fact that none of the morphisms (8.1.4), (8.1.6), and (8.1.7) in \cite{hru:loe:lef} is surjective, contrary to what is claimed in \cite[Proposition~10.10(2)]{hrushovski:kazhdan:integration:vf}, make their construction of the integral $H_m$ incomplete.

To remedy this, as an application of our main result, we complete their construction by recasting it in the formalism of the bounded integral $\int^\diamond$. More specifically, recall from \cite[\S~4.3]{hru:loe:lef} the canonical isomorphism
\[
\Theta : \sggk \RES \fun \gmv,
\]
which is essentially the operation of ``twisting back'' via a reduced cross-section (Definition~\ref{twistback}). Let  $\gmv[[\A]^{-1}]_{\loc}$ be the localization of $\gmv[[\A]^{-1}]$ by the multiplicative family generated by the elements $1 - [\A]^{-i}$, $i \in \Z^+$. Let $\ggk^{\hat \mu} \var_{\C}[[\A]^{-1}][T]_{\dag}$ be the localization of $\gmv[[\A]^{-1}][T]$ by the multiplicative family generated by the elements $1 - [\A]^{a}T^b$, where $a \in \Z$ and $b \in \Z^+$; it may be regarded as a subring of  $\ggk^{\hat \mu} \var_{\C}[[\A]^{-1}] \dbra T$. Then we have a  commutative diagram (\ref{diag-diamond-zeta-function}).
\begin{figure}[htb]
\begin{equation}\label{diag-diamond-zeta-function}
\bfig
  \iiixiii(0,400)/<-`->`->`->`->`<-`->`->`->`->``->/<1200,400>[{\ggk \VF_*}`{\ggk^\natural \mgVF^\diamond[*]}`{\ggk^\natural \mgVF[*]}`{\ggk \RV[*] / (\bm P - 1)}`{\ggk^\natural \mgRV^{\db}[*] / (\bm P_\Gamma)}`{\ggk^\natural \mgRV^{\bdd}[*] / (\bm P)}`\gmv`\gmv{[[\A]^{-1}][T]_{\dag}}`\gmv{[[\A]^{-1}]_{\loc}[T]_{\dag}}; ``\int`\int^{\diamond}`\int^\mu```\Theta \circ \bb E_b`Z`Z``]
  \square(0,0)|aamb|/`->`->`=/<1200,400>[\gmv`\gmv{[[\A]^{-1}][T]_{\dag}}`\gmv{[[\A]^{-1}]}`\gmv{[[\A]^{-1}]};
  ``-\lim`]
  \square(1200,0)/``->`->/<1200,400>[\gmv{[[\A]^{-1}][T]_{\dag}}`\gmv{[[\A]^{-1}]_{\loc}[T]_{\dag}}`
  \gmv{[[\A]^{-1}]}`\gmv{[[\A]^{-1}]_{\loc}};``-\lim`]
 \efig
\end{equation}
\end{figure}
The superscript ``$\natural$'' means that the integral is restricted to the subring of a sort of ``integrable'' sets. The map $Z$ assigns, to each object $[\bm U] \in \ggk^\natural \mgRV^{\db}[*]$,  a power series
\[
Z([\bm U])(T) = \sum_{m \in \Z^+} \bm H_m([\bm U]) T^m \in \ggk^{\hat \mu} \var_{\C}[[\A]^{-1}] \dbra T,
\]
that is, the motivic zeta function of $\bm U$, where $\bm H_m$ is the rectified version of $H_m$. Note that $Z$ is a ring homomorphism with respect to the Hadamard product in the target. It turns out that $Z([\bm U])(T)$ is rational and belongs to $\gmv[[\A]^{-1}][T]_{\dag}$. The formal process $T \limplies \infty$ mentioned above then lands us in $\gmv[[\A]^{-1}]$; this is the map $-\lim$. If we take $[\bm U] / (\bm P_\Gamma) = \int^{\diamond}([\mdl X_f])$ then the end result is the motivic Milnor fiber $\mathscr S_{f}$.

Since $\mdl X_f$ also lives in $\VF_*$, we can recover   $\mathscr S_{f}$ directly as $(\Theta \circ \bb E_b \circ \int)([\mdl X_f])$ in $\gmv$, without even inverting $[\A]$. This is indeed a salient point. However, for comparison with the Denef-Loeser construction, we have to go down to  $\gmv[[\A]^{-1}]$, since that is where the coefficients of the motivic zeta function live.

The composite  $\Theta \circ \bb E_b \circ \int$ is referred to as the \memph{motivic volume operator}, usually denoted by $\Vol$, and is already used implicitly in \cite{hrushovski:kazhdan:integration:vf}. Since its reemergence in \cite{hru:loe:lef}, it has been featured in a growing list of applications, see \cite{thuong-int-conj, nicaise_ps_tropical, Nic:Pay:trop:fub, forey_virtual_2017, NS-deg-stab-rat}.

The nonarchimedean Milnor fiber $\mdl X_f$ is closely related to the analytic Milnor fiber $\mathcal{F}^{\mathrm{an}}_{f}$ introduced in \cite{nicaise:sebag:motivic:serre}. The $\puC$-rational points of $\mathcal{F}^{\mathrm{an}}_{f}$ may be understood as an object of  $\VF_*$ (but not $\mgVF^\diamond[*]$) and hence, as such, admits a motivic volume $\Vol(\mathcal{F}^{\mathrm{an}}_{f})$. This is indeed the motivic Milnor fiber $\mathscr S_{f}$. This fact is already proven  in \cite{nicaise_ps_tropical}, but without taking the $\hat \mu$-actions into account, and also in \cite{Nic:Pay:trop:fub, forey_virtual_2017}. The arguments there all rely on resolution of singularities.

The bounded integral $\int^\diamond$ is also used in \cite{FY:Mot-Int-Mil-Fib} to obtain similar results on real motivic Milnor fiber.

In \S~\ref{csndag}, as a computational tool, we introduce the notion of a twistoid. There we prove two Fubini theorems that are not related to the rest of the paper. The first is  a version of the  tropical motivic Fubini theorem of \cite{Nic:Pay:trop:fub}. In a  simplified form (but of central interest), the second states the following.

\begin{thm*}
Let $f :A \fun B$, $f : A' \fun B$ be morphisms of varieties over $\C$. If $\Vol_b([f^{-1}(b)])=\Vol_b([f'^{-1}(b)])$ for all $b\in B$  then $\Vol([A])=\Vol([A'])$. Consequently, if $\Vol_b([f^{-1}(b)])= 0$ for all $b\in B$ then $\Vol([A]) = 0$.
\end{thm*}

Here the operator $\Vol_b$ is similar to $\Vol$, but over the parameter space $\C\dpar t (b)$. It is essential to consider $b$ in a sufficiently saturated model $\bb U$ of $\ACVF(0,0)$ instead of $\puC$, allowing the hypothesis to be tested at generic points.  A more general version, with $A$, $B$ possibly arbitrary definable sets, appears as Theorem \ref{fubini}, but one needs to be more careful about the definition of $\Vol_b$.

Some property of this form is needed in \cite{loeser2019motivic}.
We expect to find more applications of these results in motivic Donaldson-Thomas theory, along the line of \cite{Nic:Pay:trop:fub}.

The paper is organized as follows. In \S~\ref{sec-tech} we recall and extend certain facts on definable sets in $\ACVF(0,0)$ that are needed later on. In particular,  some new results about continuous definable functions in this environment are obtained, which could be of independent interest, such as fusion of continuous fibers (Proposition~\ref{prop:continuity:fiberwise}) and curve selection (Lemma~\ref{curve:select}), which closely resemble their counterparts in the \omin-minimal setting. In the short \S~\ref{sec:HK:main} is introduced notations and results from the original Hrushovski-Kazhdan paper \cite{hrushovski:kazhdan:integration:vf}. In \S~\ref{sec:retract} we show how to replace classes of polytopes with numbers  and thereby construct a retraction map from the Grothendieck ring in $\RV$ to that of varieties. Invariance and covariance are defined and investigated  in \S~\ref{sec-pscb}, where the surjectivity of the lifting homomorphism (\ref{lifting:hom}) is also established. Then in \S~\ref{section:int} we study its kernel and finish the construction of the bounded integral $\int^\diamond$. Finally, \S~\ref{sec:spec:pui} is devoted to the construction of motivic zeta function and motivic Milnor fiber, rectifying the result in \cite{hru:loe:lef}.

\section{Technicalities}
\label{sec-tech}

\subsection{Setting and notation}\label{sec:notation}

\begin{nota}[Coordinate projections]\label{indexing}
For each $n \in \N$, let $[n]$ denote the set $\{1, \ldots, n\}$. For any $E \sub [n]$, we write $\pr_E(A)$, or even $A_E$, when there is no danger of confusion, for the projection of $A$ into the coordinates contained in $E$. It is often more convenient to use simple standard descriptions as subscripts. For example, if $E$ is a singleton $\{i\}$ then we shall always write $E$ as $i$ and $\tilde E \coloneqq [n] \mi E$ as $\tilde i$; similarly, if $E = [i]$, $\set{k \given  i \leq k \leq j}$, $\set{k \given  i < k < j}$, $\{\text{all the coordinates in the sort $S$}\}$, etc., then we may write $\pr_{\leq i}$, $\pr_{[i, j]}$, $A_{(i, j)}$, $A_{S}$, etc.; in particular, we shall frequently write $A_{\VF}$ and $A_{\RV}$ for the projections of $A$ into the $\VF$-sort and $\RV$-sort coordinates.

Unless otherwise specified, by writing $a \in A$ we shall mean that $a$ is a finite tuple of elements (or ``points'') of $A$, whose length, denoted by $\lh(a)$, is not always indicated. If $a = (a_1, \ldots, a_n)$ then, for $1 \leq i < j \leq n$, following the notational scheme above, $a_i$, $a_{\tilde i}$, $a_{\leq i}$, $a_{[i, j]}$, $a_{\VF}$, etc., are shorthand for the corresponding subtuples of $a$.

We shall write $\{t\} \times A$, $\{t\} \cup A$, $A \mi \{t\}$, etc., simply as $t \times A$, $t \cup A$, $A \mi t$, etc., when it is clearly understood that $t$ is an element and hence must be interpreted as a singleton in these expressions.

For $a \in A_{\tilde E}$, the fiber $\{b : ( b, a) \in A \} \sub A_E$ over $a$ is often denoted by $A_a$. Note that, in the discussion below, the distinction between the two sets $A_a$ and $A_a \times \{a\}$ is usually immaterial and hence they may and  shall be tacitly identified. In particular, given a function $f : A \fun B$ and $b \in B$, the pullback $f^{-1}(b)$ is sometimes written as $A_b$ as well. This is a special case since functions are identified with their graphs. This notational scheme is especially useful when the function $f$ has been clearly understood in the context and hence there is no need to spell it out all the time.
\end{nota}

From here on we work in a sufficiently saturated model $\bb U$ of $\ACVF(0,0)$, together with a parameter  set $\mathbb{S}$, which is supposed to be a substructure of $\bb U$. So by a definable set  we shall mean an $\bb S$-definable set, unless indicated otherwise.

\begin{nota}
There is a special element $\infty = \rv(0)$ in the $\RV$-sort. For simplicity, we shall write $\RV$ to mean the $\RV$-sort without the element $\infty$, and $\RV_\infty$ otherwise --- that is, $\RV = \rv(\VF^\times)$ and $\RV_\infty = \rv(\VF)$ --- although the difference rarely matters (when it does we will of course provide further clarification).
\end{nota}

\begin{ter}[Sets and subsets]\label{nota:sub}
By a definable set in $\VF$ we mean a definable subset in the sort $\VF$, by which we just mean a subset of $\VF^n$ for some $n$, unless indicated otherwise; similarly for other sorts or structures  in place of $\VF$ that have been clearly understood in the context, such as $\RV_\infty$, $\K$, $\MM$, or any substructure $\bb M$ of $\bb U$. In particular,  a definable set without further qualification means a definable set in $\bb U$, that is, a  definable subset of $\VF^n \times \RV^m_\infty$ for some $n, m \in \N$.
\end{ter}

\begin{ter}[Relations as functions]
We may think of a relation $C$ between two sets $A$ and $B$, that is, a subset $C \sub A \times B$, as a function from $A$ into the powerset $\mdl P(B)$. For instance, if $A \sub \VF^n \times \RV^m$ is a definable set then it may be thought of as a \memph{definable} function from $A_{\VF}$ into  $\mdl P(\RV^m)$. Thus a definable function $f : A \fun \mdl P(B)$ is ontologically the same as the definable set $\bigcup_{a \in A} a \times f(a)$. We say that $f$ is \memph{finitary} if $f(a)$ is finite for every $a \in A$.
\end{ter}

The default topology in $\bb U$ is the valuation topology.

\begin{nota}\label{defn:disc}
A set $\gp \sub \VF^n \times \RV_\infty^m$ of the form $\prod_{i \leq n} \gb_i \times t$ is called an (\memph{open, closed, $\RV$-}) \memph{polydisc} if each $\gb_i$ is an (open, closed, $\RV$-) disc. The \memph{radius} $\rad(\gp)$ of the polydisc $\gp$ is the tuple $(\rad(\gb_1), \ldots, \rad(\gb_n))$. The open and closed polydiscs centered at a point $(a,t) \in \VF^n \times \RV^m$ with radius $\gamma \in \Gamma_\infty^n$ are denoted by $\go((a,t), \gamma)$, $\gc((a,t), \gamma)$, respectively.

The subdiscs $\go(0, \gamma)$, $\gc(0, \gamma)$ of $\VF$ are more suggestively denoted by $\MM_\gamma$, $\OO_\gamma$, respectively, and $\MM_0$, $\OO_0$ are even more suggestively denoted by $\MM$, $\OO$. Write $\RV^{\ccirc}_\infty = \rv(\MM)$ and $\RV^{\ccirc} = \RV^{\ccirc}_\infty \mi \infty$. More generally, for $\gamma \in \Gamma$, denote by $\RV^{\ccirc}_\gamma$ the set $\rv(\MM_{\gamma} \mi 0)$.

The \memph{$\RV$-hull} of a set $A$, denoted by $\RVH(A)$, is the union of all the $\RV$-polydiscs whose intersections with $A$ are nonempty. If $A$ equals $\RVH(A)$ then $A$ is called an \memph{$\RV$-pullback}.
\end{nota}

\begin{ter}\label{rvfiber}
A \memph{$\VF$-fiber} of a set $A$ is a set of the form $A_t$, where $t \in A_{\RV}$; in particular,  a $\VF$-fiber of a function $f : A \fun B$ is a set of the form $f_t$ for some  $t \in f_{\RV}$ (here $f$ also stands for its own graph) which is indeed (the graph of) a function.  We say that $A$ is open if every one of its $\VF$-fibers is, $f$ is continuous if every one of its $\VF$-fibers is, and so on. This is the right point of view since, for instance, it can be readily checked that if $f$, $g$ are continuous functions with $\ran(f) \sub \dom(g)$ then  $g \circ f$ is also continuous.
\end{ter}

\begin{nota}\label{gamwhat}
Semantically we shall treat the value group $\Gamma$ as a definable sort (the $\Gamma$-sort). However, syntactically any reference to $\Gamma$ may be eliminated in the usual way and we can still work with $\lan{}{RV}{}$-formulas for the same purpose. This implies that $\gamma\in \Gamma$ is definable means that $\vv^{-1}(\gamma)$ is definable.

We shall write $\gamma^\sharp$, $\gamma \in \Gamma$, when we want to emphasize that it is the set $\vrv^{-1}(\gamma) \sub \RV$  that is being considered. More generally, if $I$ is a set in $\Gamma$ then we write $I^\sharp = \bigcup\set{\gamma^\sharp \given \gamma \in I}$. Similarly, if $U$ is a set in $\RV$ then $U^\sharp$ stands for $\bigcup\set{\rv^{-1}(t) \given t \in U}$. In particular, $\gamma^{\sharp \sharp}=\set{ x \in \VF \given \vv(x)=\gamma}$.
\end{nota}

\begin{nota}\label{jcbG:op}
Let $U \sub \RV^n \times \Gamma^m$, $V \sub \RV^{n'} \times \Gamma^{m'}$, and $C \sub U \times V$. The \memph{$\Gamma$-Jacobian} of $C$ at $((u, \alpha), (v, \beta)) \in C$, written as $\jcb_{\Gamma} ((u, \alpha), (v, \beta))$, is the element
\[
\jcb_{\Gamma} ((u, \alpha), (v, \beta)) =  - \Sigma (\vrv(u), \alpha) + \Sigma (\vrv(v), \beta),
\]
where $\Sigma (\gamma_1, \ldots, \gamma_n) = \gamma_1 + \ldots + \gamma_n$. If $C$ is the graph of a function then we just write $\jcb_{\Gamma}C(u, \alpha)$ instead of $\jcb_\Gamma((u, \alpha), C(u, \alpha))$.
\end{nota}

%

\begin{conv}\label{topterm}
We may and shall assume that in any $\lan{}{RV}{}$-formula, every $\VF$-sort term (polynomial) occurs in the scope of an instance of the function symbol $\rv$. For example, if $f(x)$, $g(x)$ are polynomials then the formula $f(x) = g(x)$ is equivalent to $\rv(f(x) - g(x)) = \infty$. The polynomial $f(x)$ in $\rv(f(x))$ is referred to as a \memph{top term} (of the $\lan{}{RV}{}$-formula in question).
\end{conv}

\begin{nota}[The definable sort $\DC$ of discs]\label{disc:exp}
At times it will be more convenient to work in the traditional expansion $\bb U^{\textup{eq}}$ of $\bb U$ by all definable sorts. However, for our purpose, a much simpler expansion $\bb U^{\bullet}$ suffices. This expansion has only one additional sort $\DC$ that contains, as elements, all the open and closed discs. Since each point in $\VF$ may be regarded as a closed disc of valuative radius $\infty$, for convenience, we may and occasionally do think of $\VF$ as a subset of $\DC$. Heuristically, we may think of a disc that is properly contained in an $\RV$-disc as a ``thickened'' point of certain stature in $\VF$. For each $\gamma \in \Gamma$, there are two additional cross-sort maps $\VF \fun \DC$ in $\bb U^{\bullet}$, one sends $a$ to the open disc, the other to the closed disc, of radius $\gamma$ that contain $a$.

The expansion $\bb U^{\bullet}$ can help reduce the technical complexity of our discussion. However, as is the case with the definable $\Gamma$-sort, it is conceptually inessential since, for the purpose of this paper, all allusions to discs as (imaginary) elements may be eliminated in favor of objects already definable in $\bb U$.

Whether parameters in $\DC$ are used or not shall be indicated explicitly, if it is necessary, by means of  introducing parameters of the form $\code \ga$ or the phrase ``in $\bb U^{\bullet}$,'' in which case the substructure $\mathbb{S}$ may contain names for discs that may or may not be definable from $\VF(\mathbb{S}) \cup \RV(\mathbb{S})$.

Note that it is redundant to include in $\DC$ discs centered at $0$, since they may be identified with their valuative radii ($\Gamma$ is already treated as a definable sort).

For a disc $\ga \sub \VF$, the corresponding imaginary element in $\DC$ is denoted by $\code{\ga}$ when notational distinction makes the discussion more streamlined; $\code{\ga}$ may be heuristically thought of as the ``name'' of $\ga$. Conversely, a set $D \sub \DC$ is often identified with the set $\set{\ga \given \code \ga \in D}$, in which case $\bigcup D$ denotes a subset of $\VF$.
\end{nota}

\subsection{Of points, discs, and atoms}
The substructure $\bb S$ is \memph{$\VF$-generated} if it is generated solely by elements in $\VF$, similarly for being $(\VF, \RV)$-generated, and so on. If $\bb S$ is $\VF$-generated then the map $\rv$ is surjective in $\bb S$, but its value group $\Gamma(\bb S)$ could still be trivial. If $\Gamma(\mathbb{S})$ is indeed nontrivial then the model-theoretic algebraic closure $\acl \mathbb{S}$ of $\mathbb{S}$ is a model of $\ACVF(0,0)$; see, for instance, \cite[Corollary~4.18]{Yin:QE:ACVF:min}. In that case, we also know what the definable closure $\dcl \mathbb{S}$ of $\mathbb{S}$ is:

\begin{lem}\label{cut:to:hensel:substru}
Suppose that $\mathbb{S}$ is $\VF$-generated and $\Gamma(\mathbb{S})$ is nontrivial. Then $\dcl \mathbb{S} = \mathbb{S}$ if and only if the underlying valued field $(\VF(\mathbb{S}), \OO(\mathbb{S}))$ of $\mathbb{S}$ is henselian.
\end{lem}
\begin{proof}
We have $\VF(\acl \mathbb{S}) = \VF(\mathbb{S})^{\alg}$, where the latter is the field-theoretic algebraic closure of $\VF(\mathbb{S})$. Since the valued field automorphisms of $(\VF(\mathbb{S})^{\alg}, \OO(\acl \mathbb{S}))$ over $(\VF(\mathbb{S}), \OO(\mathbb{S}))$ are in one-to-one correspondence with the $\lan{}{RV}{}$-automorphisms of $\acl \mathbb{S}$ over $\mathbb{S}$ and $\VF(\mathbb{S})$ is the fixed field with respect to these automorphisms if and only if $(\VF(\mathbb{S}), \OO(\mathbb{S}))$ is henselian, we see that $\VF (\dcl \mathbb{S}) = \VF(\mathbb{S})$ if and only if $(\VF(\mathbb{S}), \OO(\mathbb{S}))$ is henselian. By quantifier elimination and routine syntactical inspection, $\RV(\dcl \mathbb{S}) = \RV(\mathbb{S}) $.  The lemma follows.
\end{proof}

We shall impose at a later stage the assumption that $\mathbb{S}$ is $\VF$-generated and (without loss of generality) definably closed,  that is, $\dcl \mathbb{S} = \mathbb{S}$.

\begin{lem}[{\cite[Lemma~4.10]{Yin:QE:ACVF:min}}]\label{finite:VF:project:RV}
For any finite definable set $A$ in $\VF$ there is a definable bijection $f : A \fun U$ for some set $U$ in $\RV$.
\end{lem}

The phenomenon exemplified by this simple lemma is dubbed ``internalizing finite sets.'' Its failure in positive residue characteristics belies the deep obstacles one faces in developing a Hrushovski-Kazhdan integration theory therein; see the discussion around \cite[Lemma~3.9]{hrushovski:kazhdan:integration:vf}.

\begin{lem}[{\cite[Lemma~3.3]{Yin:special:trans}}]\label{dcl:to:ac}
For $a, b \in \VF$ and $t \in \RV$, if $b$ is algebraic over $(a, t)$ then it is algebraic over $a$.
\end{lem}

\begin{lem}[{\cite[Lemma~4.3]{Yin:QE:ACVF:min}}]\label{exchange}
The exchange principle holds in both sorts:
\begin{itemize}
 \item For any $a$, $b \in \VF$, if $a \in \acl(b) \mi \acl(\0)$ then $b \in \acl(a)$.
 \item For any $t$, $s \in \RV$, if $t \in \acl(s) \mi \acl(\0)$ then $s \in \acl(t)$.
\end{itemize}
\end{lem}

\begin{lem}\label{finite:fib:def}
Let $f : A \fun \mdl P(B)$ be a definable finitary function, where $A$, $B$ are subsets of $\VF$. Then there is a finite definable subset $B' \sub B$ such that, for every $b \in B$, the fiber $A_b = \set{a \in A \given b \in f(a)}$ is infinite if and only if  $b \in B'$. The same holds if $A$, $B$ are subsets of $\RV$.
\end{lem}
\begin{proof}
Let $b \in B$ such that $A_b$ is infinite. Then, by compactness, there is an element $a \in A_b \mi \acl(b)$. Since $f$ is finitary, which means $b \in \acl(a)$, we must have, by Lemma~\ref{exchange},  $b \notin \acl(a) \mi \acl(\0)$ and hence $b \in \acl(\0)$. It follows that there is an algebraic subset $B' \sub B$ such that if $A_b$ is infinite then $b \in B'$, for otherwise, by compactness again, there would be a $b \in B$ that lies outside every algebraic subset of $B$ and yet the size of $A_b$ is not bounded by any natural number. Clearly we can adjust $B'$ so that it contains exactly those $b \in  B$ with $A_b$ infinite. The second case is similar.
\end{proof}

\begin{defn}
The \memph{$\VF$-dimension} of a definable set $A$,  denoted by $\dim_{\VF}(A)$, is the largest natural number $k$ such that, possibly after re-indexing of the $\VF$-coordinates, $\pr_{\leq k}(A_t)$ has nonempty interior for some $t \in A_{\RV}$.
\end{defn}

It is a fact that if $A \sub \VF^n$ is definable then $\dim_{\VF}(A)$ equals the Zariski dimension of the Zariski closure of $A$, see for example \cite[\S~3.8]{hrushovski:kazhdan:integration:vf} for more  on dimensions.

\begin{lem}[{\cite[Lemma~4.6]{Yin:special:trans}}]\label{full:dim:open:poly}
Let $A$ be a definable set with $A_{\VF} \sub \VF^n$. Then $\dim_{\VF} (A) = n$ if and only if there is a $t \in A_{\RV}$ such that $A_t$ contains an open polydisc.
\end{lem}

\begin{lem}\label{ball:to:ac}
Let $b \in \VF$ and $\ga \sub \VF^n$ be a polydisc with $\dim_{\VF}(\ga) = n$. If $b$ is algebraic over $\code \ga$ then it is indeed algebraic.
\end{lem}
\begin{proof}
Let $\alpha = \rad(\ga)$. Then, for any $a \in \ga$, $b$ is algebraic over $(a, \alpha)$ and hence, by Lemma~\ref{dcl:to:ac}, is algebraic over $a$. Since there is an $a \in \ga$ such that $a_i \notin \acl(a_{\tilde i}, b)$ for all $i \in [n]$, by Lemma~\ref{exchange},  $b$ must be algebraic.
\end{proof}

\begin{lem}[{\cite[Lemma~4.15]{Yin:QE:ACVF:min}}]\label{effectiveness}
If $\mathbb{S}$ is $(\VF, \RV,\Gamma)$-generated then every algebraic (respectively, definable) closed disc  contains an algebraic (respectively, definable) point.
\end{lem}

\begin{lem}[{\cite[Lemma~4.17]{Yin:QE:ACVF:min}}]\label{algebraic:balls:definable:centers}
If $\mathbb{S}$ is $(\VF, \Gamma)$-generated then every algebraic (respectively, definable) disc  contains an algebraic (respectively, definable) point.
\end{lem}

A pillar of the structure of definable sets in $\bb U$ is \emph{\cmin-minimality}, meaning that every definable subset of $\VF$ is a boolean combination of  (definable) valuative discs.

\begin{cor}\label{open:disc:def:point}
Suppose that $\mathbb{S}$ is $(\VF, \RV,\Gamma)$-generated.
Let $\ga \sub \VF$ be a definable disc and $A$ a definable subset of $\VF$. If $\ga \cap A$ is a nonempty proper subset of $\ga$ then $\ga$ contains a definable point.
\end{cor}
\begin{proof}
It is not hard to see that, by \cmin-minimality, if $\ga \cap A$ is a nonempty proper subset of $\ga$ then $\ga$ contains a definable closed disc and hence the claim is immediate by Lemma~\ref{effectiveness}.
\end{proof}

\begin{defn}
Let $D$ be a set of parameters. We say that a (not necessarily definable) nonempty set $A$ \memph{generates a (complete) $D$-type} if, for every $D$-definable set $B$, either $A \sub B$ or $A \cap B = \0$. In that case, $A$ is \memph{$D$-type-definable} if no set properly contains $A$ and also generates a $D$-type. If $A$ is $D$-definable and generates a $D$-type, or equivalently, if $A$ is both $D$-definable and $D$-type-definable then we say that $A$ is \memph{$D$-atomic} or \memph{atomic over $D$}.
\end{defn}

We simply say ``atomic'' when $D =\0$.

In the literature, a type could be a partial type and hence a type-definable set may have nontrivial intersection with a definable set. In this paper, since partial types do not play a role, we shall not carry the superfluous qualifier ``complete'' in our terminology.

\begin{lem}[{\cite[Lemma~3.3]{Yin:int:acvf}}]\label{atomic:base:change}
Let $A$ be an atomic set. Then $A$ is $\gamma$-atomic for all $\gamma \in \Gamma$.
\end{lem}

\begin{lem}\label{at:ball:t}
Let $\ga \sub \VF^n$ be an atomic open polydisc. Then $\ga$ is $t$-atomic for all $t \in \RV$.
\end{lem}
\begin{proof}
For the case $n=1$, if $\ga$ were not $t$-atomic then, by Corollary~\ref{open:disc:def:point} and Lemma~\ref{dcl:to:ac}, it would contain a definable point, contradicting atomicity. For the case $n>1$, by induction on $n$, $\ga_1$ is $t$-atomic and for every $a \in \ga_1$, $\ga_{>1}$ is $(a, t)$-atomic. So $\ga$ cannot have nonempty $t$-definable subset.
\end{proof}

\begin{lem}\label{atom:self}
Let $A \sub \VF^n$ be a set that generates a type and $\gb \sub A$ an open (or closed) polydisc. Then $\gb$ is $\code \gb$-atomic.
\end{lem}
\begin{proof}
We do induction on $n$. The base case $n=1$ is just \cite[Lemma~3.4]{Yin:int:acvf}. For the case $n>1$, write $\gb = \prod_i \gb_i$, $\gb' = \prod_{i>1} \gb_i$, and $\gamma=\rad(\gb_1) $. For any $a \in \gb_1$, $A_a$ generates a type. So, by Lemma~\ref{atomic:base:change} and the inductive hypothesis, $\gb'$ is $(a, \gamma, \code{\gb'})$-atomic and hence $\code \gb$-atomic. Similarly $\gb_1$ is $\code \gb$-atomic. It follows that $\gb$ must be $\code \gb$-atomic.
\end{proof}


\begin{lem}\label{at:rv:split}
Let $\ga \sub \VF^n$ be an atomic open polydisc and $f : \ga \fun \mdl P(\VF)$ a definable finitary function such that $\rv \rest f(a)$ is injective for every $a \in \ga$.  Then  $\rv \circ f$ is constant.
\end{lem}
\begin{proof}
This is clear since, by Lemma~\ref{at:ball:t}, for every $t \in \rv(\bigcup f(\ga))$, $\ga$ is $t$-atomic and hence $f$ induces a $t$-definable function $\ga \fun t^\sharp$. The lemma follows.
\end{proof}

\begin{lem}[{\cite[Lemma~3.8]{Yin:int:acvf}}]\label{hae:component:11}
Let $A \sub \VF$ be an atomic open disc or an atomic closed disc or an atomic thin annulus. Let $f : A \fun \VF$ be a definable function. Then $f(A)$ is also atomic of one of these three forms.
\end{lem}

\begin{lem}[{\cite[Lemma~3.13]{Yin:int:acvf}}]\label{atomic:open:ball:open}
Let $A \sub \VF$ be an atomic open disc and $f : A \fun \VF$ a definable function. Then either $f$ is constant or $f(A)$ is also an atomic open disc.
\end{lem}

\begin{lem}\label{atomic:open:multidisc}
Let $\ga$, $f$ be as in Lemma~\ref{at:rv:split} and suppose that $f$ is not constant. Then there are $t_1, \ldots, t_k \in \RV$ and for each $i$, a $t_i$-definable function $f_i : \ga \fun t^\sharp_i$ such that $f = \bigcup_i f_i$ and each $f_i(\ga)$ is either a point or a $t_i$-atomic open disc.
\end{lem}
\begin{proof}
By Lemmas~\ref{at:ball:t} and \ref{at:rv:split}, this is immediately reduced to the case that $f$ is a definable nonconstant function $\ga \fun t^\sharp$. We do induction on $n$. The base case $n=1$ is just Lemma~\ref{atomic:open:ball:open}. For the case $n > 1$,  by Lemma~\ref{hae:component:11}, $f(\ga)$ is an open disc or a closed disc or a thin annulus. Suppose for contradiction that $f(\ga)$ is a closed disc or a  thin annulus. By the inductive hypothesis, for every $a \in \pr_{1}(\ga)$ there is a maximal open subdisc $\gb_a \sub f(\ga)$ that contains $f(\ga_a)$, similarly for every $a \in \pr_{>1}(\ga)$. It follows that $f(\ga)$ is actually contained in a maximal open subdisc of $f(\ga)$, which is absurd.
\end{proof}

\begin{cor}\label{part:rv:cons}
Let $A \sub \VF^n$ be a definable set and $f : A \fun \mdl P(\VF)$ a definable finitary function as in Lemma~\ref{at:rv:split}. Then there are a definable finite partition $(A_i)_i$ of $A$ and, for each $i$ and all open polydiscs $\ga \sub A_i$, finitely many $t_{\ga j}$-definable functions $f_{\ga j} : \ga \fun t^\sharp_{\ga j}$, where $t_{\ga j} \in \RV$, such that  $f \rest \ga = \bigcup_j f_{\ga j}$  and each $f_{\ga j}(\ga)$ is either a point or an open disc.
\end{cor}
\begin{proof}
For $a \in A$, let $D_a \sub A$ be the type-definable subset containing $a$. By Lemma~\ref{atom:self}, every open polydisc $\ga \sub D_a$ is $\code \ga$-atomic and hence, by Lemma~\ref{atomic:open:ball:open}, the assertion holds for $\ga$. Then, by compactness, the assertion must hold in a definable subset $A_a \sub A$ that contains $a$; by compactness again, it holds in finitely many definable subsets $A_1, \ldots, A_m$ of $A$ with $\bigcup_i A_i = A$. Then the partition of $A$ generated by $A_1, \ldots, A_m$ is as desired.
\end{proof}

\begin{defn}\label{defn:dtdp}
Let $f : A \fun B$ be a bijection between two sets $A$ and $B$, each with exactly one $\VF$-coordinate. We say that $f$ is \memph{concentric} if, for every $\VF$-fiber $f_{t}$ of $f$ and all open polydiscs $\ga \sub \dom(f_{t})$, $f_{t}(\ga)$ is also an open polydisc; if both $f$ and $f^{-1}$ are concentric then $f$ has the \memph{disc-to-disc property}.
\end{defn}

\begin{lem}[{\cite[Proposition~3.19]{Yin:int:acvf}}]\label{open:pro}
Let $f : A \fun B$ be a definable bijection between two sets $A$ and $B$, each with exactly one $\VF$-coordinate. Then there is a definable finite partition $(A_i)_i$ of $A$ such that every restriction $f \rest A_i$ has the disc-to-disc property.
\end{lem}

\subsection{Continuity and fiberwise properties}\label{sec:cont:fib:prop}

\begin{defn}
The \memph{$\RV$-dimension} of a definable subset $U \sub \RV^m_\infty$, denoted by $\dim_{\RV}(U)$, is the smallest number $k$ such that there is a definable finite-to-one function $f: U \fun \RV^k$ ($\RV^0$ is taken to be the singleton $\{\infty\}$).
\end{defn}

\begin{defn}
Let $A$ be a subset of $\VF^n$. The \memph{$\RV$-boundary} of $A$, denoted by $\partial_{\RV}A$, is the definable subset of $\rv(A)$ such that $t \in \partial_{\RV} A$ if and only if $t^\sharp \cap A$ is a proper nonempty subset of $t^\sharp$. The definable set $\rv(A) \mi \partial_{\RV}A$, denoted by $\ito_{\RV}(A)$, is called the \memph{$\RV$-interior} of $A$.
\end{defn}

Obviously, $A \sub \VF^n$ is an $\RV$-pullback if and only if $\partial_{\RV} A$ is empty.

\begin{lem}\label{RV:bou:dim}
Let $A$ be a definable subset of $\VF^n$. Then $\dim_{\RV}(\partial_{\RV} A) < n$.
\end{lem}
\begin{proof}
We proceed by induction on $n$. The base case $n=1$  follows immediately from \cmin-minimality since $t^\sharp$ contains a $t$-definable proper subdisc for every $t \in \partial_{\RV} A$.

For the inductive step, since $\partial_{\RV} A_a$ is finite for every $i \in [n]$ and every $a \in \pr_{\tilde i}(A)$, by Corollary~\ref{open:disc:def:point} and compactness, there are a definable finite partition $(A_{ij})_j$ of $\pr_{\tilde i}(A)$ and, for each $A_{ij}$, a definable finitary function $f_{ij} : A_{ij} \fun \mdl P(\VF)$ such that for all $a \in A_{ij}$, $\rv(f_{ij}(a)) = \partial_{\RV} A_a$ if $A_a$ is not an $\RV$-pullback or $f_{ij}(a) = \{0\}$ otherwise. By Corollary~\ref{part:rv:cons}, we may assume that if $t^\sharp \sub A_{ij}$ then the restriction $\rv \rest f_{ij}(t^\sharp)$ is constant. Hence each $f_{ij}$ induces a definable finitary function $C_{ij} : \ito_{\RV}(A_{ij}) \fun \mdl P(\RV_\infty)$.
Let
\[
 C = \bigcup_{i, j} C_{ij} \dand B = \bigcup_{i,j} \bigcup_{t \in \partial_{\RV} A_{ij}} \rv(A)_t.
\]
Obviously $\dim_{\RV}(C) < n$. By the inductive hypothesis, for all $A_{ij}$ we have $\dim_{\RV}(\partial_{\RV} A_{ij}) < n-1$. Thus $\dim_{\RV}(B) < n$. Since $\partial_{\RV} A \sub B \cup C$, the claim follows.
\end{proof}

Let $f : \VF^n \fun \VF^m$ be a definable function.

\begin{lem}[{\cite[Lemma~8.7]{Yin:special:trans}}]\label{fun:almost:cont}
There is a definable closed set $A \sub \VF^n$ with $\dim_{\VF}(A) < n$ such that $f \rest (\VF^n \mi A)$ is continuous.
\end{lem}

Per the usual terminology, this lemma says that $f$ is continuous \memph{almost everywhere}.

\begin{defn}\label{defn:diff}
For any $a \in \VF^n$, we say that $f$ is \memph{differentiable at $a$} if there is a linear map $\lambda : \VF^n \fun \VF^m$ (of $\VF$-vector spaces) such that, for any $\epsilon \in \Gamma$, if $b \in \VF^n$ and $\vv(b)$ is sufficiently large then
\[
\vv(f(a + b) - f(a) - \lambda(b)) - \vv(b) > \epsilon.
\]
It is straightforward to check that if such a linear function $\lambda$ (represented by a matrix with entries in $\VF$) exists then it is unique and hence is called the \memph{derivative of $f$ at $a$}.

Write $f = (f_1, \ldots, f_m)$. For $a = (a_i, a_{\tilde i}) \in \VF^n$, if the derivative of the function $f_j \rest (\VF \times a_{\tilde i})$ at $a_i$ exists then we call it the \memph{$ij$th partial derivative of $f$ at $a$}.
\end{defn}

We differentiate functions between arbitrary definable sets by ``forgetting'' the $\RV$-coordinates. More precisely, let
\[
f : \VF^n \times \RV^m \fun \VF^{n'} \times \RV^{m'}
\]
be a definable function. By dimension theory, for every $t \in \RV^m$ there is an $s \in \RV^{m'}$ such that $\dim_{\VF}(\dom(f_{(t, s)})) = n$ and hence $\dom(f_{(t, s)})$ has an open subset. For such an $s \in \RV^{m'}$ and each $a$ contained in an open subset of $\dom(f_{(t, s)})$, we define the partial derivatives of $f$ at $(a, t)$ to be those of $f_{(t, s)}$ at $a$. If $n = n'$ and all the partial derivatives exist at a point $(a, t)$ then the \memph{Jacobian of $f$ at $(a,t)$} is defined in the usual way (that is, the determinant of the Jacobian matrix) and is denoted by $\jcb_{\VF} f(a,t)$.

\begin{lem}\label{diff:almost:every}
The Jacobian of $f$  exists and is a continuous function almost everywhere.
\end{lem}
\begin{proof}
This is immediate by \cite[Corollary~9.9]{Yin:special:trans}.
\end{proof}

\begin{ter}\label{bounded}
We say that a set $I \sub \Gamma_{\infty}^n$ is \memph{$\gamma$-bounded}, where $\gamma \in \Gamma$, if it is contained in the box $[\gamma, \infty]^n$, and is \memph{doubly $\gamma$-bounded} if it is contained in the box $[-\gamma, \gamma]^n$. More generally, let $A$ be a subset of $\VF^n \times \RV_\infty^m \times \Gamma_\infty^l$ and
\[
A_{\Gamma} = \set{(\vv(a), \vrv(t), \gamma) \given (a, t, \gamma) \in A} \sub \Gamma_{\infty}^{n+m+l};
\]
then we say that $A$ is \memph{$\gamma$-bounded} if $A_{\Gamma}$ is, and so on.
\end{ter}

\begin{defn}[Contractions]\label{defn:corr:cont}
A function $f : A \fun B$ is \memph{$\rv$-contractible} if there is a (necessarily unique) function $f_{\downarrow} : \rv(A) \fun \rv(B)$, called the \memph{$\rv$-contraction} of $f$, such that
\[
(\rv \rest B) \circ f = f_{\downarrow} \circ (\rv \rest A).
\]
Similarly, it is \memph{$\res$-contractible} (respectively, \memph{$\vv$-contractible}) if the same holds in terms of $\res$ (respectively, $\vv$ or $\vrv$, depending on the coordinates) instead of $\rv$.
\end{defn}

\begin{lem}\label{db:to:db}
Let $U \sub \RV^k$ be a doubly bounded set and $f: U \fun \RV$ a definable function. Then $f(U)$ is also doubly bounded. The same holds if the codomain of $f$ is $\Gamma$.
\end{lem}
\begin{proof}
By induction on $k$, both claims are immediately reduced to showing that if $k=1$  and $g : U \fun \Gamma$ is a definable function then $g(U)$ is doubly bounded. Then, by \cmin-minimality, we may assume that $g$  $\vrv$-contracts to a function $g_{\downarrow} : \vrv(U) \fun \Gamma$. Since definable functions in $\Gamma$ are piecewise $\Q$-linear, the range of $g_{\downarrow}$ must be doubly bounded.
\end{proof}

\begin{defn}
Let $A \sub \VF^n$ and $f : A \fun \mdl P(\VF^m)$ be a definable function whose range is bounded. For $a \in \VF^n$ and $L \sub \VF^m$, we say that $L$ is a \memph{limit set of $f$ at $a$}, written as $\lim_{A \rightarrow a} f \sub L$, if for every $\epsilon \in \Gamma$ there is a $\delta \in \Gamma$ such that if $c \in \go( a, \delta) \cap (A \mi a)$ then $f(c) \sub \bigcup_{b \in L'} \go(b, \epsilon)$ for some $L' \sub L$.

A limit set $L$ of $f$ at $a$ is \memph{minimal} if no proper subset of $L$ is a limit set of $f$ at $a$.
\end{defn}

Observe that if $\lim_{A \rightarrow a} f \sub L$ and $b \in L$ is not isolated in $L$ then actually $\lim_{ A \rightarrow  a} f \sub L \mi b$. So in a minimal limit set every element is isolated. Moreover, if a minimal limit set $L$ exists then its topological closure is unique:

\begin{lem}[{\cite[Lemma~9.2]{Yin:special:trans}}]
Let $L_1, L_2 \sub \VF^m$ be two minimal limit sets of $f$ at $a$. Then the closures of $L_1$, $L_2$ coincide.
\end{lem}

This lemma justifies writing $\lim_{ A \rightarrow a} f = L$ when $L$ is a closed (hence the unique) minimal limit set of $f$ at $a$, in particular, when $L$ is finite and minimal.

Let $A \sub \VF^n$ be a definable set with $\dim_{\VF}(A) = n$ and $f : A \fun \mdl P(\VF^m)$ a definable finitary function. We say that $f$ is \memph{continuous} if $\lim_{ A \rightarrow a} f = f(a)$ for all $a \in A$.

\begin{lem}\label{fini:conti}
The function $f$ is continuous almost everywhere.
\end{lem}
\begin{proof}
This can be reduced to the familiar case that $f$ is indeed a function into $\VF^m$ as follows. By Lemma~\ref{finite:VF:project:RV} and compactness, there are a definable set $U$ in $\RV$ and a definable function $g: A \times \VF^m \fun A \times U$ such that, for all $a \in A$, $\pr_A(g(a, b)) = a$  and $g \rest (a \times f(a))$ is injective. For each $t \in U$, let
\[
A_t = \set{a \in A \given (a,t) \in g(a \times f(a))}
\]
and $f_t \sub f$ be the obvious $t$-definable function $A_t \fun \VF^m$. By dimension theory, more precisely, by Lemma~\ref{full:dim:open:poly}, $\dim_{\VF}(A_t) = n$ for some $t \in U$. For such a $t \in U$, by Lemma~\ref{fun:almost:cont}, $f_t$ is continuous everywhere.  The claim follows.
\end{proof}

\begin{lem}[{\cite[Lemma~9.5]{Yin:special:trans}}]\label{lim:exists}
Suppose that $f : \MM \mi 0 \fun \mdl P (A)$ is a definable finitary function, where $A$ is a closed and bounded set in $\VF$. Then there is a definable finite set $L \sub A$ such that $\lim_{\MM \mi 0 \rightarrow 0} f = L$.
\end{lem}

We remark that if a ``point at infinity'' is adjoined in each coordinate (one-point compactification of $\VF$)  then there is no need to require that $A$ be bounded.

\begin{lem}[Curve selection]\label{curve:select}
Let $A$ be a set in $\VF$ and $f : A \fun \mdl P(\VF)$ a definable finitary function. Let $b$ be an algebraic point in the closure of $f(A)$ (possibly the ``point at infinity''). Then there is a definable continuous finitary function $g : \MM \mi 0 \fun \mdl P(A)$, understood as a ``curve,'' such that $f \circ g$ is continuous and $b \in \lim_{\MM \mi 0 \rightarrow 0} f \circ g$.

\end{lem}
\begin{proof}
Since $\acl(c)$ is a model of $\ACVF(\mathbb{S})$ for each $c \in \MM \mi 0$, we have that
\[
\set{a \in A \given \vv(f(a) - b) > \vv(c)} \cap \acl(c)
\]
is nonempty. Therefore, by compactness, there is a definable finitary function $g : \MM \mi 0 \fun \mdl P(A)$ such that $f(g(c)) \sub \go(b, \vv(c))$ for every $c \in \MM \mi 0$. If the $i$th coordinate of $b$ is the ``point at infinity'' then we use the ``open disc'' $\go(b_i, -\vv(c))$ in that coordinate instead. By Lemma~\ref{lim:exists}, $\lim_{\MM \mi 0 \rightarrow 0} f \circ g$ exists in the  closure of $f(A)$,  and it clearly contains $b$. By Lemma~\ref{fini:conti}, $g \rest (\gb \mi 0)$ and $(f \circ g) \rest (\gb \mi 0)$ are both continuous for some definable open disc $\gb$ around $0$. Dilating with a definable factor, the claim follows.
\end{proof}

This is an analogue of the highly useful \omin-minimal curve selection.


\begin{lem}\label{cb:to:cb}
Let $A$ be a closed and bounded set in $\VF$, and $f : A \fun \VF^m$ a definable continuous function. Then $f(A)$ is also closed and bounded.
\end{lem}
\begin{proof}
If $f(A)$ is not bounded then, applying curve selection to  any ``algebraic point at infinity'' $b$ in the boundary of $f(A)$,  we see that there is a curve $C \sub A$ such that one branch of the curve $f(C)$ approaches $b$. Since $A$ is closed and bounded, this would mean that $b$ belongs to $f(A)$, which is impossible. The same argument works for closedness.
\end{proof}

Again, if we work in the  one-point compactification of $\VF$ then there is no need to talk about boundedness.

\begin{cor}\label{conti:homeo}
If the function $f$ in Lemma~\ref{cb:to:cb} is injective then  it is a homeomorphism from $A$ onto $f(A)$.
\end{cor}

\begin{lem}[The closed graph criterion for continuity]\label{closed:gra}
Let $A$, $B$ be definable sets in $\VF$. Suppose that $B$ is closed (if $B$ is not bounded then the  closure is taken in the coordinatewise one-point compactification). A definable function $f : A \fun B$ is continuous if and only if its graph  is closed in $A \times B$.
\end{lem}
\begin{proof}
The ``only if'' direction follows from the usual argument. For the ``if'' direction, suppose for contradiction that there are an $a \in A$ and an open polydisc $\gb$ around $f(a)$ such that for all open polydisc $\ga$ around $a$, $f(\ga \cap A)$ is not contained in $\gb$. By Lemma~\ref{curve:select}, we can choose a curve inside the graph of $f$ such that one of its limit points is of the form $(a, b') \in A \times B$. Since $b' \neq f(a)$, we have reached a contradiction.
\end{proof}

\begin{lem}\label{lim:weak}
Let $f : A  \fun B$ be a definable bijection between two open subsets of $\VF$. Suppose that $B$ is bounded. Then $f$ is weakly concentric, that is, for every $a \in A$ there is a $b \in \VF$ such that, for every sufficiently large $\gamma \in \Gamma$, $f(\go(a, \gamma) \mi a) = \go(b, \delta) \mi b$ for some $\delta \in \Gamma$.
\end{lem}

Note that, compared to Lemma~\ref{open:pro}, the assumption of this lemma is stronger and its conclusion weaker, but the point is that we do not need a partition to achieve it.

\begin{proof}
By Lemma~\ref{lim:exists}, for any $a \in A$, $\lim_{ A \rightarrow a} f$ exists in the closure  of $B$, which we denote by $b_a$. By \cmin-minimality and \omin-minimality in the $\Gamma$-sort, $\gc(b_a, \delta) \mi \go(b_a, \delta) \sub B$ for all sufficiently large $\delta \in \Gamma$. Also,  for all sufficiently large $\gamma \in \Gamma$ there is a $\delta \in \Gamma$ such that
\[
f(\gc(a, \gamma) \mi \go(a, \gamma)) = \gc(b_a, \delta) \mi \go(b_a, \delta)
\]
The lemma follows.
\end{proof}

\begin{defn}
Let $p : A \fun \Gamma$ be a definable function. We say that $p$ is an \memph{$\go$-partition} of $A$ if, for every $a \in A$, the function $p$ is constant on $\go(a, p(a)) \cap A$.
\end{defn}

This makes sense even when $A$ has no $\VF$-coordinates, in which case any definable function $A \fun \Gamma$ is an $\go$-partition of $A$.

\begin{lem}\label{vol:par:bounded}
Let $p$ be an $\go$-partition of $A$. Suppose that $A_{\RV}$ is doubly bounded and $A_t$ is closed and bounded for every $t \in A_{\RV}$. Then $p(A)$ is doubly bounded.
\end{lem}
\begin{proof}
We first handle the case that $A$ has no $\RV$-coordinates. Since $A$ is closed and bounded, $p(A)$ must be bounded from below. Suppose for contradiction that $p(A)$ is not bounded from above. For each $\gamma \in \Gamma$, let $A_{\gamma} = \set{a \in A \given p(a) > \gamma }$. Choose a definable continuous finitary function, that is, a curve
\[
f : \MM \mi 0 \fun \mdl P \bigg( \bigcup_{c \in \MM \mi 0} c \times A_{\vv(c)} \bigg)
\]
such that $f(c)_1 = c$. Write $A_c = f(c)_{>1}$. So $p(A_c) > \vv(c)$ for all $c \in \MM \mi 0$. Since  $A$ is closed, the limit points of $f$ are all of the form $(0,e)$ with $e \in A$. So there is a $c \in \MM \mi 0$ with $\vv(c) > p( e)$ such that $A_c \cap \go(e, p( e)) \neq \0$. Since $p(a) > \vv(c)$ for all $a \in A_c$, this contradicts the assumption that $p$ is an $\go$-partition of $A$. So $p(A)$ is bounded from above.

It is easy to see that the general case follows from this special case and Lemma~\ref{db:to:db}.
\end{proof}

This lemma, in its various incarnations, is crucial for the good behavior of motivic Fourier transform (see \cite[\S~11]{hrushovski:kazhdan:integration:vf} and \cite{yin:hk:part:3}). For essentially the same reason, the  main construction of this paper depends heavily on  it (often in combination with Lemma~\ref{db:to:db}).

\begin{lem}\label{lem:open:fiberwise}
Let $A\subseteq \VF^n$ be a definable set. Suppose that, for each $a \in A_1$, the fiber $A_a \sub \VF^{n-1}$ is open (respectively, closed). Then there is a finite definable set $A' \subseteq A_1$ such that $A\cap ((A_1 \mi A')\times \VF^{n-1})$ is open (respectively, closed) in $(A_1 \mi A')\times \VF^{n-1}$.
\end{lem}
\begin{proof}
Since either one of the two cases implies the other, we shall just concentrate on the open case. Consider the  definable set
\[
S=A\cap \Cl((A_1\times \VF^{n-1}) \mi A),
\]
where $\Cl$ is the closure operator in $A_1\times \VF^{n-1}$; in other words, $S$ is the boundary of $A$ minus the frontier of $A$. It is enough to show that $S_1$ is finite. To that end, suppose for contradiction that $S_1$  contains a disc.

Since $S$ cannot contain an open subset of $\VF^n$, we have $\dim_{\VF}(S)<n$; for ease of notation, we may as well take $\dim_{\VF}(S) = n-1$. Then, up to a definable finite partition of $S$ and possibly after a reordering of the last $n-1$ coordinates, we may assume that $S$ is the graph of a definable continuous finitary function $\gp \fun \mdl P(\VF)$, where $\gp$ is an open polydisc and continuity is by Lemma~\ref{fini:conti}. So, for all $\epsilon \in \Gamma$, if $\rad(\gp)$ is sufficiently large, then $S_n$ is clustered inside finitely many discs of radius $\epsilon$. Now, for each $a \in S_1$, consider the function $\alpha(a, -): S \fun \Gamma$ given by
\[
\alpha(a,x) = \min \set{\beta\in \Gamma \given  \go(x, \beta)\subseteq A_a},
\]
which exists because $A_a$ is open. By dimension theory, $\dim_{\VF}(\alpha^{-1}(\epsilon)) = n-1$ for some $\epsilon$. Without loss of generality, we may assume that $\alpha(a,x) = \epsilon$ for all $(a,x)\in S$ and $\rad(\gp)$ is sufficiently large in the sense described above. It follows that $S$ contains an open polydisc of radius $\rad(\gp)$, which contradicts the choice of $S$.
\end{proof}

\begin{prop}\label{prop:continuity:fiberwise}
Let $A$ be as in Lemma~\ref{lem:open:fiberwise} and $f : A \fun \VF$  a definable function such that
for each $a \in A_1$, the induced function $f_a : A_a \fun \VF$ is continuous.
Then there is a finite definable set $A' \sub A_1$ such that $f$ is continuous on
$\bigcup_{a\in A_1 \mi A'} a \times A_a$.
\end{prop}

Here if $f(A)$ is unbounded then the codomain of $f$ is taken to be the one-point compactification of $\VF$.

\begin{proof}
By the  closed graph criterion for continuity (Lemma~\ref{closed:gra}), this is an immediate consequence  of    Lemma \ref{lem:open:fiberwise}.
\end{proof}

This is analogous to \cite[Corollary~6.2.4]{dries:1998} in the \omin-minimal environment and \cite[Lemma~2.38]{Yin:tcon:1.5} in the \T-convex environment.

\section{Hrushovski-Kazhdan style integral}\label{sec:HK:main}

\begin{defn}
\label{def:Kgroup}
The Grothendieck semigroup $\gsk \mdl C$ of a category $\mdl C$ endowed with a binary operation ``$\mi$'' and a binary relation ``$\subseteq$'' (subobject) is the free semigroup generated by the isomorphism classes of $\mdl C$, subject to the usual scissor relation
$[A \mi B] + [B] = [A]$, when $B\subseteq A$,
where $[A]$, $[B]$ denote the isomorphism classes of the objects $A$, $B$, and ``$\mi$'' and ``$\subseteq$'' are compatible with isomorphisms between objects, usually they are just set subtraction and set inclusion. Sometimes $\mdl C$ is also equipped with a binary operation (for example, cartesian product) that induces multiplication in $\gsk \mdl C$, in which case $\gsk \mdl C$ becomes a (commutative) semiring. The formal groupification of  $\gsk \mdl C$, which is then a ring, is denoted by $\ggk \mdl C$. 
\end{defn}

\begin{defn}
An \memph{$\RV$-fiber} of a definable set $A$ is a set of the form $A_a$, where $a \in A_{\VF}$. The \memph{$\RV$-fiber dimension} of $A$ is the maximum of the $\RV$-dimensions of its $\RV$-fibers and is denoted by $\dim^{\fib}_{\RV}(A)$.
\end{defn}

\begin{lem}[{\cite[Lemma~4.13]{Yin:special:trans}}]\label{RV:fiber:dim:same}
Let $f : A \fun A'$ be a definable bijection. Then $\dim^{\fib}_{\RV}(A) = \dim^{\fib}_{\RV} (A')$.
\end{lem}

\begin{defn}[$\VF$-categories]\label{defn:VF:cat}
The objects of the category $\VF[k]$ are the definable sets of $\VF$-dimension less than or equal to $k$ and $\RV$-fiber dimension $0$ (that is, all the $\RV$-fibers are finite). Any definable bijection between two such objects is a morphism of $\VF[k]$. Set $\VF_* = \bigcup_k \VF[k]$.
\end{defn}

As soon as one considers adding volume forms to definable sets in $\VF$, the question of ambient dimension arises and, consequently, one has to take ``essential bijections'' as morphisms.

 \begin{defn}
\label{defn:VF:mu:G}
An object of the category $\mgVF[k]$ is a definable pair $(A, \omega)$, where $A \in \VF[k]$, $A_{\VF} \sub \VF^k$, and $\omega : A \fun \Gamma$ is a function, which is understood as a \memph{definable $\Gamma$-volume form} on $A$. A \memph{morphism} between two such objects $(A, \omega)$, $(B, \sigma)$ is a definable \memph{essential bijection} $F : A \fun B$, that is, a bijection that is defined outside definable subsets of $A$, $B$ of $\VF$-dimension less than $k$, such that, for almost all $x \in A$,
\begin{equation}\label{cond:mg}
\omega(x) = \sigma(F(x)) + \vv(\jcb_{\VF} F(x)).
\end{equation}
We also say that such an $F$ is \memph{$\Gamma$-measure-preserving}.
\end{defn}

\begin{nota}
In \cite{hrushovski:kazhdan:integration:vf}, the category $\mgVF[k]$ is denoted by ${\mu_{\Gamma}}{\VF[k]}$ to indicate that the volume forms take values in $\Gamma$ as opposed to $\RV$. Here the subscript ``$\Gamma$'' is dropped since  we will not consider $\RV$-volume forms at all.
\end{nota}

In the definition above and other similar ones below, for the cases $k =0$, the reader should interpret things such as $\VF^0$ and how they interact with other things in a natural way. For instance, $\VF^0$ may be treated as the empty tuple, the only definable set of $\VF$-dimension less than $0$ is the empty set, and $\jcb_{\VF}$ is always $1$ on sets that have no $\VF$-coordinates. So $(A, \omega) \in \mgVF[0]$ if and only if $A$ is a finite definable set in $\RV_\infty$.

Set $\mgVF[{\leq} k] = \coprod_{i \leq k} \mgVF[i]$ and $\mgVF[*] = \coprod_{k} \mVF[k]$; similarly for the other categories below (with or without volume forms).

\begin{rem}\label{mor:equi}
Let $F : (A, \omega) \fun (B, \sigma)$ be a $\mgVF[k]$-morphism. Our intention is that such an $F$ should identify the two objects, that is, an isomorphism. However, if $F$ is not defined everywhere in $A$ then obviously it does not admit an inverse. We remedy this by introducing the following congruence relation $\sim$ on $\mgVF[k]$. Let $G : (A, \omega) \fun (B, \sigma)$ be another $\mgVF[k]$-morphism. Then $F \sim G$ if $F(a) = G(a)$ for all $a \in A$  outside a definable subset  of $\VF$-dimension $< k$. The morphisms of the quotient category $\mgVF[k] /{\sim}$ have the form $[F]$, where $F$ is a $\mgVF[k]$-morphism. Clearly every $(\mgVF[k] / {\sim})$-morphism is an isomorphism and hence $\mgVF[k] / {\sim}$ is a groupoid. In fact, all the categories of definable sets we shall work with should be and are groupoids.

It is certainly more convenient to work with representatives than equivalence classes. In the discussion below, this quotient category $\mgVF[k] /{\sim}$ will almost never be needed except when it comes to forming the Grothendieck semigroup or, by abuse of terminology, when we speak of two objects of $\mgVF[k]$ being isomorphic.
\end{rem}

\begin{defn}[$\RV$-categories]\label{defn:c:RV:cat}
The objects of the category $\RV[k]$ are the pairs $(U, f)$ with $U$ a set in $\RVV$ and $f : U \fun \RV^k$ a definable finite-to-one function. Given two such objects $(U, f)$, $(V, g)$, any definable bijection $F : U \fun V$ is a \memph{morphism} of $\RV[k]$.
\end{defn}

By Lemma~\ref{finite:VF:project:RV}, the two categories $\VF[0]$, $\RV[0]$ are actually equivalent, similarly for other such categories.

\begin{nota}\label{0coor}
We emphasize that if $(U, f)$ is an object of $\RV[k]$ then $f(U)$ is a subset of $\RV^k$ instead of $\RV^k_\infty$, while $\infty$ can occur in any coordinate of $U$. An object of  $\RV[*]$ of the form $(U, \id)$ is often just written as $U$.

More generally, if $f : U \fun \RV_\infty^k$ is a definable finite-to-one function then $(U, f)$ denote the obvious object of $\RV[{\leq} k]$. For example, the inclusion $\{\infty\} \fun \RV_{\infty}$ gives rise to an object of $\RV[0]$, the inclusion $\{(1, \infty)\} \fun \RV_{\infty}^2$ gives rise to an object of $\RV[1]$, and so on. Often $f$ will be a coordinate projection (every object in $\RV[*]$ is isomorphic to an object of this form). In that case, $(U, \pr_{\leq k})$ is simply denoted by $U_{\leq k}$ and its class in $\gsk \RV[k]$ by $[U]_{\leq k}$, etc.
\end{nota}

\begin{defn}[$\RES$-categories]\label{defn:RES:cat}
The category $\RES[k]$ is the full subcategory of $\RV[k]$ such that $(U, f) \in \RES[k]$ if and only if $\vrv(U)$ is finite.
\end{defn}

\begin{nota}\label{mor:ftf}
Every $\RV[k]$-morphism $F: (U, f) \fun (V, g)$  induces a definable finite-to-finite correspondence $F^\dag \sub f(U) \times g(V)$, which may also be thought of as a finitary function $f(U) \fun \mdl P(g(V))$. Recall the operator $\jcb_{\Gamma}$ from Notation~\ref{jcbG:op}. For $ u \in U$, we abbreviate $\jcb_{\Gamma}  F^\dag(f(u), g \circ F(u))$ as $\jcb_{\Gamma}  F^\dag(u)$
\end{nota}

\begin{defn}[$\RV$- and $\RES$-categories with $\Gamma$-volume forms]\label{defn:RV:mu}
An object of the category $\mgRV[k]$ is a definable triple $(U, f, \omega)$, where $(U, f)$ is an object of $\RV[k]$ and $\omega : U \fun \Gamma$ is a function, which is understood as a \memph{definable $\Gamma$-volume form} on $(U, f)$. A \memph{morphism} between two such objects $(U,f, \omega)$, $(V, g, \sigma)$ is an $\RV[k]$-morphism $F: (U, f) \fun (V, g)$ such that for every $ u \in U$,
\begin{equation}\label{cond:mg2}
\omega(u) = (\sigma \circ F)(u) + \jcb_{\Gamma}  F^\dag(u).
\end{equation}

The category $\mgRES[k]$ is the obvious full subcategory of $\mgRV[k]$.
\end{defn}

\begin{defn}[$\Gamma$-categories]\label{def:Ga:cat}
The objects of the category $\Gamma[k]$ are the finite disjoint unions of definable subsets of $\Gamma^k$. A definable bijection between  two such objects is a  morphism of $\Gamma[k]$ if and only if it is the $\vrv$-contraction of a definable bijection.

The category $\Gamma^{\fin}[k]$ is the full subcategory of $\Gamma[k]$ such that $I \in \Gamma^{\fin}[k]$ if and only if $I$ is finite.
\end{defn}

\begin{rem}
If $\gamma \in \Gamma$ is definable then it is in the divisible hull of $\Gamma(\mathbb{S})$. This does not mean that the definable set $\gamma^{\sharp} \sub \RV$ contains a definable point unless $\gamma \in \Gamma(\mathbb{S})$.
\end{rem}

\begin{rem}\label{GLZ:char}
By \cite[Remark~2.28]{Yin:int:expan:acvf}, if a definable function between two sets in $\Gamma$ is a $\vrv$-contraction then it is $\Z$-linear (with constant terms of the form $\vrv(t)$, where $t \in \RV$ is definable, that is, $t \in \RV(\mathbb{S})$). Moreover,  a definable bijection  between two objects of $\Gamma[k]$ is a  $\Gamma[k]$-morphism if and only if it is definably a piecewise $\mgl_k(\Z)$-transformation. The ``if'' direction is clear. For the ``only if'' direction, see \cite[Lemma~10.1]{hrushovski:kazhdan:integration:vf} or \cite[Lemma~2.29]{Yin:int:expan:acvf}.
\end{rem}

\begin{defn}[$\Gamma$-categories with volume forms]\label{def:Ga:cat:mu}
An object of the category $\mG[k]$ is a definable pair $(I, \omega)$, where $I \in \Gamma[k]$ and $\omega : I \fun \Gamma$ is a function. A $\mG[k]$-morphism between two objects $(I, \omega)$, $(J, \sigma)$ is a $\Gamma[k]$-morphism $F : I \fun J$ such that for all $\alpha  \in I$,
\[
\omega(\alpha) = \sigma(F(\alpha)) + \jcb_{\Gamma} F(\alpha).
\]

The category $\mG^{\fin}[k]$ is the obvious full subcategory of $\mG[k]$.
\end{defn}

\begin{nota}\label{comp:form}
For each $(U, f, \omega) \in \mgRV[k]$, write $\omega_f : U \fun \Gamma$ for the function given by $u \efun \Sigma(\vrv \circ f)(u) + \omega(u)$. Similarly, for $(I, \sigma) \in \mG[k]$, write $\sigma_I : I \fun \Gamma$ for the function given by $\gamma \efun \Sigma \gamma + \sigma(\gamma)$.
\end{nota}

%

There is a natural map $\Gamma[*] \fun \RV[*]$ given by $I \efun (I^\sharp, \id)$ (see Notation~\ref{gamwhat}). This map induces a commutative diagram in the category of graded semirings:
\[
\bfig
 \Square(0,0)/>->`>->`>->`>->/<400>[{\gsk \Gamma^{\fin}[*]}`{\gsk \RES[*]}`{\gsk \Gamma[*]}`{\gsk \RV[*]}; ```]
\efig
\]
where all the arrows are monomorphisms. The map
\[
\gsk \RES[*] \times \gsk \Gamma[*] \fun \gsk \RV[*]
\]
determined by the assignment
\begin{equation}\label{bil:def}
([(U, f)], [I]) \efun [(U \times I^\sharp, f \times \id)]
\end{equation}
is well-defined and is clearly $\gsk \Gamma^{\fin}[*]$-bilinear. Hence it induces a $\gsk \Gamma^{\fin}[*]$-linear map
\begin{equation}\label{tensor:expre}
\Psi: \gsk \RES[*] \otimes_{\gsk \Gamma^{\fin}[*]} \gsk \Gamma[*] \fun \gsk \RV[*],
\end{equation}
which is a homomorphism of graded semirings. By the universal mapping property, groupifying a tensor product in the category of $\gsk \Gamma^{\fin}[*]$-semimodules is, up to isomorphism, the same as taking the corresponding tensor product in the category of $\ggk \Gamma^{\fin}[*]$-modules; the groupification of $\Psi$ is still denoted by $\Psi$. Similarly, there is a
$\gsk \mG^{\fin}[*]$-linear map
\begin{equation}\label{bdd:to:2bdd}
{\mu}\Psi: \gsk \mgRES[*] \otimes_{\gsk \mG^{\fin}[*]} \gsk \mG[*] \fun \gsk \mgRV[*].
\end{equation}
We shall abbreviate $\otimes_{\gsk \Gamma^{\fin}[*]}$, $\otimes_{\gsk \mG^{\fin}[*]}$ both as ``$\otimes$'' below when no confusion can arise.

\begin{prop}[{\cite[Corollary~10.3, Proposition~10.10(1)]{hrushovski:kazhdan:integration:vf}}]\label{red:D:iso}
$\Psi$ and ${\mu}\Psi$ are both isomorphisms of graded semirings.
\end{prop}

Note that, however, \cite[Proposition~10.10(2)]{hrushovski:kazhdan:integration:vf} does not hold. One of the main motivations of this paper is to construct an alternative to be used in applications.

\begin{nota}\label{drop:0}
For simplicity, we often drop the constant volume form $0$ from the notation. For instance, if $\bm U$ is an object of $\RV[*]$ then it may also denote the  object $(\bm U, 0)$ of $\mgRV[*]$ with the constant volume form $0$.
\end{nota}

\begin{nota}
Let $[1] \in \gsk \RES[1]$ be the class of the singleton $\{1\}$. The class of the singleton $\{1\}$ in $\gsk \RES[0]$ is the multiplicative identity of $\gsk \RES[*]$ and hence is simply denoted by $1$. We have $[\RV^{\circ\circ}_\infty] = [\RV^{\circ\circ}] + 1$ in $\gsk \RV[{\leq} 1]$.

Let $\isp$ be the (nonhomogenous) semiring congruence relation on $\gsk \RV[*]$ generated by the pair $([1], [\RV^{\circ\circ}_\infty])$. Let
\[
\bm P = [1]-[\RV^{\ccirc}]  \in \ggk \RV[1].
\]
The corresponding principal ideal of $\ggk \RV[*]$ is thus generated by the element $\bm P - 1$.

Similarly, let $\mgisp$ be the  semiring congruence relation on $\gsk \mgRV[*]$ generated by the pair $([1], [\RV^{\circ\circ}])$, which is homogenous, and the corresponding principal ideal of $\ggk \mgRV[*]$ is generated by the element $\bm P$.
\end{nota}

\begin{nota}\label{RVcat:tag}
For each  $\bm U = (U, f) \in  \RV[k]$, let $U_f$ be the set $\bigcup \set{f(u)^\sharp \times u \given u \in U}$. Let $\mathbb{L}_{\leq k}: \RV[{\leq}k] \fun \VF[k]$ be the map  given by $\bm U \efun U_f$. Set $\bb L = \bigcup_k \bb L_{\leq k}$.

Let $\mgL_k: \mgRV[k] \fun \mgVF[k]$ be the map given by $(\bm U, \omega) \efun (\bb L \bm U, \bb L \omega)$, where $\bb L \omega$ is the obvious function on $\bb L \bm U$ induced by $\omega$. Set $\mgL = \bigoplus_k \mgL_{k}$.
\end{nota}

The map $\bb L$ induces a surjective semiring homomorphism $\gsk \RV[*] \fun \gsk \VF_*$ and the map $\mgL$ induces a surjective graded semiring homomorphism $\gsk \mgRV[*] \fun \gsk \mgVF[*]$, see \cite[\S~4, \S~6]{hrushovski:kazhdan:integration:vf} or \cite[Corollaries~7.7, 10.6]{Yin:special:trans}; we use the same notation for these homomorphisms as well as their groupifications. We have
\[
\bb L ([1]) = [1 + \MM] = [\MM] \dand \bb L ([\RV^{\circ\circ}])=[\MM \mi 0],
\]
and hence $\bb L (\bm P - 1)=0$. Similarly, $\mgL (\bm P)=0$ since, as we have seen above, $\MM$ and $\MM \mi 0$ are in essential bijection. It so happens that these relations are the only ones needed to describe the kernels of $\bb L$ and $\mgL$.

\begin{thm}\label{main:prop:HK}
Suppose that $\mathbb{S}$ is $(\VF, \Gamma)$-generated. For each $k \geq 0$ there is a canonical isomorphism of semigroups
\[
\int_{+} : \gsk  \VF[k] \fun \gsk  \RV[{\leq}k] /  \isp
\]
such that $ \int_{+} [A] = [\bm U]/  \isp$ if and only if $[A] = [{\bb L}{\bm U}]$. These isomorphisms are  compatible with the inductive systems. Passing to  the colimit of the groupifications, we obtain a canonical isomorphism of rings
\[
\int : \ggk \VF_* \fun \ggk  \RV[*] /  (\bm P - 1).
\]
Similarly, for each $k \geq 0$ there is a canonical isomorphism of semigroups
\[
\int_{+}^\mu : \gsk  \mgVF[k] \fun \gsk  \mgRV[k] /  \mgisp
\]
such that $\int_{+} [\bm A] = [\bm U]/  \mgisp$  if and only if $[\bm A] = [{\mgL}{\bm U}]$.
Taking the direct sum of the groupifications yields a canonical isomorphism of graded rings
\[
\int^\mu : \ggk \mgVF[*] \fun \ggk  \mgRV[*] /  (\bm P).
\]
\end{thm}

This is a combination of two main theorems, Theorems~8.8 and 8.29, of \cite{hrushovski:kazhdan:integration:vf}.

\section{Uniform retraction to $\RES$}
\label{sec:retract}

\begin{defn}[Dimension-free $\RES$-categories]\label{df:res}
The objects of the category $\RES$ is obtained from $\RES[*]$ by forgetting the function $f$ in the pair $(U, f)$. Any definable bijection between two such objects is a morphism of $\RES$.
\end{defn}


\begin{rem}
As  the  $\Gamma$-sort is an \omin-minimal group, there is  the \memph{geometric} \omin-minimal Euler characteristic  $\chi_g : \ggk \Gamma[*] \fun \Z$. This is roughly constructed as follows. Every definable set $X \sub \Gamma^k$ is definably bijective to a disjoint union of open cubes $\prod^k_{i= 1}(\alpha_i,\beta_i)$, where $\alpha_i, \beta_i \in \Gamma \cup \{-\infty, +\infty\}$. One sets $\chi_g(\prod^k_{i= 1}(\alpha_i,\beta_i)) = (-1)^k$ and then defines $\chi_g(X)$ by additivity. This does not depend on the chosen partition of $X$ and is invariant under definable bijections. It can also be shown that $\chi_g(X \cap [-\gamma, \gamma]^k)$ stabilizes as $\gamma$ approaches $+\infty$. The induced map, denoted by $\chi_b: \ggk \Gamma[*] \fun \Z$,  is the so-called \memph{bounded} \omin-minimal Euler characteristic. Obviously the two maps $\chi_g$, $\chi_b$  coincide on doubly bounded sets. On the other hand,  $\chi_g((0,+\infty))=-1$ whereas $\chi_b((0,+\infty))=0$. Both of them shall be denoted simply by $\chi$ when no distinction is needed.

See \cite[\S~9]{hrushovski:kazhdan:integration:vf} for more details, where $\chi_g$ is denoted by $\chi$ and $\chi_b$ by $\chi_c$.
\end{rem}

\begin{nota}\label{res:note}
Let $[\A] \in \ggk \RES$ denote the class of the affine line over the residue field. The class of the multiplicative torus $[\A] - 1$ over the residue field is written as $[\G_m]$; note that the multiplicative identity $1$ of $\ggk \RES$ is indeed the class $[1]$, but $[1]$ is not the multiplicative identity of the graded ring $\ggk \RES[*]$.

Let $!I$ be the ideal of $\ggk \RES$ generated by the elements $[\gamma^\sharp] - [\G_m]$, where $\gamma \in \Gamma$ is definable. The quotient ring $\ggk \RES / !I$ is written as $\sggk \RES$. The ideal $!I[*]$ of $\ggk \RES[*]$ and the (graded) quotient ring $\sggk \RES[*] = \ggk \RES[*] / !I[*]$ are constructed in the same way.

The ideal of $\ggk \mgRES[*]$ generated by the  elements
\[
[(\G_m, \alpha + \beta)] - [(\alpha^\sharp, \beta)] \in \ggk \mgRES[1],
\]
where $\alpha, \beta \in \Gamma$ are definable, is denoted by $!\mu I[*]$, and the (graded) quotient ring $\ggk \mgRES[*] / !\mu I[*]$ by $\sggk \mgRES[*]$.
\end{nota}

The quotient maps from ``$\ggk$'' to ``$\sggk$'' will all be denoted by $\iota$, and will often be omitted when it is clear in the context that the latter one is meant. We will also use the same notation for elements when passing from the former to the latter.

Note that the category $\mgRES[*]$ and the ring $\sggk \mgRES[*]$  will be dismissed after Proposition~\ref{prop:eu:retr:k} as inadequate for our purpose, and a new construction is introduced thereafter.

\subsection{Without volume forms}\label{subs:uni:retr:novol}

The discussion here more or less follows \cite[\S~2.5]{hru:loe:lef} (also see \cite[Theorem~10.5]{hrushovski:kazhdan:integration:vf}).

For $0 \leq l \leq n$, consider the maps
\[
\epsilon^{n,l}_g : \ggk \RES[n{-}l] \otimes \ggk \Gamma[l] \fun \sggk \RES[n]
\]
given by
\begin{equation}\label{kill:gamma}
x \otimes y \efun \chi_g(y) x [\G_m]^l.
\end{equation}
Note that here we are forced to pass from ``${\ggk}$'' to ``$\sggk$'' in the target by the tensor $\otimes_{\gsk \Gamma^{\fin}[*]}$, for otherwise $\epsilon^{k,l}_g$ is not well-defined. Then $\bigoplus_{0 \leq l \leq n} \epsilon^{n,l}_g$ and $\Psi^{-1}$ induce a homomorphism
\[
\ggk \RV[n] \fun \sggk \RES[n].
\]
Composing this with the homomorphism $\sggk \RES[n]  \fun \sggk \RES[m]$
given by $x \efun x [\A^{m-n}]$ and taking the direct sum of the resulting homomorphisms over all $n$ with $n \leq m$, we obtain a homomorphism
\[
\epsilon^m_g : \ggk \RV[{\leq} m] \fun \sggk \RES[m].
\]
These homomorphisms $\epsilon^m_g$, $m = 0, 1, \ldots$, are compatible in the obvious sense. Consequently, they give rise to a homomorphism
\[
\bb E_g : \ggk \RV[*] \fun \sggk \RES[*][[\A]^{-1}]
\]
whose range is precisely the zeroth piece $(\sggk \RES[*][[\A]^{-1}])_0$ of the $\Z$-graded ring on the righthand side. We have
\begin{equation}\label{eg:isp:van}
\bb E_g([\RV^{\ccirc}] + 1 - [1]) = - [\G_m][\A]^{-1} + 1 - [\A]^{-1} = 0.
\end{equation}
Thus $\bb E_g$ induces an eponymous homomorphism
\[
\bb E_g : \ggk \RV[*] / (\bm P - 1) \fun \sggk \RES[*][[\A]^{-1}].
\]

There is a parallel construction, where  we replace $\chi_g$ with $\chi_b$ and $[\A]^{-1}$ with $[1]^{-1}$. The latter replacement is needed so to make $(\bm P - 1)$ vanish as in (\ref{eg:isp:van}). The resulting homomorphism is
\begin{equation}\label{EB}
\bb E_b : \ggk \RV[*] / (\bm P - 1) \fun \sggk \RES[*][[1]^{-1}],
\end{equation}
whose range is precisely the zeroth piece  $(\sggk \RES[*][[1]^{-1}])_0$.

\begin{rem}\label{eb:no:eg}
The zeroth graded piece $(\sggk \RES[*][[\A]^{-1}])_0$ of $\sggk \RES[*][[\A]^{-1}]$ is canonically isomorphic to a colimit of the groups $\sggk \RES[n]$, which is actually what appears in the construction above. Thus there is an epimorphism from $(\sggk \RES[*][[\A]^{-1}])_0$ to the similar-looking ring $\sggk \RES[[\A]^{-1}]$. It is then routine to check that this epimorphism is also injective, and hence $(\sggk \RES[*][[\A]^{-1}])_0$ is canonically isomorphic to $\sggk \RES[[\A]^{-1}]$.
For the same reason, $(\sggk \RES[*][[1]^{-1}])_0$ is canonically isomorphic to $\sggk \RES[[1]^{-1}] \cong \sggk \RES$.

If the codomain of $\bb E_b$ is changed to $\sggk \RES[[\A]^{-1}]$ in the obvious way then it may be compared with $\bb E_g$. They are different since
\begin{equation}\label{eb:neq:eg:counter}
\bb E_b([1]) = 1 \dand \bb E_g([1]) = \bb E_g([\RV^{\circ\circ}]) + 1 = -[\G_m][\A]^{-1} + 1 = [\A]^{-1}.
\end{equation}
We can equalize them by forcing $[\A] = 1$ (hence $\bb E_{g}(x \otimes y) = \bb E_{b}(x \otimes y) = 0$ if $y \notin \ggk \Gamma[0]$). The (complex version of the) resulting homomorphism is indeed the one constructed in \cite[(2.5.7)]{hru:loe:lef}.
\end{rem}

The homomorphism $\bb E_b$ will be used in constructing motivic Milnor fiber below, but not $\bb E_g$. It is not that  $\bb E_g$ would not work; it is just that, conceptually, $\bb E_b$ is a better choice,  see Remarks~\ref{eb:not:eg} and \ref{rem-comp-milnor-fib}.


\subsection{With volume forms}
Let $\fn(\Gamma, \Z)$ be the set of functions $f: \Gamma \fun \Z$ such that the range of $f$ is finite and $f^{-1}(m)$ is definable for every $m \in \Z$. Addition in $\fn(\Gamma, \Z)$ is defined in the obvious way. Multiplication in $\fn(\Gamma, \Z)$ is given by the convolution product as follows. For $f, g \in \fn(\Gamma, \Z)$ and $\gamma \in \Gamma$, let $h_\gamma$ be the function $\Gamma \fun \Z$ given by $\alpha \efun f(\alpha)g(\gamma - \alpha)$. The range of $h_\gamma$ is obviously finite. For each $m \in \Z$, $h_\gamma^{-1}(m)$ is a finite disjoint union of sets of the form
\[
f^{-1}(m') \cap (\gamma - g^{-1}(m'')),
\]
where $m' m'' = m$, and hence is $\gamma$-definable. Let
\[
 h_\gamma^* = \sum_{m} m \chi(h_\gamma^{-1}(m)) \in \Z
\]
and $f * g : \Gamma \fun \Z$ be the function  given by $\gamma \efun h_\gamma^*$. It is not hard to see that, by \omin-minimality in the $\Gamma$-sort, $f * g = g * f$ and it belongs to $\fn(\Gamma, \Z)$. Thus $\fn(\Gamma, \Z)$ is a commutative ring.

\begin{nota}\label{gam:equ:cons}
In $\fn(\Gamma, \Z)$, let $r$ be the constant function $1$ and, for each definable element $\gamma \in \Gamma$, let $p_{\gamma}$,  $q_{\gamma}$ be the characteristic functions of $\gamma$, $(\gamma, \infty)$, respectively. By \omin-minimality, $\fn(\Gamma, \Z)$ is generated as a $\Z$-module by elements of the forms $r$, $p_\gamma$, $q_\gamma$. We have the following equalities:
\[
r p_\beta   = r, \quad p_\alpha p_\beta = p_{\alpha+ \beta}, \quad p_\alpha q_\beta  = q_{\alpha+ \beta}, \quad q_\alpha q_\beta = - q_{\alpha+ \beta}.
\]
In addition,
\begin{itemize}
  \item if $\chi = \chi_{g}$ then $r^2 = r q_\beta = - r$,
  \item if $\chi = \chi_{b}$ then $r^2 = r$ and $r q_\beta = 0$.
\end{itemize}
\end{nota}

\begin{rem}\label{int:Gam}
For each $\bm I = (I, \mu) \in \mG[k]$, let $\mu_I : I \fun \Gamma$ be the function given by $\gamma \efun \mu(\gamma) + \Sigma\gamma$ and $\lambda_{\bm I} : \Gamma \fun \Z$  the function given by $\gamma \efun \chi(\mu^{-1}_I(\gamma))$. By \omin-minimality, the range of $\lambda_{\bm I}$ is finite and, for every $m \in \Z$, $\lambda_{\bm I}^{-1}(m)$ is definable. So $\lambda_{\bm I} \in \fn(\Gamma, \Z)$. If $F :(I, \mu) \fun (J, \sigma)$ is a $\mG[k]$-morphism then $F$ restricts to a bijection from $\mu^{-1}_I(\gamma)$ to $\sigma_J^{-1}(\gamma)$ for every $\gamma \in \Gamma$ and hence $\lambda_{\bm I} = \lambda_{\bm J}$ and we may write $\lambda_{\bm I}$ as $\lambda_{[\bm I]}$. Let $\bm J = (J, \sigma) \in \mG[l]$. If $k=l$ then clearly $\lambda_{[\bm I]} + \lambda_{[\bm J]} = \lambda_{[\bm I]+[\bm J]}$. An easy computation shows that, for every $\gamma \in \Gamma$,
\[
(\lambda_{[\bm I]} * \lambda_{[\bm J]})(\gamma) = \chi( (\mu \sigma)^{-1}_{I \times J}(\gamma)),
\]
where $\mu \sigma : I \times J \fun \Gamma$ is the function given by $(\alpha, \beta) \efun \mu(\alpha) + \sigma(\beta)$, and hence $\lambda_{[\bm I]} * \lambda_{[\bm J]} = \lambda_{[\bm I][\bm J]}$. Let $\Z\Gamma$ be the group ring of $\Gamma(\mathbb{S})$ over $\Z$, which is viewed as the subring of $\fn(\Gamma, \Z)$ generated, as a $\Z$-module, by $p_{\gamma}$. We have $\ggk \mG[0] \cong \Z\Gamma$. It follows that the map
\[
\lambda : \ggk \mG[*] \fun \Z\Gamma \oplus X \fn(\Gamma, \Z)[X]
\]
determined by the assignment
\[
[\bm I] \efun \lambda_{[\bm I]}X^k, \quad \bm I \in \mG[k],
\]
is a graded ring homomorphism. Of course, there are two such homomorphisms $\lambda_g$ and $\lambda_b$, corresponding to the two choices $\chi = \chi_{g}$ and $\chi = \chi_{b}$.
\end{rem}


We claim that the assignments
\begin{equation}\label{Gam:ass}
p_\gamma X^k \efun  [\G_m]^{k-1}[(\G_m, \gamma)], \quad q_\gamma X^k \efun  [\G_m]^{k-1}[(\{1\}, \gamma)], \quad r X^k \efun 0
\end{equation}
induce a graded ring homomorphism
\[
\psi : \Z\Gamma \oplus X\fn(\Gamma, \Z)[X] \fun \sggk \mgRES[*] / ([\A]).
\]
Note that $\Z\Gamma$ may be regarded as a subring of $\ggk \mgRES[0]$ by identifying $p_\gamma$ with $[(\infty, \gamma)]$, which is the intended interpretation of the first assignment for $k = 0$. For the claim, it suffices to check that $\psi$ obeys the equational constraints for the generators above in Notation~\ref{gam:equ:cons}, which are all straightforward (the equality $\psi(q_\alpha) \psi( q_\beta) = \psi(- q_{\alpha + \beta})$ holds since $[\A] = [\G_m] +[1] = 0$ is forced).

Let $\gh$ denote the element $[(\RV^{\ccirc},  -\vrv)] \in \ggk \mgRV[1]$.

\begin{prop}\label{prop:eu:retr:k}
There are two graded ring homomorphisms
\[
  \mgE_{g}, \mgE_{b}: \ggk \mgRV[*] \fun   \sggk \mgRES[*] / ([\A])
\]
such that
\begin{itemize}
  \item $\bm P \in \ggk \mgRV[1]$ vanishes under both of them,
  \item for all $x \in \ggk \mgRES[*]$, $\mgE_{g}(x) = \mgE_{b}(x) = x/([\A])$,
  \item they are distinguished by $\mgE_{g}(\gh) = -[1]/([\A])$ and $\mgE_{b}(\gh) = 0$.
\end{itemize}
\end{prop}
\begin{proof}
The product map
\[
(\iota /([\A])) \times (\psi \circ \lambda) : \ggk \mgRES[*] \times  \ggk \mG[*] \fun    \sggk \mgRES[*] /([\A])
\]
is $\ggk \mG^{\fin}[*]$-bilinear (because of  the ideal $!\mu I[*]$ in the target). We have
\[
{\mu}\Psi^{-1}(\bm P) = [1]-[(0, \infty)]  \dand {\mu}\Psi^{-1}(\gh) = [((0, \infty), -\id)].
\]
The existence of the desired homomorphisms follow from a straightforward computation.
\end{proof}

The discussion above is more or less an expanded version of the proof of \cite[Theorem~10.11]{hrushovski:kazhdan:integration:vf}.

\begin{rem}
Alternatively, for $\chi = \chi_{g}$, we may replace $r X^k \efun 0$ in (\ref{Gam:ass}) with $r X^k \efun [\G_m]^{k-1}[1]$, that is, sending $r X^k$ and $q_0 X^k$ to the same element.  Then $\psi$ is still a graded ring homomorphism provided that we enlarge the ideal $!\mu I[*]$ by adding the elements $[1] - [(\{1\}, \gamma)] \in \ggk \mgRES[1]$. Consequently, for instance, we can construct a ring homomorphism
\begin{equation}\label{res:forget}
\ggk \mgRV[*] \fun \sggk \mgRES[*] / ([\A]) \to^{\phi} \sggk \RES[*] / ([\A]),
\end{equation}
where the homomorphism $\phi$ is induced by the obvious forgetful functor.

We shall not, however, pursue this further, since taking the quotient by the ideal $([\A])$ is detrimental to our purpose here. To avoid it, we need to work with the categories of doubly bounded objects.
\end{rem}

\begin{defn}[Doubly bounded $\RV$-categories]\label{db:RV}
An object $(U, f) \in \RV[k]$ is \memph{doubly bounded} if $U$ is. Denote the full subcategory of $\RV[k]$ of doubly bounded objects by $\RV^{\db}[k]$, and similarly for the categories $\mgRV^{\db}[k]$, $\mG^{\db}[k]$.
\end{defn}

\begin{rem}\label{tensor:bd}
By Lemma~\ref{db:to:db}, if $(U, f, \omega) \in \mgRV^{\db}[k]$ then $f(U)$, $\omega(U)$ are both doubly bounded as well. For essentially the same reason, ${\mu} \Psi$ restricts to an isomorphism
\[
{\mu} \Psi^{\db}: \gsk \mgRES[*] \otimes \gsk \mG^{\db}[*] \fun \gsk \mgRV^{\db}[*].
\]
\end{rem}

\begin{nota}\label{db:generator}
In $\fn(\Gamma, \Z)$, for each definable element $\gamma \in \Gamma \mi 0$, let $o_{\gamma}$ be the characteristic function of $(0, \gamma)$; also set $o_0 = -p_0$. As a $\Z$-module, $\fn^{\db}(\Gamma, \Z)$ is generated by elements of the forms $p_\gamma$, $o_\gamma$. Of course $o_{\gamma} = q_0 - q_\gamma - p_\gamma$ if $\gamma> 0$, etc. An easy  computation shows that
\begin{itemize}
  \item $p_\alpha o_\beta$ equals $o_{\alpha + \beta} - o_\alpha - p_\alpha$ or $- o_{\alpha + \beta} + o_\alpha - p_{\alpha + \beta}$ or $ o_{\alpha + \beta} + o_\alpha + p_{0}$,
  \item $o_\alpha o_\beta$ equals $- o_{\alpha + \beta}$ or $- o_{\alpha} - o_\beta - p_{0}$.
\end{itemize}
\end{nota}

\begin{nota}\label{db:gam:fn}
The homomorphism
\begin{equation*}
\lambda : \ggk \mG^{\db}[*] \fun \Z\Gamma \oplus X \fn^{\db}(\Gamma, \Z)[X]
\end{equation*}
is constructed as before (there is only one such homomorphism now since  $\chi_{g}$, $\chi_{b}$ agree on doubly bounded sets).
\end{nota}

\begin{nota}\label{pgammabis}
Recall that $\RV^{\ccirc}_\gamma = \rv(\MM_{\gamma} \mi 0)$ (Notation~\ref{defn:disc}). For each definable $\gamma \in \Gamma^+_0$, let
\[
\bm P_\gamma = [\RV^{\ccirc} \mi \RV^{\ccirc}_{\gamma}] + [\{t_\gamma\}] - [1] \in \ggk \RV^{\db}[1],
\]
where $t_\gamma \in \gamma^\sharp$ is any definable point.  It also stands for the corresponding element  in $\ggk \mRV^{\db}[1]$ (with the constant volume form $0$). Note that, although $[\{t_\gamma\}] = [1]$ in $\ggk \RV^{\db}[1]$,  $[\{t_\gamma\}] \neq [1]$ in $\ggk \mRV^{\db}[1]$  unless $\gamma = 0$.

Clearly  $\bm P_\gamma$ does not depend on the choice of $t_\gamma \in \gamma^\sharp$. The ideal of $\ggk \mRV^{\db}[*]$   generated by the elements $\bm P_\gamma$ is denoted by $(\bm P_\Gamma)$. The images of $(\bm P_\Gamma)$ under the obvious forgetful homomorphisms
\[
\ggk  \mgRV^{\db}[*] \fun \ggk \RV[*], \quad \ggk  \mgRV^{\db}[*] \fun \ggk  \mgRV[*]
\]
are contained in $(\bm P - 1)$, $(\bm P)$, respectively,
\end{nota}

\begin{rem}\label{db:gam:res}
Let $!P$ denote the ideal of $\fn^{\db}(\Gamma, \Z)$ generated by the elements $p_\gamma - p_0$, $o_\gamma + p_0$. By the computation in Notation~\ref{db:generator}, $\fn^{\db}(\Gamma, \Z) / !P \cong \Z$; the quotient homomorphism is denoted by $\chi$ because a simple computation shows that if $\bm I = (I, \mu) \in \mG^{\db}[*]$ then indeed $\lambda_{\bm I} / !P = \chi(I)$. Now it is straightforward to check that the assignments
\begin{equation}\label{Gam:ass:bd}
p_\gamma X^k \efun  [\G_m]^{k}, \quad o_\gamma X^k \efun  - [\G_m]^{k}
\end{equation}
induce a graded ring homomorphism
\[
\psi^{\db} : \Z\Gamma \oplus X\fn^{\db}(\Gamma, \Z)[X] \fun \sggk \RES[*].
\]
In a nutshell, the  homomorphism $\psi^{\db} \circ \lambda$ is a ``forgetful'' map  given by $[(I, \mu)] \efun \chi(I)[\G_m]^{k}$, $(I, \mu) \in \mG^{\db}[k]$.
\end{rem}

The forgetful homomorphism $\phi : \ggk \mgRES[*] \fun \sggk \RES[*]$ is similar to the one in  (\ref{res:forget}).

\begin{prop}\label{prop:eu:retr:k:db}
There is a graded ring homomorphism
\[
  \mgE^{\db}: \ggk \mgRV^{\db}[*] \fun \sggk \RES[*]
\]
such that $(\bm P_\Gamma)$ vanishes and, for all $x \in \ggk \mgRES[*]$, $\mgE^{\db}(x) = \phi(x)$.
\end{prop}
\begin{proof}
The homomorphism $\phi \times (\psi^{\db} \circ \lambda)$ is still $\ggk \mG^{\fin}[*]$-bilinear. We have
\[
({\mu} \Psi^{\db})^{-1}(\bm P_\gamma) = [(0, \gamma)] + [\{\gamma\}] + [\{t_\gamma\}] - [1],
\]
which vanishes under $\phi \otimes (\psi^{\db} \circ \lambda)$.
\end{proof}

The composition of $\mgE^{\db}$ with the obvious forgetful homomorphism $\sggk \RES[*] \fun \sggk \RES$ is denoted by $\bb E^{\diamond}$.  By Remark~\ref{db:gam:res}, the diagram commutes:
\begin{equation}\label{edb:eb:com}
\bfig
 \Vtriangle(0,0)/->`->`->/<400, 400>[{\ggk \mgRV^{\db}[*]}`{\ggk \RV[*]}`\sggk \RES; `\bb E^{\diamond}`\bb E_b]
\efig
\end{equation}
which may serve as an alternative and more direct (but less comprehensive) construction of $\bb E^{\diamond}$.

\begin{rem}\label{eb:not:eg}
Changing $\sggk \RES$ to $\sggk \RES[[\A]^{-1}]$ in (\ref{edb:eb:com}), we see that  $\bb E_g$ does not commute with $\bb E^{\diamond}$: by (\ref{eb:neq:eg:counter}),
\[
\bb E_g([1]) = [\A]^{-1} \quad \text{whereas} \quad \bb E^{\diamond}([1]) = \bb E_b([1]) = 1.
\]
This is but an example of a general phenomenon, see Remark~\ref{eg:eb:factor}.

We can certainly redefine $\bb E^{\diamond}$ as the composition of $\mgE^{\db}$ with  the homomorphism $\sggk \RES[*] \fun \sggk \RES[[\A]^{-1}]$ implicit in Remark~\ref{eb:no:eg} and thereby make it compatible with $\bb E_g$, but as we have mentioned above, $\bb E_b$ is ultimately a better choice; see Remark~\ref{rem-comp-milnor-fib}.
\end{rem}

\section{Interlude: Reduced cross-section and Fubini theorems}\label{csndag}

A \memph{cross-section} of $\Gamma$ is a group homomorphism $\csn : \Gamma \fun \VF^{\times}$ such that $\vv \circ \csn = \id$. The corresponding \memph{reduced cross-section} of $\Gamma$ is the function
\[
\rcsn = \rv \circ \csn : \Gamma \fun \RV^{\times}.
\]
These are usually augmented by $\csn(\infty) = 0$ and $\rcsn(\infty) = \infty$.

We can add a reduced cross-section $\rcsn$ to the language $\lan{}{RV}{}$,  denoted by $\lan{}{RV}{}^\dag$ the extension, and consider the corresponding integration theory; this has been worked out in \cite{Yin:int:expan:acvf}. We shall, however, only need a few facts about definable sets in $\RV$ in this setting. Strictly speaking, this extra ingredient is not really needed, but it does help simplify certain computations.

In this section we work in an $\lan{}{RV}{}^\dag$-expansion $\bb U^\dag$ of $\bb U$. Definability, if unqualified, is interpreted accordingly.

\begin{defn}
A \memph{$\Gamma$-partition} of a definable set $A$ in $\RV$ is a definable function $p : A \fun \Gamma_{\infty}^l$ such that, for all $\gamma \in \Gamma_{\infty}^l$, $\vrv(A_{\gamma})$ is a singleton and is $\rcsn(\gamma)$-$\lan{}{RV}{}$-definable. If $p$ is a $\Gamma$-partition of $A$ then the \memph{$\RV^\dag$-dimension} of $p$, denoted by $\dim_{\RV^\dag} (p)$, is the number $\max \set{\dim_{\RV} (A_{\gamma}) \given \gamma \in \Gamma_{\infty}^{l} }$.
\end{defn}

\begin{lem}[{\cite[Lemma~3.2]{Yin:int:expan:acvf}}]\label{G:part:dim:inv}
If $p_1$, $p_2$ are $\Gamma$-partitions of $A$ then  $\dim_{\RV^\dag} (p_1) = \dim_{\RV^\dag} (p_2)$.
\end{lem}

So the $\RV^\dag$-dimension $\dim_{\RV}(A)$ of a definable set $A$ in  $\RV$ may be defined as the $\RV^\dag$-dimension of any $\Gamma$-partition of $A$. It may also be shown that there is a definable finite-to-one function $f : A \fun \RV^k \times \Gamma_{\infty}^l$ if and only if there is a definable function $f : A \fun \RV^k$ such that all fibers of $f$ are of $\RV^\dag$-dimension $0$ if and only if $\dim_{\RV^\dag}(A) \leq k$.

\begin{defn}[$\RV^\dag$-categories]\label{defn:f:RV:cat}
The objects of the category $\RV^\dag[k]$ are the pairs $(U, f)$ with $U$ a set in $\RV_\infty$ and $f : U \fun \RV^k$ a definable function such that $\dim_{\RV^\dag}(U_t) = 0$ for all $t \in \RV^k$. For two such objects $(U, f)$, $(V, g)$, any definable bijection $F : U \fun V$ is a \memph{morphism} of $\RV^\dag[k]$.
\end{defn}

The categories $\RES^\dag[k]$, $\RES^\dag$ are formulated similar to $\RES[k]$, $\RES$.

\begin{defn}\label{twistback}
The \memph{twistback} function $\tbk : \RV \fun \K$ is given by $u \efun u / \rcsn(\vrv(u))$, where $\infty / \infty = 0$. For any set $U \sub \RV^n_\infty$ and $\gamma \in \Gamma_{\infty}^n$, the set $\tbk(U_{\gamma}) \sub \K^n$ is called the \memph{$\gamma$-twistback} of $U$. If $\tbk(U_{\gamma}) = \tbk(U_{\gamma'})$ for all $\gamma, \gamma' \in \vrv(U)$ then $U$ is called a \memph{twistoid}, in which case we simply write $\tbk(U)$ for the unique twistback.
\end{defn}

\begin{lem}[{\cite[Corollaries~2.23, 3.4]{Yin:int:expan:acvf}}]\label{gam:red:lrv}
Every definable set in $\Gamma$ is $\lan{}{RV}{}$-definable and every definable set in $\RV$ with $\vrv(U)$ finite is $\rcsn(\vrv(U))$-$\lan{}{RV}{}$-definable.
\end{lem}

\begin{lem}[{\cite[Lemma~3.3]{Yin:int:expan:acvf}}]\label{shift:K:csn:def}
Let $U \sub \RV^n$ be a definable set and $\vrv(U) = D$. Then there is a definable finite partition $(D_i)_i$ of $D$ such that each $U_i = U \cap D_i^\sharp$ is a definable twistoid and the corresponding twistback is $\lan{}{RV}{}$-definable.
\end{lem}

A definable finite partition $(U_i)_i$ of $U$ is called a \memph{twistoid decomposition} of $U$ if every $U_i$ is a twistoid. For instance,  the partition $(U_i)_i$ of $U$ induced by the partition $(D_i)_i$ of $D$ above is a \memph{$\Gamma$-cohesive} twistoid decomposition, which, by Lemma~\ref{gam:red:lrv}, is  $\lan{}{RV}{}$-definable if $U$ is.

\begin{ter}\label{twist:rv:ob}
For any object $\bm U = (U, f) \in \RV^\dag[*]$, by a ``twistoid decomposition of $\bm U$'' we actually mean a twistoid decomposition $(V_i)_i$ of the graph of the function $f$. In that case, clearly $(U_i)_i = (\pr_U(V_i))_i$ is a twistoid decomposition of $U$. For convenience, we shall refer to $(U_i)_i$ as a twistoid decomposition of $\bm U$.
\end{ter}

Let  $(U_i)_i$ be a twistoid decomposition of $\bm U = (U, f) \in \RV^\dag[*]$. Set
\begin{equation}\label{ebdag}
 \bb E^\dag_b(\bm U) = \sum_i \chi_b(\vrv(U_i))[\tbk(U_i)] \in \ggk \RES^\dag.
\end{equation}

\begin{prop}[{\cite[Propositions~3.21, 3.30]{Yin:int:expan:acvf}}]
The assignment (\ref{ebdag}) does not depend on the choice of the twistoid decomposition  and is invariant on isomorphism classes. The resulting map
\[
\bb E^\dag_b : \ggk \RV^\dag[*] \fun \ggk \RES^\dag
\]
is a ring homomorphism that vanishes on the ideal $(\bm P - 1)$.
\end{prop}

We consider $\RV[*]$ as a subcategory of $\RV^\dag[*]$ and denote the induced homomorphism between the Grothendieck rings by
\[
\Lambda : \ggk\RV[*] \fun \ggk\RV^\dag[*].
\]

In the current environment, the transition from ``$\ggk \RES^\dag$'' to ``$\sggk \RES^\dag$'' is superfluous since we already have $[\gamma^\sharp] = [\G_m]$ in $\ggk \RES^\dag$ for all definable $\gamma \in \Gamma$ (every element in  $\gsk \RES^\dag$ is represented by a definable set in $\K$). So there is a natural homomorphism  $\sggk \RES \fun \ggk \RES^\dag$, which is also denoted by $\Lambda$. It then follows from  Proposition~\ref{red:D:iso}, more precisely, the correspondence (\ref{bil:def}), that we have a commutative diagram
\begin{equation}\label{dag:red}
\bfig
\Square(0,0)/->`->`->`->/<400>[{\ggk \RV[*]}`{\sggk \RES}`{\ggk \RV^\dag[*]}`{\ggk \RES^\dag}; \bb E_b`\Lambda`\Lambda`{\bb E^\dag_b}]
\efig
\end{equation}

\begin{rem}\label{dag:g}
We may replace $\chi_b$ with $\chi_g$ and $[\tbk(U_i)]$ with $[\tbk(U_i)][\A]^{-n}$ (assuming $\bm U \in \RV^\dag[n]$) in (\ref{ebdag}) and thereby obtain a ring homomorphism
\[
\bb E^\dag_g : \ggk \RV^\dag[*] \fun \ggk \RES^\dag[[\A]^{-1}]
\]
that also vanishes on the ideal $(\bm P - 1)$. The diagram (\ref{dag:red}) still commutes if $\bb E_b$, $\bb E^\dag_b$ are replaced by $\bb E_g$, $\bb E^\dag_g$.
\end{rem}

One main advantage of $\bb E^\dag_b$ over $\bb E_b$ is that it makes computations much easier, essentially because  we can use twistoid decompositions given by Lemma~\ref{shift:K:csn:def} instead of ``tensor decompositions'' given by Proposition~\ref{red:D:iso}. To make (\ref{ebdag}) work for $\bb E_b$ as intended, though, we need a special kind  of twistoid decompositions:

\begin{defn}\label{bipo:defn}
Suppose that $(U_i)_i$ is an $\lan{}{RV}{}$-definable twistoid decomposition of $U$.
If, for each $i$, there is an $\lan{}{RV}{}$-definable bijection $U_i \fun V_i \times I_i^\sharp$, where $\vrv(V_i)$ is a singleton and $I_i \in \Gamma[*]$, such that its $\vrv$-contraction is a bijection  and its graph is a twistoid as well, then we say that $(U_i)_i$ is  \memph{bipolar}. Naturally $U$ is a \memph{bipolar twistoid} if it admits a trivial bipolar twistoid decomposition.
\end{defn}

\begin{rem}\label{bipo:desc}
The point of Definition~\ref{bipo:defn} is this: Suppose that $\bm U = (U, f) \in \RV[n]$ and  $(U_i)_i$ is a bipolar twistoid decomposition of $U$, as witnessed by bijections $U_i \fun V_i \times I_i^\sharp$ with $I_i \in \Gamma[l_i]$. By dimension theory, for each $i$, there is an $\lan{}{RV}{}$-definable finite-to-one function $g_i: V_i \fun \RV^{n-l_i}$, that is, $(V_i, g_i) \in \RES[n{-}l_i]$. Let $\gamma_i$ be a definable element in $\vrv(U_i)$. Since $U_{i,\gamma_i}$ is $\lan{}{RV}{}$-definably bijective to $V_i \times \alpha_i^{\sharp}$ for some definable $\alpha_i \in I_i$, we see that, in $\sggk \RES$
\[
[U_{i,\gamma_i}] = [V_i] [\alpha_i^{\sharp}] = [V_i] [\G_m]^{l_i}.
\]
So,  by the construction of $\bb E_b$ and bipolarity,
\begin{equation}\label{Eb:twist:com}
\bb E_b([\bm U]) = \sum_i \chi_b(I_i)[V_i] [\G_m]^{l_i} = \sum_i \chi_b(\vrv(U_i))[U_{i,\gamma_i}].
\end{equation}
Observe that this does not really follow without some stronger condition such as bipolarity.
\end{rem}

Obviously a $\Gamma$-cohesive twistoid decomposition of a bipolar twistoid is also bipolar. There is indeed an abundant supply of bipolar twistoid decompositions:

\begin{lem}\label{coh:decom}
Every $\lan{}{RV}{}$-definable set $U \sub \RV^n$ admits a $\Gamma$-cohesive twistoid decomposition that is also bipolar.
\end{lem}
\begin{proof}
This follows from \cite[Lemmas~3.21, 3.25]{hrushovski:kazhdan:integration:vf}. In more detail, the proof of \cite[Lemmas~3.21]{hrushovski:kazhdan:integration:vf} shows that the definable finite partition $(D_i)_i$ given by Lemma~\ref{shift:K:csn:def} can be refined so as to make the following condition hold:  each $U_i = U \cap D_i^\sharp$ is of the form $\set{t \in D_i^\sharp \given N_it \in W_i}$, where $N_i$ is a $k \times n$ matrix over $\Z$ and $W_i$ is a definable subset of $\alpha_i^\sharp$ for some $\alpha_i \in \Gamma^k$. Next, there exists a matrix $M_i \in \mgl_n(\Z)$ such that $N_iM_i$ is in lower echelon form (in general $M_i$ is not a product of ``standard'' column operations since $\Z$ is not a field, but  such an invertible matrix exists over any  principal ideal domain). Observe that if $N_iM_i$ does not have zero columns then $D_i$ must be a singleton. At any rate, the set $M_i^{-1}U_i$ must be of the form $\pr_{\leq m}(U_i) \times I_i^\sharp \sub \RV^n$, where $n-m$ is the number of zero columns in $N_iM_i$ and $\vrv(\pr_{\leq m}(U_i))$ is a singelton.
\end{proof}

\subsection{Tropical motivic Fubini theorem}\label{jp:fubini}

In this subsection we work in the field $\puC$ of complex Puiseux series, considered as an elementary substructure of $\bb U$, and take the parameter space  to be the substructure $\mathbb{S} = \C \dpar{ t }$. The value group $\Gamma$ is identified with $\Q$. As an elementary substructure of $\bb U^\dag$, $\puC$ is denoted by  $\tilde\C^{\dag}$. We reiterate that the reduced cross-section $\rcsn$ is not something of central interest, it is just an (implicit) computational tool.

\begin{rem}
We may think of the profinite group $\hat {\mu}$ as the Galois group $\gal(\puC / \C \dpar t)$, since they are canonically isomorphic. For each element $\xi = (\xi_n)_n \in \hat \mu$, the assignment $n \efun \xi_n t^{1/n}$ indeed induces a reduced cross-section $\rcsn_\xi : \Q \fun \RV$, and the map given by $\xi \efun \rcsn_\xi$ is indeed a bijection between $\hat \mu$ and the set $\Omega$ of reduced cross-sections $\rcsn : \Q \fun \RV$ with $\rcsn(1) = \rv(t)$; in other words, $\hat \mu$ acts freely and transitively on $\Omega$ via multiplication in the obvious way.

Here is a slightly different perspective on $\hat \mu$. By the structural theory of valued fields (in particular, \cite[Lemma~5.3.2]{engler:prestel:2005}), an element $\sigma \in \gal(\puC / \C \dpar t)$ is in the ramification subgroup if and only if it fixes $\RV$ pointwise. But it can be easily checked that every $\sigma \in \gal(\puC / \C \dpar t)$ moves some element of $\RV$ unless $\sigma = \id$. So $\gal(\puC / \C \dpar t) \cong \hat \mu$ may be identified with $\aut(\RV / \RV(\C \dpar t))$, the group of  model-theoretic automorphisms, where $\RV(\C \dpar t)$ is equal to the subgroup generated by $\rv(t)$ over $\K^\times$.
\end{rem}

\begin{nota}\label{bipo:gdv}
The presence of $\rcsn$ induces an intrinsic isomorphism $\RV \cong {\K^\times} \oplus \Q$. Consequently, $\ggk \RES^\dag \cong \ggk \var_{\C}$. Let $\phi : \gmv \fun \ggk \var_{\C}$ be the obvious forgetful homomorphism.  Using the twistback function $\tbk$ as in  \cite[\S~4.3]{hru:loe:lef}, we construct an isomorphism
\begin{equation}
\Theta : \sggk \RES \fun \gmv
\end{equation}
such that $\phi \circ \Theta = \Lambda$. By the discussion in Remark~\ref{bipo:desc}, if $U$ is a bipolar twistoid then the class $[U_\gamma] \in \sggk \RES$ is invariant for definable $\gamma \in \vrv(U)$, and hence it makes sense to denote the class $\Theta([U_\gamma]) \in \gmv$ by $[\tbk(U)]^{\hat \mu}$. In general, for any $\bm U = (U, f) \in \RV[*]$, we may write
\begin{equation}\label{twist:C}
  (\Theta \circ \bb E_b)([\bm U]) = \sum_i \chi_b(\vrv(U_i))[\tbk(U_{i})]^{\hat \mu} \in \gmv,
\end{equation}
where $(U_i)_i$ is any bipolar twistoid decomposition of $U$.
\end{nota}

\begin{nota}\label{mov:vol}
The \memph{motivic volume operator} is the composition
\[
  \Vol : \ggk \VF_* \to^{\int} \ggk \RV[*] / (\bm P - 1)    \to^{\bb E_b} \sggk \RES \to^{\Theta} \gmv.
\]
\end{nota}

For each $\gamma \in \Q^n$, let $A_\gamma \in \VF_*$ and $\bm U_\gamma \in \RV[*]$ such that $\int [A_\gamma] = [\bm U_\gamma] / (\bm P - 1)$,  which we think of  as holding at the semiring level. Suppose that $A = \bigcup_{\gamma \in \Q^n} \gamma^\sharp  \times A_\gamma \in \VF_*$, that is, $A$ is definable, for instance,  there could be a definable set $B \sub \VF^n$ such that $A_\gamma = B \cap \gamma^\sharp$ for all $\gamma \in \Q^n$. For ease of notation, we take each $\bm U_\gamma$ to be simply a definable set $U_\gamma$ in $\RV$ and assume, by compactness, $U = \bigcup_{\gamma \in \Q^n} \gamma^\sharp  \times U_\gamma$ is definable as well. By Lemma~\ref{coh:decom} and compactness again, there is a definable finite partition $(D_i)_i$ of $\vrv(U)$ such that
\begin{itemize}
  \item every set $U \cap D_{i}^\sharp$ is a bipolar twistoid,
  \item the sets $E_i = \pr_{\leq n}(D_i)$ form a partition of $\Q^n$,
  \item for all $\gamma, \gamma' \in E_i$, $\chi_b(D_{i,\gamma}) = \chi_b(D_{i,\gamma'})$.
\end{itemize}
It follows  that the map $\Q^n \fun \ggk^{\hat \mu} \var_{\C}$ given by
\begin{equation}\label{trop:fub:cons}
\gamma \efun \Vol([A_\gamma]) = \sum_{i} \chi_b(D_{i,\gamma}) [\tbk (U_\gamma \cap D^\sharp_{i,\gamma})]^{\hat \mu}
\end{equation}
is constant over each $E_i$, which we denote by $v_i$, and hence
\[
 \Vol([A]) = \sum_i \chi_b(E_i)v_i[\G_m]^{n}.
\]
This  is essentially the content of the tropical motivic Fubini theorem  of \cite{Nic:Pay:trop:fub}.

\begin{rem}
The definable sets $E_i$ may be viewed as rational polyhedrons in $\R^n$, that is, intersections of closed half spaces in $\R^n$ defined by linear equations with rational coefficients and (definable) constant terms. In general they are not cones, in other words, some of the constant terms may be nonzero (every element in $\Q$ is definable). This issue is moot, however, if the parameter space is $\mathbb{S} = \C$, since in that case the only definable element in $\Q$ is $0$. This is needed for the applications in the motivic Donaldson-Thomas theory in  \cite{Nic:Pay:trop:fub}.
\end{rem}

\subsection{Fubini theorem}
Let $A, A' \in \VF_*$ and $f : A \fun B$, $f': A' \fun B$ be definable functions. Write the fibers $f^{-1}(b)$, $b \in B$, as $A_b$, similarly for $f'^{-1}(b)$. The purpose of this subsection is to prove the following result.

\begin{thm}\label{fubini}
If, for all $b \in B$, there are $b$-definable objects $\bm U_{b}, \bm U'_{b} \in \RV[*]$ such that $[A_b] = [\bb L \bm U_{b}]$, $[A'_b] = [\bb L \bm U'_{b}]$, and $\bb E_b ([\bm U_b]) = \bb E_b (  [\bm U'_b])$, then $\bb E_b (\int [A]) = \bb E_b ( \int [A'])$.
\end{thm}

\begin{rem}\label{rem:volb}
Suppose that $B$ is a set in $\VF$, or more generally, for every $b \in B$, the substructure $\bb S \la b \ra$ generated by it over $\bb S$ is $\VF$-generated. The latter condition holds if, for instance, there is a definable bijection $B \fun B'$ with $B'$ a set in $\VF$. Then the hypothesis of the theorem can be simply reformulated as $\bb E_b (\int [A_b]) = \bb E_b ( \int [A'_b])$. This equality is construed over  $\bb S \la b \ra$, which makes sense since $\bb S \la b \ra$ is $\VF$-generated (recall the provisions of Theorem~\ref{main:prop:HK}).

As in \S~\ref{jp:fubini}, we are mainly interested in applying this result in $\puC$ with $\bb S = \C \dpar{ t }$, in which case $\bb S\la b \ra = \C \dpar{ t }(b)$ is an algebraic extension of $\C \dpar{ t }$ and the conclusion of the theorem  may take the form $\Vol([A]) = \Vol([A'])$; however,  it is clearly not enough to test the hypothesis just inside $\puC$.

Take $\bb S = \C \dpar{ t }$  and $b \in B$ in $\bb U$. Denote the algebraic closure of the residue field $\K(\C \dpar{ t }( b ))$ by $\C_b$. Let $r_b$  be the $\Q$-rank of the value group $\Gamma(\C \dpar{ t }( b ))$. Then, working in the algebraic closure of $\C \dpar{ t }( b )$ (this is indeed the smallest model of $\ACVF(0, 0)$ that contains $\C \dpar{ t }( b )$), we can construct, similar to $\Theta$, an isomorphism
\[
\Theta_b:  \sggk \RES \fun \ggk^{{\hat \mu}^{r_b}}\var_{\C_b}.
\]
Setting $\Vol_b = \Theta_b \circ \bb E_b \circ \int$, the hypothesis of the theorem can then be equivalently reformulated as  $\Vol_b([A_b]) = \Vol_b([A'_b])$ for all $b \in B$. Since this construction is not really needed, we omit the details. Only note that if $\C \dpar{ t }( b )$ is an unramified extension of $\C \dpar{ t }$ then $\Theta_b$ lands in $\ggk^{{\hat \mu}}\var_{\C_b}$ and hence it is essentially the same as $\Theta$. It will also be clear from the proof that one can test the hypothesis only for those $b \in B$ that generate such unramified extensions (because ultimately the task can be reduced to checking a relation among certain objects of $\RES$ whose projection into $\Gamma$ is in $\Z = \Gamma(\bb S)$).
\end{rem}

The following fact shall be needed:

\begin{lem}\label{special:term:constant:disc}
Let $F_i(x) = F_i(x_1, \ldots, x_n)$ be a finite list of polynomials with coefficients in $\VF(\bb S)$. Let $A \sub \VF^n$ be a definable set.  Then there is a definable bijection $T : A \fun A^\flat$ such that $A^\flat$ is an $\RV$-pullback and  every composite function $F_i \circ T^{-1}$ is $\rv$-contractible, that is, for each $i$ there is a commutative diagram
\[
\bfig
  \square(0,0)/`->`->`->/<1000,400>[A^\flat`\VF`\rv(A^\flat)`\RV_\infty;
  `\rv`\rv`(F_i \circ  T^{-1})_{\downarrow}]
  \morphism(0,400)<500,0>[A^\flat`A; T^{-1}]
  \morphism(500,400)<500,0>[A`\VF; F_i]
 \efig
\]
\end{lem}
This lemma follows from \cite[Theorem~5.5]{Yin:special:trans}, cited as Lemma~\ref{special:bi:term:constant} below, where $T$ may be taken to be a so-called special bijection, as in Definition~\ref{defn:special:bijection} but much simpler (we shall not need this here, though, because no volume forms are involved).

\begin{proof}[Proof of Theorem~\ref{fubini}]
By partitioning $B$ if necessary, it is enough to prove the claim either when $f$, $f'$ are both surjective or when $f$ is surjective and $A'$ is empty. We will  deal with the first case, the second being similar but simpler (just skip all the steps that involve $A'$ in the argument below).


By  compactness, there is a definable finite partition $(B_i)_i$ of $B$ such that, for each $i$, the objects $\bm U_{b}$,  $\bm U'_{b}$ are defined uniformly over $b \in B_i$. Let $\xi_i(x, y)$, $\xi'_i(x, y')$ be quantifier-free formulas such that $\xi_i(b, y)$, $\xi'_i(b, y')$  define the objects $\bm U_{b}$, $\bm U'_{b}$ for each $b \in B_i$, respectively.

Suppose that $B$ is a set in $\VF$. Applying Lemma~\ref{special:term:constant:disc} to the top terms $F(x)$ of $\xi_i(x, y)$, $\xi'_i(x, y')$  (recall Convention~\ref{topterm}), we obtain a definable bijection $T_i : B_i \fun B_i^\flat$ as described. So, to show the claim, we may replace $B$ with any $B_i^\flat$, and thereby assume that $B$ is an $\RV$-pullback and for each $b \in B$, the objects $\bm U_{b}$, $\bm U'_{b}$ depend on $\rv(b)$ instead of $b$ itself; henceforth $\bm U_{b}$, $\bm U'_{b}$ are denoted by $U_{\rv(b)}$, $U'_{\rv(b)}$ (as above, for ease of notation, we take them to be simply definable sets in $\RV$). In fact, by Theorem~\ref{main:prop:HK}, setting $S = \rv(B)$, we may just assume that $A = \bigcup_{t \in S} \bb L U_{t} \times t^\sharp$ and $f$ is the obvious coordinate projection. Similarly  $A' = \bigcup_{t \in S} \bb L U'_{t} \times t^\sharp$.

Of course, in general, $B$ has $\RV$-coordinates. But the same arrangement can be made by first applying the argument above to each fiber $B_t$, $t \in  \pr_{\RV}(B)$, and then compactness.
Without loss of generality, we just take $A$, $A'$, $B$ to be sets in $\VF$.

Write $U = \rv(A)$, $D = \vv(A)$, $U '= \rv(A')$, $D' = \vv(A')$, and $E = \vv(B)$. Since  $\Gamma$ is stably embedded, upon further partitioning of $S$, we may assume that $\vrv(U_{t})$ depends on $\vrv(t)$ instead of $t \in S$ itself, and hence may be written as $D_{\vrv(t)}$; similarly, we may write $\vrv(U'_{t})$ as $D'_{\vrv(t)}$.

For each $t \in S$, let $(D_{\vrv(t),i})_i$ be a definable finite partition of $D_{\vrv(t)}$ as given by Lemma~\ref{coh:decom}, that is, $(U_{t,i})_i = (U_t \cap D_{\vrv(t),i}^\sharp)_i$ is a bipolar and $\Gamma$-cohesive twistoid decomposition of $U_t$. Again, the notation  $(D_{\vrv(t),i})_i$ makes sense because the partition only depends on $\vrv(t)$; in fact, we may even assume so  for the $\vrv$-contractions of the witness bijections as required by Definition~\ref{bipo:defn}. Then, by compactness and \omin-minimality, we are further reduced to the case that $(D_{\gamma,i})_i$ is uniformly defined over $\gamma \in E$, and $(\chi_b (D_{\gamma,i}))_i$ does not depend on $\gamma \in E$ at all and hence may be simply written as $(e_i)_i$. Let $D_i = \bigcup_{\gamma \in E} D_{i,\gamma} \times \gamma$. So $\chi_b(D_i) = e_i \chi_b(E)$. All this is just a more elaborate version of what we did before (\ref{trop:fub:cons}).

Similarly we obtain a bipolar and $\Gamma$-cohesive twistoid decomposition $(U'_{t,i'})_{i'} = (U'_t \cap {D'}_{\vrv(t),i'}^\sharp)_{i'}$ of $U'_t$, uniformly defined over $\gamma\in E$, such that $e_{i'}=\chi_b (D'_{\gamma,i'})$ do not depend on $\gamma \in E$ and hence $\chi_b(D'_{i'}) = e_{i'} \chi_b(E)$, where $D'_{i'}=\bigcup_{\gamma \in E} D'_{i',\gamma} \times \gamma$.

Applying Lemma~\ref{coh:decom} again to $U$, $U'$ and refining $(D_{i})_i$, $(D'_{i'})_{i'}$ if necessary, we may now assume that $(U_i)_i = (U \cap D^\sharp_i)_i$, $(U'_{i'})_{i'} = (U' \cap {D'}^\sharp_{i'})_{i'}$ are bipolar and $\Gamma$-cohesive twistoid decompositions of $U$, $U'$, respectively. They induce two twistoid decompositions of $S$ (not necessarily bipolar, which we do not need). Refining $(D_{i})_i$, $(D'_{i'})_i'$ one last time, we may assume that both induce the same decomposition of $S$ after all and hence, replacing $S$ with one of the pieces,  $S$ itself is a twistoid.

\begin{rem}
The witness bijection for $U_i$ as a bipolar twistoid is in general unrelated to those for $U_{t, i}$, $t \in S$, since parameterized bipolar twistoids over a base may not even form a single  bipolar twistoid. But our argument will not require this.
\end{rem}

Let $\gamma_i \in \vrv(U_i)$, $\gamma_{i'} \in \vrv(U'_{i'})$ be definable elements such that $\pr_E(\gamma_{i}) = \pr_E(\gamma_{i'})$ for all $i$, $i'$. Set $\dot U_i = U_{i, \gamma_i}$ and $\dot U'_{i'} = U'_{i', \gamma_{i'}}$. So $\pr_S(\dot U_{i}) = \pr_S(\dot U'_{i'})$ for all $i$, $i'$, which we denote by $\dot S$.



Denote by $I$, $I'$ the sets of indices $i$, $i'$. Set $J = I \uplus I'$. For each $j\in J$, let $V_j=\dot U_j$ if $j\in I$ and $V_j=\dot U'_j$ otherwise. Set
\[
J_1=\set{i\in I \given e_i>0} \cup \set{i'\in I' \given e_{i'}<0},
\]
\[
J_2=\set{i\in I \given e_i<0}\cup \set{i'\in I'  \given e_{i'}>0},
\]
and for $j\in J$, $\tilde e_j = \abs{e_j}$.

By (\ref{Eb:twist:com}), we have $\bb E_b (\int [A] ) = \chi_b(E) \sum_{i\in I} e_i [\dot U_i]$ and $\bb E_b ( \int [A'] ) = \chi_b(E) \sum_{i'\in I'} e_{i'} [\dot U'_{i'}]$. Thus, to prove the theorem, it is enough to show that $\sum_{j\in J_1} \tilde e_j [V_j]=\sum_{j\in J_2} \tilde e_j [V_j]$ in $\sggk \RES$.  Since both sides can be understood in $\gsk \RES$, it is enough to show that the following condition holds.

\begin{test}\label{test}
There are objects $W_1, W_2, X \in \RES$ such that
\begin{itemize}
  \item $[W_1] = \sum_{j\in J_1} \tilde e_j [V_j] + [X]$ and $[W_2] = \sum_{j\in J_2} \tilde e_j [V_j] + [X]$ in $\gsk \RES$,
  \item $([W_1], [W_2])$ is in the semiring congruence relation $\mdl E_{sc}$ on $\gsk \RES$ generated by pairs of the form $([\gamma^\sharp \times Z], [\G_{m}^n \times Z])$, where $Z \in \RES$ and $\gamma \in \Gamma^n$ is definable (we allow $n = 0$, which gives the pairs $(Z, Z)$).
\end{itemize}
\end{test}

What the second item means is that  if we close off the set $\mdl A$ of these pairs $([\gamma^\sharp \times Z], [\G_{m}^n \times Z])$ by applying inductively the six conditions, namely identity (since each stage needs to be accumulative), the three equivalence conditions (symmetry, reflexivity, and transitivity), and the two semiring congruence conditions (if $([X], [Y])$ is already in at some stage then, for any $[Z] \in \gsk \RES$, $([X \uplus Z], [Y \uplus Z])$ and  $([X \times Z], [Y \times Z])$ are thrown in at the next stage), then $([W_1], [W_2])$ appears at some stage. Actually it is clear that reflexivity is redundant. We show that, beyond $\mdl A$, product is also redundant, as follows.

\begin{rem}
There are in general more than one way to inductively construct an element $([X], [Y])$ in $\mdl E_{sc}$. Each inductive construction of $([X], [Y])$ may be represented by a finite graph $T$, partitioned into levels (numbered inversely), such that
\begin{itemize}
  \item every vertex is an element in $\mdl E_{sc}$, in particular, there is only one vertex in the top level, which is $([X], [Y])$ itself, and the bottom vertices are all elements in $\mdl A$,
  \item each vertex above level $0$ has $1$ or $2$ (the latter only for transitivity) labeled edges connecting it with vertices in the level below, with the labels indicating how it is constructed from these latter vertices,
  \item each vertex below the top level is connected via an edge with a vertex in the level above (this is for efficiency, that is, no redundant vertices for the construction).
\end{itemize}
We call $T$ a \memph{construction graph} of $([X], [Y])$, and the number of levels in it its \memph{height}. Observe that, for any construction graph $T$ of $([X], [Y])$ and any $[Z] \in \gsk \RES$, if we replace every vertex $([X'], [Y'])$ of $T$ with $([X' \times Z], [Y' \times Z])$ then the resulting graph is a construction graph for $([X \times Z], [Y \times Z])$. Let $\rho([X], [Y])$ be the least natural number such that there is a construction graph for $([X], [Y])$ of height $\rho([X], [Y])$; let $T$ be such a construction graph. We claim that $T$ may be chosen so that  none of its edges is a product. This is vacuous if $\rho([X], [Y]) = 0$ and is clear if $\rho([X], [Y]) = 1$. For $\rho([X], [Y]) > 1$, note that if $([X'], [Y'])$, $([X''], [Y''])$ are the (possibly identical) vertices in level $\rho([X], [Y]) - 1$ of $T$ then
\[
\max(\rho([X'], [Y']), \rho([X''], [Y''])) = \rho([X], [Y]) - 1.
\]
Also, if the edge going into $([X], [Y])$ is a product then, by what we have just observed, there would be a construction graph of $([X], [Y])$ of height $\rho([X], [Y]) - 1$, contradicting the choice of $\rho([X], [Y])$. So the desired conclusion follows from a routine induction on $\rho([X], [Y])$.
\end{rem}

We need this seemingly obscure fact so as to be able to apply compactness uniformly over a base below.

\begin{rem}
The semiring congruence relation $\mdl E_{sc}$ may also be constructed ``without the brackets,'' that is, it may be  expressed as an equivalence relation $\mdl E^*_{sc}$ on the objects of $\RES$ that is generated from the set of definable pairs of the form $(\gamma^\sharp \times Z, \G^n_{m} \times Z')$, where $Z \cong Z'$ in $\RES$, via the six conditions. However, if $T$ is a construction graph of  $([X], [Y])$ then simply deleting the brackets from the vertices of $T$ does not give us a graph that witnesses $(X, Y) \in \mdl E^*_{sc}$. There are two issues. The first is transitivity, that is, if $([X], [Y])$ and $([Y'], [Z])$ are in  $\mdl E_{sc}$ with $[Y] = [Y']$ then of course so is $([X], [Z])$, but to conclude $(X, Z) \in \mdl E^*_{sc}$ from $(X, Y), (Y', Z) \in  \mdl E^*_{sc}$ we must add a middle step $(Y, Y')$, justified  by $Y \cong Y'$ in $\RES$, and apply transitivity twice. The second is similar: we can now translate $T$ into a graph that witnesses $(X', Y') \in \mdl E^*_{sc}$, but only with $[X'] = [X]$ and $[Y'] = [Y]$. But then, adding at most two levels on top, we arrive at a \memph{construction graph} $T^*$ of $(X, Y) \in \mdl E^*_{sc}$. Moreover, if $T$ is product-free as given above then $T^*$ is also  product-free, in the obvious sense.
\end{rem}
The reason that we  opt for working with $T^*$ instead of $T$ is that it can be described by  an $\lan{}{RV}{}$-formula and thereby makes itself amenable to compactness.

Finally, by the assumption, we have $\sum_{j\in J_1} \tilde e_j [V_{j,t}] =  \sum_{j\in J_2} \tilde e_j [V_{j,t}] $ in $\sggk \RES$ for every $t \in \dot S$. For each $t \in \dot S$, we can choose $t$-definable objects $W_{1,t}$, $W_{2,t}$, $X_t$ of $\RES$ with respect to which Test~\ref{test} holds, as witnessed by a product-free construction graph $T^*_t$. By compactness, there is a definable finite partition $(\dot S_i)_i$ of $\dot S$ such that, for each $i$, the graph $T^*_t$ is uniformly defined over $t \in \dot S_i$ by an $\lan{}{RV}{}$-formula; without loss of generality, we may assume this partition is trivial. Moreover, if level $0$ of $T^*_t$ consists of definable pairs of the forms $(\gamma^\sharp \times Z_t, \G^n_{m} \times Z'_t)$ then $Z_t$, $Z'_t$ in general depend on $t$, but $\gamma^\sharp$ does not, and neither do other $\Gamma$-data over $t$, such as $\vrv(Z_t)$, $\vrv(\dot U'_t)$, and $\vrv(V_t)$ (of course they do depend on $\vrv(\dot S)$). Let
\[
W_1 = \bigcup_{t \in \dot S} W_{1,t}, \quad W_2 = \bigcup_{t \in \dot S} W_{2,t}, \quad X = \bigcup_{t \in \dot S} X_t.
\]
Thus, and this is why we have taken effort to find product-free construction graphs above, the union $\bigcup_{t \in \dot S} T^*_t$ is indeed a construction graph of $(W_1, W_2)$. So Test~\ref{test} holds with respect to $W_1$, $W_2$, and $X$. So $\bb E_b (\int [A]) = \bb E_b ( \int [A'])$.
\end{proof}

\begin{cor}
In $\puC$ with $\bb S = \C \dpar{ t }$, if $f$ is a morphism of varieties over $\C$ with $\Vol_b([A_b]) = 0$ for all $b \in B$ then $\Vol([A]) = 0$.
\end{cor}

\section{Special covariant bijections}\label{sec-pscb}

From here on, we assume that the substructure $\mathbb{S}$ of parameters is $\VF$-generated (but $\Gamma(\mathbb{S})$ could be trivial, in which case the only definable element in $\Gamma$ is $0$, and hence $A_\Gamma = 0$ for any doubly bounded definable set $A$). So every definable disc contains a definable point (Lemma~\ref{algebraic:balls:definable:centers}). Another simple fact we shall occasionally use is that if $\gamma \in \Gamma$ is definable then there is a $\gamma' \geq \gamma$ such that $\gamma'^\sharp$ contains a definable point.

\subsection{Invariance and covariance}\label{sec:inv:cov}

\begin{defn}\label{defn:binv}
For $\gamma \in \Gamma$, let $\pi_\gamma: \VF \fun \VF / \MM_{\gamma}$ be the quotient  map. If $\gamma = (\gamma_1, \ldots, \gamma_n) \in \Gamma^n$ then $\pi_{\gamma}$ denotes the product of the maps $\pi_{\gamma_i}$.

Let $\alpha \in \Gamma^n$ and $\beta \in \Gamma^m$. We say that a function $f :A \fun B$ with $A \sub \VF^n$ and $B \sub \VF^m$  is \memph{$(\alpha, \beta)$-covariant} if it $(\pi_{\alpha}, \pi_{\beta})$-contracts to a function $f_\downarrow : \pi_{\alpha}(A) \fun \pi_{\beta}(B)$, that is, $\pi_{\beta} \circ f = f_\downarrow \circ \pi_{\alpha}$. For simplicity, we shall often suppress parameters in notation and refer to $(\alpha, \beta)$-covariant functions as $(\alpha, -)$-covariant or $(-,\beta)$-covariant or just  covariant functions.

A set  in $\VF$ is \memph{$\alpha$-invariant} if its characteristic function is $(\alpha, 0)$-covariant.

More generally, for sets $A$, $B$ with $\RV$-coordinates, a function $f : A \fun B$ is \memph{covariant} if every one of its $\VF$-fibers    is covariant (recall Terminology~\ref{rvfiber}). Similarly, a set is \memph{invariant} if every one of its $\VF$-fibers is invariant.
\end{defn}

\begin{defn}\label{prop:pseu}
A definable function $f : A \fun B$ is \memph{proper covariant} if
\begin{itemize}
  \item the sets $A_{\VF}$, $f(A)_{\VF}$ are bounded and the sets $A_{\RV}$, $f(A)_{\RV}$ are  doubly bounded,
  \item for each $\VF$-fiber $f_{t}$ of $f$ there is a $t$-definable tuple $(\alpha_t, \beta_t) \in \Gamma$ such that $f_t$ is $(\alpha_{t}, \beta_{t})$-covariant, $\dom(f_{t})$ is $\alpha_t$-invariant, and $\ran(f_{t})$ is $\beta_t$-invariant.
\end{itemize}
A definable set $A$ is \memph{proper invariant}   if $A_{\VF}$ is bounded, $A_{\RV}$ is doubly bounded, and for each $\VF$-fiber $A_t$ of $A$ there is a $t$-definable tuple $\alpha_t\in \Gamma$ such that $A_t$ is $\alpha_t$-invariant.
\end{defn}

Note that covariance makes sense even when the domain or the range of the function $f$ has no $\VF$-coordinates, in which case the contraction $f_\downarrow$ would be the empty function. In particular, if $B$  is a subset of $\RV^m$ or even $\mdl P(\RV^m)$ then $f$ being proper covariant simply means that every $f^{-1}(x)$ is proper invariant. Then $A$ being proper invariant may also be construed as the condition that its projection into the $\RV$-coordinates is proper covariant.

\begin{exam}
A proper invariant set  is not required to be doubly bounded. However, if $\bm U \in \RV^{\db}[*]$ then $\bb L \bm U$ is indeed a doubly bounded proper invariant set.
\end{exam}

\begin{rem}
In defining proper covariance above, if  $A$ is closed then  there is no need to demand that $f(A)_{\RV}$ be doubly bounded, because it is implied by the other conditions, as follows. There is an $\go$-partition $p: A \fun \Gamma$ such that, for each open polydisc $\gb$ in question, that is, for each open polydisc $\gb \sub A$ of the form $\go(a, p(a))$, $f(\gb)$ lies in the range of some $\VF$-fiber of $f$. By Lemma~\ref{vol:par:bounded}, $p(A)$ is doubly bounded. So the definable surjection $A \fun \vrv(f(A)_{\RV})$ induced by $f$ satisfies the assumption of Lemma~\ref{invar:db} below. So  $f(A)_{\RV}$ is doubly bounded.
\end{rem}

\begin{lem}\label{inv:crit}
Let $A$ be a definable set such that $A_{\VF}$ is bounded and $A_{\RV}$ is doubly bounded. Then $A$ is proper invariant if and only if $A$ is clopen.
\end{lem}
\begin{proof}
The ``only if'' direction is clear. For the ``if'' direction, since every $\VF$-fiber of $A$ is open, there is an obvious  $\go$-partition $p: A \fun \Gamma$. Since every $\VF$-fiber of $A$ is also closed, by Lemma~\ref{vol:par:bounded}, $p(A)$ is doubly bounded. The claim follows.
\end{proof}

\begin{lem}\label{int:inv}
If $A$, $B$ are proper invariant subsets of $\VF^n \times \RV^m$ then $A \cap B$, $\RVH(A) \mi A$, and $B \mi A$ are also proper invariant.
\end{lem}
\begin{proof}
Clearly $(A \cap B)_{\VF}$ is bounded, $(A \cap B)_{\RV}$ is doubly bounded, and every $\VF$-fiber of $A \cap B$ is clopen; similarly for the other cases. So this follows from Lemma~\ref{inv:crit}.
\end{proof}

\begin{lem}\label{conti:covar}
Let $f : A \fun B$ be a definable continuous surjection between two proper invariant sets such that  every $\dom(f_t)$ is open. If every $\ran(f_t)$ is clopen then $f$ is proper covariant. If $f$ is bijective then the same holds under the weaker assumption that   every $\ran(f_t)$ is open.
\end{lem}
\begin{proof}
For every $(r, s) \in f_{\RV}$ and every $b \in \ran(f_{r,s})$, let
\[
q(b,r,s) = \min\set{\beta \in \Gamma \given \go((b,s), \beta) \sub \ran(f_{r,s})}.
\]
The definable function $q$ is clearly an $\go$-partition of its domain, which, by the assumption, is closed. By Lemma~\ref{vol:par:bounded}, $\ran(q)$ is doubly bounded by, say, $\beta \in \Gamma^+$. By continuity, there is another $\go$-partition $p: A \fun \Gamma$ such that, for every open polydisc $\gb \sub A$ in question, $f(\gb)$ lies in an open polydisc of radius $\beta$. By Lemmas~\ref{inv:crit} and \ref{vol:par:bounded} again,  $p(A)$ is doubly bounded as well. It follows that $f$ is  proper covariant.

For the second claim, since $f$ is bijective, for each $(b, s) \in B$ there is only one  $f_{r,s}$ with $(b,r, s) \in \ran(f_{r,s})$. So the parameter $r$ is redundant in the definition of $q$, and  $q$ is  simply an $\go$-partition of $B$. The rest of the argument is the same.
\end{proof}

\begin{lem}\label{conv:unif}
If $f : A \fun B$ is a proper covariant surjection then both $A$ and $B$ are proper invariant. If, in addition, $f$ is continuous and bijective then $f$ is a proper covariant homeomorphism, which just means that $f^{-1}$ is also proper covariant and continuous. Moreover, in that case, if $A'$ is a proper invariant subset of $A$ then $f \rest A'$ is a proper covariant homeomorphism as well.
\end{lem}
\begin{proof}
Since $f_{\RV}$ is  doubly bounded, by Lemma~\ref{db:to:db}, for all sufficiently large  $\epsilon \in \Gamma$, every $\VF$-fiber $f_{t}$ of $f$ is $(\epsilon, \beta_{t})$-covariant (but not necessarily $(\epsilon, \epsilon)$-covariant) and $\dom(f_{t})$,  $\ran(f_{t})$ are both $\epsilon$-invariant. The first claim follows.

For the second claim, since $\dom(f_{t})$ is clopen, by  Corollary~\ref{conti:homeo}, every $f_t$ is a homeomorphism. Since  $\ran(f_{t})$ is open, it follows from the continuity of $f^{-1}$, as in the proof of Lemma~\ref{conti:covar}, that every  $f^{-1}_t$ is  $(\delta, \epsilon)$-covariant  for all sufficiently large $\delta > \epsilon$.

For the third claim, clearly every $\dom(f_t) \cap A'$ is   clopen and, since $f_t$ is a homeomorphism, every $\ran(f_t) \cap f(A')$ is  clopen too. This is sufficient for the proof of (the first claim of) Lemma~\ref{conti:covar} to go through. By the second claim, then, $f \rest A'$ is a proper covariant homeomorphism.
\end{proof}

\begin{lem}\label{pr:co:home:com}
If $f : A \fun B$ and $g : B  \fun C$ are proper covariant homeomorphisms then $g \circ f$ is also a proper covariant homeomorphism.
\end{lem}
\begin{proof}
By Lemma~\ref{conv:unif}, we just need to show that $g \circ f$ is proper covariant. Then, by Lemma~\ref{conti:covar}, we are reduced to showing that the domain and range of every $\VF$-fiber of $g \circ f$ are open. Arguing as in the proofs of Lemmas~\ref{conti:covar} and \ref{conv:unif}, we see that there are $\epsilon, \epsilon', \epsilon'' \in \Gamma$, each one  sufficiently larger than the next, such that every $\VF$-fiber of $f$ is $(\epsilon, \epsilon')$-covariant and every $\VF$-fiber of $g$ is $(\epsilon', \epsilon'')$-covariant. It follows that the domain of every $\VF$-fiber of $g \circ f$ is $\epsilon$-invariant and hence is clopen. Since the same holds for $(g \circ f)^{-1}$, the range of every $\VF$-fiber of $g \circ f$ is clopen as well.
\end{proof}

\begin{conv}\label{conv:can}
The trivial-looking operation we describe here is actually quite crucial for understanding the discussion below, especially those passages that involve special covariant bijections (defined below). For each definable set $A$, let
\[
\can(A) = \set{(a, \rv(a), t) \given (a, t) \in A \tand a \in A_{\VF}}.
\]
The obvious definable bijection $\can : A \fun \can(A)$ is  referred to as the \memph{regularization} of $A$.

Observe that, for a proper invariant set $A$, $\can(A)$ is proper invariant if and only if $A_{\VF}$ is  doubly bounded. Going forward, when $A$ is doubly bounded and proper invariant, it is often more convenient or even necessary to substitute $\can(A)$ for $A$. Whether such a substitution has been performed (tacitly) or not should be clear in the context (or rather, it is always performed).
\end{conv}

\subsection{Proper covariant bijections that are special}

\begin{defn}\label{defn:special:bijection}
Let $A$ be a doubly bounded proper invariant set  or, more generally, a finite disjoint union of such sets,  regularized in the sense of  Convention~\ref{conv:can}. Suppose that the first coordinate of $A$ is a $\VF$-coordinate (of course nothing is special about the first $\VF$-coordinate, we choose it simply for ease of notation).

Let $C \sub \RVH(A)$ be an $\RV$-pullback. Let
\[
\lambda: \pr_{>1}(C \cap A) \fun \VF
\]
be a definable continuous function whose graph is contained in $C$, that is, for each $\RV$-polydisc $\gp \sub C$, $\lambda$ restricts to a function
\[
\lambda_{\gp} : \pr_{>1}(\gp \cap A) \fun \pr_{1}(\gp).
\]
Suppose that there is an $\rv(\gp)$-definable tuple $(\alpha_{\gp}, \beta_{\gp}) \in \Gamma$ such that $\gp \cap A$ is $(\alpha_{\gp}, \beta_{\gp})$-invariant and $\lambda_{\gp}$ is  $(\beta_{\gp}, \alpha_{\gp})$-covariant. Choose a definable $\gamma \in \Gamma^+_0$  such that there is a definable point $t \in \gamma^\sharp$ and for all $\gp$,
\[
\gamma_{\gp} \coloneqq \rad(\pr_{1}(\gp)) + \gamma \geq \alpha_{\gp};
\]
the existence of such a $\gamma$ is guaranteed by Lemma~\ref{db:to:db} and the assumption that $\mathbb{S}$ is $\VF$-generated. For each $\gp$,  set
\[
t_{\gp} = t \cdot \rv(\pr_{1}(\gp)) \in \RV.
\]
Then the \memph{centripetal transformation $\eta$ on $A$ with respect to $\lambda$} is given by
\begin{equation*}\tag{CT}\label{centri}
\begin{cases}
  \eta (a, x) = (a - \lambda(x), x), & \text{if } (a, x) \in \gp \cap A \text{ and } \pi_{\gamma_{\gp}}(a) \neq \pi_{\gamma_{\gp}}(\lambda(x)),\\
  \eta (a, x) = (t_\gp, x), & \text{if } (a, x) \in \gp \cap A \text{ and } \pi_{\gamma_{\gp}}(a) = \pi_{\gamma_{\gp}}(\lambda(x)),\\
  \eta = \id, & \text{on } A \mi C.
\end{cases}
\end{equation*}
Observe that $\eta(A)$ is doubly bounded. The function $\lambda$ is called the \memph{focus} of $\eta$, the $\RV$-pullback $C$  the \memph{locus} of $\lambda$ (or $\eta$), and the pair $(\gamma, t)$   the \memph{aperture} of $\lambda$ (or $\eta$). Note that if $(\gamma, t)$ is the aperture of $\lambda$ then every other pair $(\gamma', t')$
of this form with $\gamma' \geq \gamma$ could be an aperture of $\lambda$ as well, so the aperture of $\lambda$ must be given as a part of $\lambda$ itself. Actually, all the data above should be regarded as part of $\lambda$, including the tuples $(\alpha_{\gp}, \beta_{\gp})$.

A \memph{special covariant transformation} $T$ on $A$ is an alternating composition of centripetal transformations and regularizations. The \memph{length} of such a special covariant transformation $T$, denoted by $\lh(T)$, is the number of centripetal transformations in it.

Choose a definable point $c \in t^\sharp$ (again, this is possible since $\mathbb{S}$ is assumed to be $\VF$-generated). If $(a, x) \in \gp \cap A$ and $\pi_{\gamma_{\gp}}(a) = \pi_{\gamma_{\gp}}(\lambda(x))$  then $(\lambda(x), x) \in \gp \cap A$. Thus  the second clause of (\ref{centri}) may be changed to
\[
\eta (a, x) = (a - \lambda(x) + \lambda(x)c, x).
\]
If so, the images from the first two clauses of (\ref{centri}) may  overlap, but we take their disjoint union and thereby always assume that the resulting function $\eta^\flat$ is injective. In so doing, a special covariant transformation may be \memph{lifted} to a \memph{special covariant bijection} $T^\flat$ on $A$. This of course does depend on the choice of the definable point $c \in t^\sharp$.
The image of $T^\flat$ is often denoted by $A^{\flat}$.
\end{defn}

The formulation of a special covariant bijection here is considerably more complicated than  that of a special bijection in \cite[Definition~5.1]{Yin:special:trans}. A special covariant bijection is a special bijection if all of the apertures are $(\infty, 0)$ in the obvious sense or the second clause of (\ref{centri}) does not occur. The extra generality is needed to achieve better control in transforming one proper invariant set into another.

\begin{rem}\label{bigger:aper}
If we change the aperture $(\gamma, t)$ of $\eta$ to $(\gamma', t')$ with $\gamma' > \gamma$ and write the resulting  centripetal transformation as $\eta'$ then  every open polydisc $\gb \sub \eta(A) \mi \eta'(A)$ of radius $\alpha_\gp$, where $\gp$ is the $\RV$-polydisc that contains $\eta^{-1}(\gb)$, has an extra $\RV$-coordinate $t_\gp$ that is contracted from open discs of radius $\gamma_\gp$ in the first $\VF$-coordinate of $A$. Each of these open polydiscs has a counterpart in $\eta'(A) \mi \eta(A)$, in which $t_\gp$ is replaced by  $(\MM_{\gamma_\gp} \mi \MM_{\gamma'_\gp}) \cup t'_{\gp}$. On the other hand, $\eta(A) = \eta'(A)$ if and only if the second clause of (\ref{centri}) does not occur.
\end{rem}

\begin{rem}\label{spec:addendum}
Suppose that $A$ is proper $\delta$-invariant but is not doubly bounded. Then, by definition, there  are no special covariant transformations on $A$. We can prepare $A$, though, as follows. Choose a definable point $t_\delta \in \delta^\sharp$. For each open polydisc $\gb_1 \times \ldots \times \gb_n \times t \sub A$ of radius $\delta$, if $\gb_i= \MM_\delta$ then replace it with $t_\delta$. The resulting definable set, or rather, the resulting finite disjoint union of definable sets, is doubly bounded and remains $\delta$-invariant. We may and do regard this operation as a centripetal transformation with respect to the constant focus map $0$.
\end{rem}

\begin{lem}\label{spec:inve}
Special covariant bijections  are  partially differentiable proper covariant homeomorphisms with the constant Jacobian $1$.
\end{lem}
\begin{proof}
In the notation of Definition~\ref{defn:special:bijection}, since $\pr_{>1}(C \cap A)$ is clopen (implied by proper invariance and Lemma~\ref{inv:crit}) and  $\lambda$ is  continuous, arguing as in the proof of Lemma~\ref{conti:covar}, we see that for all sufficiently large $\epsilon \in \Gamma$ there is an even larger $\delta \in \Gamma$ such that the image under $\lambda$ of any open polydisc of radius $\delta$ contained in $\pr_{>1}(C \cap A)$ is contained in an open disc of radius $\epsilon$. Thus, for all $\RV$-polydisc $\gp \sub C$, the image under the lifted centripetal transformation $\eta^\flat$ of any open polydisc of radius $(\gamma_{\gp}, \delta)$ contained in $\gp  \cap A$ is again an open polydisc of radius $(\gamma_{\gp}, \delta)$. This means that $\eta^\flat(A)$ is proper invariant and, moreover, the domain and range of every $\VF$-fiber of $\eta^\flat$ are clopen. More generally, since the $\RV$-coordinates are carried along by centripetal transformations, it is not hard to see that this actually holds  for all  special covariant bijections. Now the claim follows from  Lemmas~\ref{conti:covar} and \ref{conv:unif} (that the Jacobian is $1$ everywhere is obvious).
\end{proof}

So calling them ``special proper covariant bijections'' may seem more appropriate, but the verbosity is somewhat confusing.

\begin{rem}\label{spec:is:iso}
Let $(A,\omega)$, $(B, \sigma)$ be objects of $\mgVF[k]$ and $F : A \fun B$ a special covariant bijection. Then  $F$ represents a $\mgVF[k]$-isomorphism if $\omega(x) = \sigma(F(x))$ for all $x \in A$ outside a definable subset of $\VF$-dimension less than $k$ (recall Remark~\ref{mor:equi}).
\end{rem}

\begin{lem}[{\cite[Theorem~5.5]{Yin:special:trans}}]\label{special:bi:term:constant}
Suppose that $A$ is a definable set in $\VF$ and $f : A \fun \mdl P(\RV^m)$ is a definable function. Then there is a special bijection $T$ on $A$ such that $A^\flat$ is an $\RV$-pullback and the function $f \circ T^{-1}$ is $\rv$-contractible.
\end{lem}

\begin{rem}\label{special:pseudo}
The proof of Lemma~\ref{special:bi:term:constant} in \cite{Yin:special:trans} actually shows that for every definable set $A$, there is a special bijection $T$ on  $\RVH(A)$  such that $A^{\flat}$ is an $\RV$-pullback. In the present context, for reasons that will become clear, we would like to extend this result, using special covariant bijections  on proper invariant sets $A$. This is not guaranteed by Lemma~\ref{special:bi:term:constant} since the focus maps in $T$ are not required to be (suitably) covariant within each $\RV$-polydisc; except when $A_{\VF} \sub \VF$, in which case  the covariance requirement is half vacuous and it is easy to see how to turn $T$ into a special covariant bijection whose components all have the same aperture. The technique underlying such a modification could be generalized so as to be applicable in other similar situations, see  Terminology~\ref{FMT} below.
\end{rem}

\begin{ter}\label{FMT}
Often in a construction there is obstructing data that is contained in finitely many distinct definable open discs $\ga_i$, all of the same radius, and is presented in the form of a definable function $f : A \fun \VF$ with $A$ a set in $\VF$ and $\ga_i \sub \pr_1(A)$. Then the \memph{FMT procedure} may be applied to circumvent the difficulties, which is simply this: choose a definable point $a_i$ in each $\ga_i$ and replace $f$ with another function $f' : A \fun \VF$ given by $(a,b) \efun f(a_i, b)$ if $a \in \ga_i$ and $(a,b) \efun f(a,b)$ otherwise, now the construction over each $\ga_i$ depends on $a_i$ and hence no longer varies.
\end{ter}

\begin{prop}\label{simplex:with:hole:rvproduct}
Let $H$ be a proper invariant $\RV$-pullback, not necessarily doubly bounded, and $(A_i)_i$ a definable finite partition of $H$ such that each $A_i$ is proper invariant. Then there is a special covariant bijection $T$ on $H$ such that
\begin{itemize}
  \item every $T \rest A_i$ is indeed a special covariant bijection,
  \item every $A_i^{\flat} \sub H^{\flat}$  is a doubly bounded $\RV$-pullback.
\end{itemize}
\end{prop}
\begin{proof}
To begin with, by Remark~\ref{spec:addendum}, we may assume that $H$ is doubly bounded. It is equivalent and less cumbersome to construct a special covariant transformation such that its lift is as required (see the last paragraph of this proof). To that end, we proceed by induction on $n$, where $H_{\VF} \sub \VF^n$. The base case $n=1$ follows from Lemma~\ref{special:bi:term:constant} and the discussion in Remark~\ref{special:pseudo}.

For the inductive step, let $A_{i1} = \pr_1(A_i)$ and $H_1 = \pr_1(H)$. By Lemma~\ref{int:inv}, we may assume that $(A_{i1})_i$ form a partition of $H_1$. For each $a \in H_1$, the inductive hypothesis gives an $a$-definable special covariant transformation $T_a$ on $H_a$ such that every $T_a \rest A_{ia}$ is a special covariant transformation and every $T_a(A_{ia}) \sub T_a(H_a)$ is an $\RV$-pullback. Our goal then is to fuse together these transformations $T_a$ so to obtain one special covariant transformation on $H$ as desired. This is in general not possible without first modifying $H_1$ in a suitable way, which constitutes the bulk of the work below.

Let $U_{ak}$ enumerate the loci of the components of $T_{a}$, $\lambda_{ak}$ the corresponding continuous focus maps, and $(\gamma_{ak}, t_{ak})$ their apertures; for each $\RV$-polydisc $\gq \sub U_{ak}$, the map $\lambda_{ak\gq}$ is $(\alpha_{ak\gq}, \beta_{ak\gq})$-covariant. By compactness, there is a definable set $V \sub \VF \times \RV^l$ such that $\pr_1(V) = H_1$ and, for each $a \in H_1$, the set $V_a$ contains the following $\RV$-data of $T_a$:
\begin{itemize}
  \item $\rv(T_{a} \rest A_{ia})$, $\rv(T_{a})$, and $\rv(U_{ak})$,
  \item the $\VF$-coordinates targeted by the focus maps $\lambda_{ak}$,
  \item the $a$-definable apertures $(\gamma_{ak}, t_{ak})$,
   \item the $(a, \rv(\gq))$-definable tuples $(\alpha_{ak\gq}, \beta_{ak\gq})$;
\end{itemize}
the set  $\rv(T_{a})$ is determined by other data in this list and hence is redundant, but we add it in anyway for clarity. Note that $V$ is  doubly bounded but not necessarily proper invariant.

Let $\phi(x, y)$ be a quantifier-free formula that defines $V$ and $G_{i}(x)$ enumerate its top terms. By Lemma~\ref{special:bi:term:constant}, there is a special bijection $R : H_1 \fun H_1^\flat$ such that each $A_{i1}^\flat \sub H_1^\flat$ is an $\RV$-pullback and every function $G_{i} \circ R^{-1}$ is $\rv$-contractible. This implies that,  for every $\RV$-polydisc $\gp \sub H_1^\flat$, the $\RV$-data $V_a$ is constant over $a \in R^{-1}(\gp)$. Also, since $H_1$ is an $\RV$-pullback of $\RV$-fiber dimension $0$, by \cmin-minimality and Lemma~\ref{RV:fiber:dim:same}, each focus map in $R$ consists of only finitely many points.

Fix a sufficiently large definable $\delta \in \Gamma$ such that every $A_i$ is proper $\delta$-invariant.  Then there is an algebraic set of open discs $\ga_j \sub H_1$ of radius $\delta$ such that the restriction of $R$ to $H_1 \mi \bigcup_j \ga_j$ is actually a special covariant bijection --- the reason simply being that, after deleting all the discs $\ga_j$, each focus map in $R$ lies outside the set in question. Each $\ga_j$ contains an $\code {\ga_j}$-definable point $a_j$ and, for all $a, a' \in \ga_j$ and every $A_i$, we have $A_{ia} = A_{ia'}$ (because  $A_i$ is $\delta$-invariant). It follows that, over each $\ga_j$, we can use the same  transformation $T_{a_j}$ to achieve the desired effect. This is the first instance where FMT is applied, namely to $T$, and then $R$ is adjusted so that every $\ga_j$ is mapped to the same $\RV$-disc $t^\sharp_{\delta}$, where $t_{\delta} \in \delta^\sharp$ is a chosen definable point.


Therefore, we may assume that $R$ is a special covariant bijection whose components all have the same  aperture $(\delta, t_{\delta})$.

By compactness, there is  a definable finite partition of $H_1$ such that, over each piece, the  focus maps $\lambda_{ak}$ are uniformly defined by formulas $\lambda_k(a, y, z)$. By \cmin-minimality, there are only finitely many open discs $\ga_i \sub H_1$ of radius $\delta$ that are split by this partition. Thus, by FMT again, we may assume that the partition is indeed trivial. Since $R$ induces a special covariant bijection on $H$, we may actually assume that $R$ is trivial as well.

Over each $t^\sharp \sub H_1$, we can now write $U_{ak}$ as $U_{tk}$, $\alpha_{ak\gq}$ as $\alpha_{k\gq}$ (the first $\RV$-coordinate of $\gq$ is actually $t$), and so on. We are almost ready to fuse together the transformations $T_a$ over $a \in t^\sharp$. The remaining problem  is that, for any $a, a' \in t^\sharp$, although the two focus maps $\lambda_{a1\gq}$, $\lambda_{a'1\gq}$ are both $(\alpha_{1\gq}, \beta_{1\gq})$-covariant,  the images of the same open polydisc of radius $\alpha_{1\gq}$ may lie in two distinct open discs of radius $\beta_{1\gq}$. To solve this problem, consider an open polydisc $\gp$ of radius $\alpha_{1\gq}$ that is contained in $\dom(\lambda_{a1\gq})$ for some (hence all) $a \in t^\sharp$. For each $b \in \gp$, let $\lambda_{1b} : t^\sharp \fun \VF$ be the function defined by $\lambda_1(x, b, z)$. By \cmin-minimality and Corollary~\ref{part:rv:cons}, there are a $b$-definable finite set $C_b \sub \VF$ and, for any $a \in t^\sharp \mi C_b$, an open disc $\ga_a \sub t^\sharp \mi C_b$ around $a$ such that $\lambda_{1b}(\ga_a)$ lies in an open disc of radius $\beta_{1\gq}$. Since for any other $a' \in \ga_a$,  $\lambda_{a'1\gq}(\gp)$ also lies in an open disc of radius $\beta_{1\gq}$, we see that $\lambda_1(\ga_a \times \gp)$ lies in an open disc of radius $\beta_{1\gq}$, where $\lambda_1$ stands for the function defined by $\lambda_1(x, y, z)$. Therefore, we may assume that the finite set $C_b$ is actually $\code \gp$-definable. But then, by Lemma~\ref{ball:to:ac}, it is even definable, that is, it does not depend on $\gp$ (or rather it works for all $\gp$). By compactness and FMT, we may assume that there is an $\go$-partition $p: \dom(\lambda_1) \fun \Gamma$ such that, for each open polydisc $\gb$ in question,  $\lambda_1(\gb)$ lies in a disc of radius $\beta_{1\gq}$, where $\gq$ is related to $\gb$ in the obvious way. By Lemma~\ref{vol:par:bounded}, the image of $p$ is doubly bounded and, by Lemma~\ref{db:to:db}, there is a definable $\gamma_1 \in \Gamma$ with $\gamma_1 \geq \gamma_{t1}$ for all $t \in \rv(H_1)$. We may choose such a $\gamma_1$ together with a definable point $t_1 \in \gamma^\sharp_1$ (see the opening paragraph of this section). All this means that $\lambda_1$ can serve as the focus map of a centripetal transformation $T_1$ on $H$ of aperture $(\gamma_{1}, t_{1})$. To ensure that $\lambda_1$ is continuous, we can appeal to Proposition~\ref{prop:continuity:fiberwise} and, once again, FMT.

At this point the proof would be complete if we could repeat the procedure above for $\lambda_2$, and so on. We still have a small issue, namely some part of the locus $U_{a2}$ may have disappeared because the aperture of  $\lambda_{a1}$ is bumped up to $(\gamma_{1}, t_{1})$; see Remark~\ref{bigger:aper}. It is not hard to see that the inductive hypothesis may be applied to the $\RV$-pullback $W$ contained in $T_1(H)$ that corresponds to the missing locus, ignoring the coordinate of $W$ that was targeted by $T_1$, because the missing locus must have one less $\VF$-coordinate and its intersection with each $T_1(A_i)$ is proper invariant.
\end{proof}

This proposition can indeed be applied to any proper invariant set $A$ against its $\RV$-hull $\RVH(A)$, without specifying a partition of $\RVH(A)$ as required, because, by  Lemma~\ref{int:inv},   $\RVH(A) \mi A$ is proper invariant too.

\begin{rem}\label{cons:extra}
As in the statement of Lemma~\ref{special:bi:term:constant} above, if $f : H \fun \mdl P(\RV^m)$ is a proper covariant function then we can demand as a clause of the conclusion of Proposition~\ref{simplex:with:hole:rvproduct} that the function $f \circ T^{-1}$ be $\rv$-contractible, which would follow if it is simply included as a clause of the inductive hypothesis in the proof.
\end{rem}

Let $A$ be a definable set and $\omega : A \fun \Gamma$ a definable function. We say that the definable pair $\bm A = (A, \omega)$ is \memph{proper invariant} if $A$ is proper invariant and $\omega$ is proper covariant, that is, every fiber $A_\gamma$ is proper invariant.

\begin{lem}\label{L:measure:surjective}
Suppose that $A$ is an $\RV$-pullback and  $(A_i)_i$ is a definable finite partition of $A$ such that each $\bm A_i = (A_i, \omega \rest A_i)$ is proper invariant. Then there is a special covariant bijection $T$ as given in Proposition~\ref{simplex:with:hole:rvproduct} such that $\omega \circ T^{-1}$ is $\rv$-contractible. In particular, if $\bm A_i \in \mVF[k]$ then $T$ induces a morphism  $\bm A_i \fun \mL(\bm U_i, \pi_i)$, where $(\bm U_i, \pi_i) \in \mRV^{\db}[k]$.
\end{lem}
\begin{proof}
We view $\omega$ as a definable function $A \fun \mdl P(\RV)$. By Proposition~\ref{simplex:with:hole:rvproduct}, there is a special covariant bijection $T$ on $A$ such that, for each $i$, $T \rest A_{i}$ is a special covariant bijection and $A_{i}^\flat$ is an $\RV$-pullback. Moreover, applying Remark~\ref{cons:extra}, we get that each function $\sigma_i \coloneqq (\omega \circ T^{-1}) \rest A_{i}^\flat$ is constant on every $\RV$-polydisc in $A_{i}^\flat$ and hence induces a function $\pi_i : \rv(A_{i}^\flat) \fun \Gamma$. Now, if $\bm A_i \in \mVF[k]$ then, by Lemma~\ref{RV:fiber:dim:same}, $\bm U_i\coloneqq \rv(A_{i}^\flat)_{\leq k} \in \RV^{\db}[k]$ (recall Notation~\ref{0coor}). We see that $(\bm U_i, \pi_i)$ is as required because, by Remark~\ref{spec:is:iso}, $\bm A_i$ is isomorphic to $(\mathbb{L} \bm U_i, \sigma_i)$,
\end{proof}

As above, this lemma may be applied to any proper invariant object $(A, \omega) \in \mVF[k]$, since  the object $(\RVH(A) \mi A, 0) \in \mVF[k]$  is also proper invariant.

\begin{lem}\label{invar:db}
Let $A$ be a definable set with $A_{\VF}$  bounded and $A_{\RV}$  doubly bounded. Let $f : A \fun \Gamma$ be a definable function  that is constant on $A \cap \ga$ for all open polydiscs $\ga$ of radius $\gamma \in \Gamma$. Then $f(A)$ is doubly bounded.
\end{lem}
\begin{proof}
Observe that, without loss of generality, we can replace $A$ with its $\gamma$-tubular neighborhood and thereby  assume that $(A,f)$ is indeed proper invariant. By Lemma~\ref{L:measure:surjective}, $f$ factors through a doubly bounded definable set in $\RV$. So, by Lemma~\ref{db:to:db},  $f(A)$ must be doubly bounded.
\end{proof}

\begin{cor}
If $(A, \omega)$ is proper invariant then $\omega(A)$ is doubly bounded.
\end{cor}
\begin{proof}
We have seen this many times above: as a consequence of Lemma~\ref{vol:par:bounded}, for all sufficiently large $\beta \in \Gamma$, every fiber $A_\gamma$ of $\omega$ is $\beta$-invariant. So the claim is immediate by Lemma~\ref{invar:db}.
\end{proof}

\subsection{Standard contractions}

\begin{defn}\label{rela:unary}
Let $A \sub \VF^{n} \times \RV^{m}$, $B \sub \VF^{n} \times \RV^{m'}$ be definable sets. A definable bijection $G : A \fun B$ is \memph{relatively unary}, or more precisely, \memph{relatively unary in the $(i, j)$-coordinate}, where $i, j \in [n]$, if $(\pr_{\tilde{j}} \circ G)(x) = \pr_{\tilde{i}}(x)$ for all $x \in A$. If, in addition, $G \rest A_a$ is also a special covariant bijection for every $a \in \pr_{\tilde{i}} (A)$ then $G$ is \memph{relatively special covariant} in the $(i, j)$-coordinate.

Let $\bm U = (U, f)$, $\bm V = (V, g)$ be objects of $\RV[k]$. A morphism $F : \bm U \fun \bm V$ is \memph{relatively unary in the $(i, j)$-coordinate}, where $i,j \in [k]$, if $(\pr_{\tilde{j}} \circ g \circ F)(u) = (\pr_{\tilde{i}} \circ f)(u)$ for all $u \in U$.
\end{defn}

Alternatively, for simplicity, we may take $i = j$ in this definition and stipulate that the condition in question holds up to coordinate permutations.

Since identity functions are relatively unary in any coordinate, if a  bijection  is piecewise (over a definable finite partition) a composition of relatively unary bijections then it is indeed a composition of relatively unary bijections.

\begin{exam}
Every  special covariant bijection $T$ of length $1$ is relatively special covariant, but clearly not vice versa (to begin with, the relevant data is not necessarily uniform across the fibers in each $\RV$-polydisc).
\end{exam}

Let $A \sub \VF^{n} \times \RV^{m}$ be a proper invariant set. Let $i \in [n]$ and $T_i : A \fun A^\flat$ be a proper covariant bijection that is
\begin{itemize}
  \item relatively unary in the $(i, i)$-coordinate (up to coordinate permutations),
  \item partially differentiable with the constant Jacobian $1$.
\end{itemize}
Moreover, the $\RV$-coordinates of $A$ are passed on to $A^\flat$ via $T_i$ in the obvious way (the simplest example is that $T_i$ is the regularization of $A$). Adjusting $T_i$ by way of Remark~\ref{spec:addendum} if necessary, we take $A$ and $A^\flat$ to be both doubly bounded. By Lemma~\ref{conv:unif}, $A^\flat$ is proper invariant as well, and $T_i$ is a proper covariant homeomorphism. Suppose that, for all $a \in A_{\tilde i}$ and all $t \in (A_{a})_{\RV}$, the image $T_i(A_{a, t}) \sub A_{a, t}^\flat$ of the fiber $A_{a, t} \sub A$ is an $\RV$-pullback.  Let
\[
 A_{-i} = \bigcup_{a \in A_{\tilde i}} a \times (A_{a}^\flat)_{\RV} \sub \VF^{n-1} \times \RV^{m_i},
\]
which is doubly bounded and proper invariant, and  we think of it as obtained from $A^\flat$ by ``$\rv$-contracting'' the $\VF$-coordinate in question.  Also, under Convention~\ref{conv:can}, $A^\flat$ can be ``fully recovered'' from  $A_{-i}$. Let $\hat T_i : A \fun A_{-i}$ be the function induced by $T_i$.

For any $j \in [n{-}1]$, we repeat the above setup on $A_{-i}$ with respect to the $j$th $\VF$-coordinate, and thereby obtain a doubly bounded proper invariant set $A_{-j} \sub \VF^{n-2} \times \RV^{m_j}$ and a function $\hat T_{j} : A_{-i} \fun A_{-j}$ that is induced by a relatively unary, partially differentiable, proper covariant homeomorphism $T_j : A^\flat \fun A^{\flat\flat}$ (the $\VF$-coordinate in $A^\flat$ that gets $\rv$-contracted in $A_{-i}$ is simply carried along by $T_j$). Continuing thus,  a sequence of such homeomorphisms  $T_{\sigma(1)}, \ldots, T_{\sigma(n)}$ and a corresponding function $\hat T_{\sigma} : A \fun \RV^{l}$ result, where $\sigma$ is the permutation of $[n]$ in question. The composition $T_{\sigma(n)} \circ \ldots \circ T_{\sigma(1)}$, which is referred to as the \memph{lift} of $\hat T_{\sigma}$, is denoted by $T_{\sigma}$. By Lemma~\ref{pr:co:home:com}, $T_{\sigma}$ is a proper covariant homeomorphism. Also note that $T^{-1}_{\sigma}$ is $\rv$-contractible.

\begin{defn}\label{defn:standard:contraction}
The function $\hat T_{\sigma}$, or the set $\hat T_{\sigma}(A)$, is called a \memph{standard contraction} of  $A$.
\end{defn}

\begin{lem}\label{exit:stand}
Every proper invariant set $A \sub \VF^{n} \times \RV^{m}$  admits a standard contraction with respect to any permutation $\sigma$ of $[n]$.
\end{lem}
\begin{proof}
We first remark that, although the lift of a standard contraction is in effect a more general (or rather less stringent) notion than that of a special covariant bijection, for existence, it is enough to exhibit a special covariant bijection that is also the lift of a standard contraction. To that end, the proof of Proposition~\ref{simplex:with:hole:rvproduct} almost works. The  issue is that a standard contraction is required to transform fiberwise a $\VF$-coordinate into an $\RV$-pullback in one fell swoop and subsequently carry the result in the $\RV$-coordinates, whereas in a special covariant bijection one is allowed to come back to modify the same $\VF$-coordinate many times over. We can of course simply add a clause in the inductive hypothesis to fulfill the new requirement. In that case, the special covariant bijection induced by $R$ there --- let us call it $R^*$ here --- should also be made to comply. This needs work: it is clear how to transform fiberwise the first $\VF$-coordinate into an $\RV$-pullback on top of $R^*$, but continuity is not automatic. Fortunately, it is only for convenience that  $R^*$ is carried out in the beginning, that is, we may equivalently graft the $\RV$-coordinates of $\ran(R)$ onto $A$ directly in the obvious way (strictly speaking this operation is not allowed in  a special covariant bijection, and that is why we cannot really opt for it in the proof of Proposition~\ref{simplex:with:hole:rvproduct}) and carry out $R^*$ in the end. With this modification, the same construction yields a standard contraction as desired.
\end{proof}

\begin{rem}\label{stand:vol}
Suppose that there is a $k \in 0 \cup [m]$ such that  $(A_{a})_{\leq k} \in \RV^{\db}[k]$  for  every $a \in A_{\VF}$. If $k=0$ then this just means $A \in \VF_*$. By Lemma~\ref{RV:fiber:dim:same}, $\hat T_{\sigma}(A)_{\leq n+k}$ is an object of $\RV^{\db}[n{+}k]$. Let $\bm A = (A, \omega)$ be a proper invariant pair such that the proper covariant function $\omega : A \fun \Gamma$   induces a function $\omega_{\hat T_{\sigma}}$ on $\hat T_{\sigma}(A)$ via $T_{\sigma}$, which just means that $\omega \circ T^{-1}_{\sigma}$ is invariant on $\RV$-polydiscs. Then the function $\hat T_{\sigma}$, or the object
\[
\hat T_{\sigma}(\bm A)_{\leq n+k} = (\hat T_{\sigma}(A), \omega_{\hat T_{\sigma}})_{\leq n+k} \in \mRV^{\db}[n{+}k],
\]
is understood as  a \memph{standard contraction} of $\bm A$. Using Lemma~\ref{exit:stand} instead of Proposition~\ref{simplex:with:hole:rvproduct} in the proof of Lemma~\ref{L:measure:surjective}, we see that such a $\hat T_{\sigma}$ does exist. Moreover, if we set $A^\sharp = \bigcup_{a \in A_{\VF}} a \times \bb L (A_a)_{\leq k}$, which is a  proper invariant object of $\VF_*$, and denote the  proper covariant function on $A^\sharp$ induced by  $\omega$ still by $\omega$, then  $T_{\sigma}$  indeed induces a $\mVF[n{+}k]$-morphism
\[
(A^\sharp, \omega) \fun {\mL}\hat T_{\sigma}(\bm A)_{\leq n+k}
\]
that can be written as a composition of relatively unary proper covariant homeomorphisms.

Here $k$ is dubbed the \memph{head start} of $\hat T_{\sigma}$, which is usually implicit in the sequel. In fact,  it never has to be anything other than $0$ except in Lemma~\ref{isp:VF:fiberwise:contract}, and can be circumvented even there. This seemingly needless gadget only serves to make the discussion more streamlined: if $A \in \VF_*$ then the intermediate steps of a standard contraction of $A$ may or may not result in objects of $\VF_*$ and hence the definition cannot be formulated entirely within $\VF_*$.
\end{rem}


\subsection{Lifting from $\RV$ to $\VF$}

\begin{lem}\label{bijection:partitioned:unary}
Every morphism in $\RV[k]$ can be written as a composition of relatively unary morphisms, similarly in $\mRV[k]$, $\RV^{\db}[k]$, and $\mRV^{\db}[k]$.
\end{lem}
\begin{proof}
It is enough to show this piecewise, and the situation can be  simplified accordingly step by step. Let $F : (U, f) \fun (V, g)$ be an $\RV[k]$-morphism. The idea is to adjust  $f(u)$ towards $g(F(u))$, simultaneously for all $u \in U$, one coordinate at a time, and the issue is how to stay in $\RV[k]$, that is, how to make the adjusted $f$ remain finite-to-one along the way. To that end, we proceed by induction on $n = \dim_{\RV}(U)$ and, without loss of generality, may take $F = \id$. The base case $n=0$ is straightforward and is left to the reader. For the inductive step, we may assume that both $\pr_{\leq n} \rest f(U)$ and $\pr_{\leq n} \rest g(V)$ are finite-to-one. Let $f' = \pr_{\leq n} \circ f$ and $g' = \pr_{\leq n} \circ g$. Since, as in the base case, the two obvious $\RV[k]$-morphisms
\[
(U,f) \fun (U \times 1, f' \times 1), \quad (U \times 1, g' \times 1) \fun (U, g)
\]
are compositions of relatively unary morphisms, it is enough to show that $(U, f') \fun (U, g')$ is such a morphism as well. So we may assume $n=k$.

Write $\dot U = f(U)$ and $\ddot U = g(U)$. For each $t \in \dot U_{< n}$ and each $i \in [n]$, let
$F^\dag_{t, i} : \dot U_t \fun \mdl P(\ddot U_i)$
be the definable finitary function induced by $F^\dag$ (recall Notation~\ref{mor:ftf}). Then, by Lemma~\ref{finite:fib:def}, there is a $t$-algebraic subset $B_{t, i} \sub \ddot U_i \sub \RV$ such that every fiber of $F^\dag_{t, i}$ over an element in $\ddot U_i \mi B_{t, i}$ is finite; denote by $W_{t, i} \sub \dot U$ the union of such fibers. Indeed, by the compactness argument thereof,  $B_{t, i}$ can be defined uniformly. Let $U_{t, i}$ be the set of those $u \in f^{-1}( t \times W_{t, i})$ such that $g(u)_i \notin B_{t, i}$, and $f'_{t, i} : U_{t, i} \fun \RV^n$ the function given by $u \efun (t, g(u)_i)$, which is finite-to-one. Let $U_i = \bigcup_{t}  U_{t, i}$ and $f'_i = \bigcup_{t}  f'_{t, i}$. Then $(U_i, f \rest U_i) \fun (U_i, f'_i)$ is a relatively unary $\RV[n]$-morphism. By applying the inductive hypothesis fiberwise to $(U_i, f'_i) \fun (U_i, g \rest U_i)$ in the obvious sense, we may actually assume $U_i = \0$, for every $i \in [n]$. Now, since $F^\dag(\dot U_t) \sub \prod_{i \in [n]} B_{t, i}$, which is finite, we must have $\dim_{\RV}(U) < n$. So the claim follows from the inductive hypothesis.

If $F$ is indeed a $\mRV[k]$-morphism then it is straightforward to equip the object in each intermediate step with a volume form so that the map in question  becomes a $\mRV[k]$-morphism. The other cases are rather similar.
\end{proof}

\begin{lem}\label{RVlift}
Let $U$, $V$ be sets in $\RV$ and $f : U \fun \mdl P(V)$ a definable finitary function. Then there is a definable finitary function $f^\sharp : U^\sharp \fun \mdl P(V^\sharp)$ that $\rv$-contracts to $f$. Moreover, for every $(u, v) \in f$, $f^\sharp_{u, v} \coloneqq (u, v)^\sharp \cap f^\sharp$ is a partially differentiable function $u^\sharp \fun v^\sharp$.
\end{lem}

Similar to Definition~\ref{defn:special:bijection}, this is another place where we need the property that every algebraic $\RV$-disc contains an algebraic point, and the easiest way to achieve
this is to assume that the substructure $\bb S$ is $\VF$-generated, which validates the repeated use of Lemma~\ref{algebraic:balls:definable:centers} below.

\begin{proof}
We proceed by induction on $\dim_{\RV}(U) = n$. For the base case $n=0$, $U$ is finite and hence, for every $u \in U$, the $\RV$-polydisc $u^\sharp$ contains a $u$-definable point, and similarly for every element in $f(u) \sub V$. So lifting $f$ as desired is trivial. For the inductive step, we may assume that $\pr_{\leq n} \rest U$ is finite-to-one. Let $U' = \pr_{\leq n}(U)$ and $f' : U' \fun \mdl P(V)$ be the definable finitary function induced by $f$. Observe that if $f'$ can be lifted as desired then  $f$ can be lifted (again trivially) as desired as well. So, without loss of generality,  $U \sub \RV^n$. By Lemma~\ref{algebraic:balls:definable:centers} and compactness, there is indeed a definable finitary function $f^\sharp : U^\sharp \fun \mdl P(V^\sharp)$ that $\rv$-contracts to $f$. By Lemma~\ref{diff:almost:every}, every $f^\sharp_{u, v}$ is partially differentiable almost everywhere. By Lemma~\ref{RV:bou:dim}, the $\RV$-boundary of the differential locus is of $\RV$-dimension less than $n$. The lemma follows.
\end{proof}

Let $F: \bm U \fun \bm V$ be an $\RV[k]$-morphism, where $\bm U = (U, f)$ and $\bm V = (V, g)$. If $F^{\sharp} : \bb L \bm U \fun \bb L \bm V$ is a definable bijection that $\rv$-contracts to the function $U_f \fun V_g$ induced by $F$ (recall Notation~\ref{RVcat:tag}) then it is called a \memph{lift} of $F$. We shall also think of the finite-to-finite correspondence $F^\dag$ as the $\rv$-contraction of such a lift.



\begin{lem}[{\cite[Lemma~6.3]{hrushovski:kazhdan:integration:vf}}]\label{iso:lifted:vol}
There is a lift $F^{\sharp}$ of $F$ such that
\begin{equation}\label{jac:mat}
\vv(\jcb_{\VF} F^{\sharp}(a,  u)) = \jcb_{\Gamma} F^\dag(u), \quad \text{for all } (a, u) \in \bb L \bm U.
\end{equation}
\end{lem}

See \cite[Corollary~7.7]{Yin:special:trans} for an alternative proof (the weaker qualifier ``for almost all $(a, u) \in \bb L \bm U$'' there may be upgraded to ``for all $(a, u) \in \bb L \bm U$'' here by appealing to Lemma~\ref{RV:bou:dim} at suitable places).

\begin{rem}\label{rvdisc:stre}
Let $t, s \in \RV$ such that the $\RV$-discs $t^\sharp$, $s^\sharp$ contain definable points. It is easy to see that for any definable $c \in \VF^\times$ with $\vv(c) = \vrv(s/t)$, there is a definable bijection $f : t^\sharp \fun s^\sharp$ such that $\ddx f = c$.
\end{rem}

\begin{lem}\label{L:measure:class:lift}
Every $\mRV[k]$-morphism $F : (\bm U, \omega) \fun (\bm V, \sigma)$ admits a lift
\[
F^\sharp: {\mu}\bb{L}(\bm U, \omega) \fun {\mu}\bb{L}(\bm V, \sigma)
\]
that is piecewise a composition of partially differentiable relatively unary $\mVF[k]$-morphisms. Moreover, if $F$ is a $\mRV^{\db}[k]$-morphism then every component of the composition is a  proper covariant homeomorphism and hence so is $F^\sharp$.
\end{lem}
\begin{proof}
We do induction on $\dim_{\RV}(U) = n$. For the base case $n=0$, we may assume that $U$ is just a singleton and $F$ is relatively unary. Since  every $\RV$-disc $t^\sharp$ involved contains a $t$-definable point, it is easy to lift $F$  as desired by applying Remark~\ref{rvdisc:stre} in the coordinate in question.

For the inductive step, again, we may assume that $f' = \pr_{\leq n} \circ f$, $g' = \pr_{\leq n} \circ g$ are finite-to-one. Write $\bm U' = (U, f')$ and $\bm V' = (V, g')$. By Lemma~\ref{RVlift}, there is a definable partially differentiable finitary function that $\rv$-contracts to the finitary function $f(U)_{\leq n} \fun \mdl P(f(U)_{> n})$. This means that the obvious $\mRV[k]$-morphism $(\bm U, \omega) \fun (\bm U' \times 1, \omega')$ can be lifted as desired by applying Remark~\ref{rvdisc:stre} fiberwise as in the base case above, and similarly for the obvious $\mRV[k]$-morphism $(\bm V' \times 1, \sigma') \fun (\bm V, \sigma)$. Thus we are reduced to the case $n = k$. By Lemma~\ref{bijection:partitioned:unary}, we may assume that $F$ is  relative unary. Then, applying Lemma~\ref{iso:lifted:vol} fiberwise in the obvious way, we can lift $F$ to a relatively unary bijection. By Lemmas~\ref{diff:almost:every} and~\ref{RV:bou:dim}, the lift is partially differentiable outside a definable subset of $\RV$-dimension less than $k$, and the restriction  is indeed a $\mVF[k]$-morphism due to the condition (\ref{jac:mat}). So we are done by  the inductive hypothesis.

The second claim is a consequence of Lemmas~\ref{conti:covar}, \ref{conv:unif}, and \ref{pr:co:home:com}.
\end{proof}

\begin{defn}\label{defn:diamond}
The subcategory $\mVF^{\diamond}[k]$ of $\mVF[k]$ consists of the proper invariant objects and the morphisms that are compositions of relatively unary proper covariant homeomorphisms.
\end{defn}

\begin{rem}\label{dia:cat:well}
Obviously the composition law holds in  $\mgVF^{\diamond}[k]$ and hence it is indeed a category. Moreover, every morphism in it is an honest  bijection, as opposed to merely an essential bijection, and is in effect required to admit an inverse.  Thus, $\mgVF^{\diamond}[k]$ is already a groupoid and there is no need to pass to a quotient category as in Remark~\ref{mor:equi}. On the other hand, Proposition~\ref{simplex:with:hole:rvproduct}, Lemma~\ref{spec:inve}, and Remark~\ref{stand:vol} give various constructions of nontrivial morphisms.
\end{rem}

\begin{rem}
The main reason that we need the lift in Lemma~\ref{L:measure:class:lift} to be in that particular form is that, as it will become clear, we are forced to work with explicit compositions of relatively unary proper covariant morphisms, due to the failure of generalizing Lemma~\ref{simul:special:dim:1} below to higher dimensions; see \cite[\S~1]{Yin:int:acvf} for further explanation.

Every definable bijection between definable sets with the same number of $\VF$-coordinates is a composition of relatively unary bijections, see \cite[Lemma~5.6]{Yin:int:acvf}.   It is certainly more satisfying to simply take proper covariant homeomorphisms as morphisms in $\mVF^{\diamond}[k]$ and then show that they are all indeed compositions of relatively unary proper covariant homeomorphisms. We suspect this is not true, although it does not seem straightforward to come up with a counterexample.
\end{rem}

\begin{cor}\label{L:sur:c}
The lifting maps ${\mu}\bb L_{k}$ induce surjective homomorphisms, which are often simply denoted by ${\mu}\bb L$, between the Grothendieck semigroups
\[
 \gsk \mRV^{\db}[k] \epi \gsk \mVF^{\diamond}[k].
\]
\end{cor}
\begin{proof}
By Lemma~\ref{L:measure:class:lift}, every $\mRV^{\db}[k]$-morphism can be lifted to a $\mVF^{\diamond}[k]$-morphism. Thus, ${\mu}\bb L_{k}$ induces a map on the isomorphism classes, which is easily seen to be a semigroup homomorphism. By Lemma~\ref{L:measure:surjective} (see the remark thereafter) or Remark~\ref{stand:vol}, this  homomorphism is surjective.
\end{proof}

\begin{lem}\label{simul:special:dim:1}
Let $f : (A, \omega) \fun (B, \sigma)$ be a $\mgVF^{\diamond}[1]$-morphism.  Then there are special covariant bijections $T_A : A \fun A^{\flat}$, $T_B : B \fun B^{\flat}$ such that $A^{\flat}$, $B^{\flat}$ are doubly bounded $\RV$-pullbacks and the functions $f^{\flat}$, $f^{\flat}_{\downarrow}$ are bijective in
the diagram
\[
\bfig
  \square(0,0)/->`->`->`->/<600,400>[A`A^{\flat}`B`B^{\flat};
  T_A`f``T_B]
 \square(600,0)/->`->`->`->/<600,400>[A^{\flat}`\rv(A^{\flat})`B^{\flat} `\rv(B^{\flat});  \rv`f^{\flat}`f^{\flat}_{\downarrow}`\rv]
 \efig
\]
where the second square commutes (but not necessarily the first one, although it fails at only finitely many points). Moreover, $\omega' = \omega \circ T_A^{-1}$ and $\sigma' = \sigma \circ T_B^{-1}$ are both $\rv$-contractible, and  $f^{\flat}_{\downarrow}$ is a $\mRV^{\db}[1]$-morphism $(\rv(A^{\flat}), \omega')_{1} \fun (\rv(B^{\flat}), \sigma')_1$ (recall Notation~\ref{0coor}).
\end{lem}
\begin{proof}
This is a variation of \cite[Lemma~5.2]{Yin:int:acvf} adapted to the current environment and can be established by modifying two special bijections $T'_A$, $T'_B$ as constructed in the proof thereof as follows.

By Lemma~\ref{L:measure:surjective}, we may assume that $A$, $B$ are doubly bounded $\RV$-pullbacks and $\omega$, $\sigma$ are already $\rv$-contractible. By Lemmas~\ref{dcl:to:ac} and \ref{RV:fiber:dim:same}, each focus map in $T'_A$ consists of an algebraic set of focus points, similarly for the focus maps in $T'_B$. We think of these focus points as contained in $A$, $B$. Each one of the focus maps may produce $0$ in the $\VF$-coordinate in the resulting set, but this is not allowed in a special covariant bijection. To remedy  this, we choose pairwise disjoint open discs of sufficiently large definable radii such that each of them contains exactly one focus point and the volume forms $\omega$, $\sigma$ are constant on them. By the construction in the proof of \cite[Lemma~5.2]{Yin:int:acvf}, $f$ has the disc-to-disc property outside these open discs. Inside these discs, by Lemma~\ref{lim:weak}, the radii may be chosen so that $f$ is weakly concentric at the focus points. In particular, since $f^\flat$ $\rv$-contracts to a bijection,  these open discs can be paired up via the induced bijection  in question. At this point it is not hard to see how to modify $T'_A$, $T'_B$ according to the second clause of (\ref{centri}). The first claim follows. For the second claim, since the volume forms $\omega'$, $\sigma'$ are $\rv$-contractible, we can use the disc-to-disc property (or weak concentricity for the exceptional discs) to gauge the difference of the radii of the $\RV$-discs in question against the constant valuation of the differentials, and this shows that $\omega'$, $\sigma'$  do match under $f^{\flat}_{\downarrow}$ in the required sense.
\end{proof}

\section{Integrating proper invariant sets}
\label{section:int}

\subsection{Annular blowups}
Recall that we write $\RV$ to mean the $\RV$-sort without the  element $\infty$, and $\RV_\infty$ otherwise, although quite often the difference is immaterial and it does not matter which set $\RV$ stands for. In the definition below, this difference does matter, in particular, $\RV^{\ccirc}$ denotes the set $\rv(\MM \mi 0)$.

\begin{nota}\label{punc:RVd}
Recall from Notation~\ref{defn:disc} that the set $\rv(\MM_{\gamma} \mi 0)$ is denoted by $\RV^{\ccirc}_\gamma$. For $\gamma \in \Gamma^+$ and $t \in \gamma^\sharp$, write
\[
\RV^{\ccirc}(t)_\gamma = (\RV^{\ccirc} \mi \RV^{\ccirc}_\gamma) \uplus t;
\]
if $\gamma = \infty$ then $\RV^{\ccirc}(t)_\gamma = \RV^{\ccirc}_\infty$ and if  $\gamma = 0$ then $\RV^{\ccirc}(1)_\gamma = 1 \in \RV$.
\end{nota}


\begin{defn}\label{defn:blowup:coa}
Let $\bm U = (U, f, \omega)$ be an object of $\mRV[k]$, where $k > 0$, such that the function $\pr_{\tilde j} \rest f(U)$ is finite-to-one for some $j \in [k]$. Write $f = (f_1, \ldots, f_k)$. Let $t \in \RV_\infty$ be a definable element with $\tau = \vrv(t)$ nonnegative.  An \memph{elementary annular blowup} of $\bm U$ in the $j$th coordinate of \memph{aperture} $(\tau, t)$ is the triple $\bm U^{\flat} = (U^{\flat}, f^{\flat}, \omega^\flat)$, where $U^{\flat} = U \times \RV^{\ccirc}(t)_{\tau}$ and for every $(u, s) \in U^{\flat}$,
\[
f^{\flat}_{i}(u, s) = f_{i}(u) \text{ for } i \neq j, \quad f^{\flat}_{j}(u, s) = s f_{j}(u), \quad \omega^\flat(u, s) = \omega(u).
\]

Let $\bm V = (V, g, \sigma)$ be another object of $\mRV[k]$ and $C_i \sub V$ finitely many pairwise disjoint definable sets. Each triple
\[
\bm C_i = (C_i, g \rest C_i, \sigma \rest C_i) \in \mRV[k]
\]
is referred to as a \memph{subobject} of $\bm V$. Suppose that $F_i : \bm U_i \fun \bm C_i$ is a $\mRV[k]$-morphism and $\bm U^{\flat}_i$ is an elementary annular blowup of $\bm U_i$. Let $C = \bigcup_i C_i$, $\bm C = \bigcup_i \bm C_i$,  and $F = \biguplus_i F_i$. Then the object
$
 (\bm V \mi \bm C) \uplus \biguplus_i \bm U_i^{\flat}
$
is  an \memph{annular blowup of $\bm V$ via $F$}, denoted by $\bm V^{\flat}_F$. The subscript $F$ is usually dropped. The object $\bm C$ (or the set $C$) is referred to as the \memph{locus} of $\bm V^{\flat}_F$.

An \memph{annular blowup of length $n$} is a composition of $n$ annular blowups.
\end{defn}

\begin{rem}
An elementary annular blowup could be an elementary blowup (see \cite[Definition~6.1]{Yin:int:acvf}) if the aperture is allowed to be $(\infty, \infty)$. If there is an elementary annular blowup of $\bm U$ then, \textit{a posteriori}, $\dim_{\RV}(U) < k$. For any coordinate of $f(U)$, there is at most one elementary blowup of $\bm U$, whereas there could be many annular elementary blowups of $\bm U$. We should have included the coordinate that is blown up as a part of the data. However, in context, either this is
clear or it does not need to be spelled out, and we shall suppress mention of it for ease of notation.
\end{rem}

\begin{rem}
Clearly if $\bm U \in \mRV^{\db}[k]$  then $\bm U^{\flat} \in \mRV^{\db}[k]$ too. We shall only consider annular blowups  in $\mRV^{\db}[k]$ (then, for simplicity, they will just be referred to as blowups). This notion is dual to that of a  special covariant bijection. As a matter of fact, to make this duality precise is essentially what is left to do for the rest of our main construction.
\end{rem}

For the next few lemmas, let $\bm U = (U, f, \omega)$ and $\bm V = (V, g, \sigma)$ be objects of $\mRV^{\db}[k]$.


\begin{lem}\label{elementary:blowups:preserves:iso:vol}
Suppose that $[\bm U] = [\bm V]$ and $\bm U^{\flat}$, $\bm V^{\flat}$ are elementary blowups of $\bm U$, $\bm V$ in the $k$th coordinate. Then there are  blowups $\bm U^{\flat\flat}$, $\bm V^{\flat\flat}$ of $\bm U^{\flat}$, $\bm V^{\flat}$ of length $1$ such that $[\bm U^{\flat\flat}] = [\bm V^{\flat\flat}]$.
\end{lem}
\begin{proof}
Let $(\rho, r)$, $(\tau, t)$ be the apertures of  $\bm U^{\flat}$, $\bm V^{\flat}$, respectively. If $\rho = \tau$ then clearly any morphism $\bm U \fun \bm V$ induces a morphism $\bm U^{\flat} \fun \bm V^{\flat}$. So we may assume $\rho' \coloneqq \tau - \rho > 0$. Choose a definable element $r' \in \rho'^\sharp$. Let $W = U \times r$ and $\bm W$ be the corresponding subobject of $\bm U^{\flat}$. Then it is easy to see that the blowup $\bm U^{\flat\flat}$  of $\bm U^{\flat}$ with locus $\bm W$ and aperture $(\rho', r')$ is isomorphic to $\bm V^{\flat}$.
\end{proof}

\begin{cor}\label{eblowup:same:loci:iso}
Let $F : \bm U \fun \bm V$ be a morphism and $\bm U^{\flat}$, $\bm V^{\flat}$ blowups of  $\bm U$, $\bm V$ of length $1$. Then there are blowups $\bm U^{\flat\flat}$, $\bm V^{\flat\flat}$ of $\bm U^{\flat}$, $\bm V^{\flat}$ of length $1$ such that $[\bm U^{\flat\flat}] = [\bm V^{\flat\flat}]$.
\end{cor}
\begin{proof}
Let  $C$, $D$ be the loci of $\bm U^{\flat}$, $\bm V^{\flat}$, respectively. By Lemma~\ref{elementary:blowups:preserves:iso:vol}, we can make the claim hold by restricting the loci of $\bm U^{\flat}$, $\bm V^{\flat}$ to $C \cap F^{-1}(D)$, $F(C) \cap D$. But in the meantime we can also blow up $\bm U^{\flat}$, $\bm V^{\flat}$ at the loci $F^{-1}(D) \mi C$, $F(C) \mi D$ using the apertures induced by those on $D \mi F(C)$, $C \mi F^{-1}(D)$, respectively. Then the resulting compounded blowups of $\bm U^{\flat}$, $\bm V^{\flat}$ of length $1$ are as desired.
\end{proof}

This corollary also holds in $\RV[k]$, which is essentially the only thing that the otherwise formal proof of \cite[Lemma~6.5]{Yin:special:trans} depends on. Thus, the same proof yields the following analogue in the present context:

\begin{lem}\label{blowup:equi:class:coa}
If $[\bm U] = [\bm V]$ and $\bm U_1$, $\bm V_1$ are blowups of $\bm U$, $\bm V$ of lengths $m$, $n$, respectively, then there are blowups $\bm U_2$, $\bm V_2$ of $\bm U_1$, $\bm V_1$ of lengths $n$, $m$, respectively, such that $[\bm U_2] = [\bm V_2]$.
\end{lem}

\begin{cor}\label{blowup:equi:class}
Suppose that $[\bm U] = [\bm U']$ and $[\bm V] = [\bm V']$. If there are isomorphic blowups of $\bm U$, $\bm V$ then there are isomorphic blowups of $\bm U'$, $\bm V'$.
\end{cor}

\begin{defn}
Let $\mispdb[k]$ be the set of pairs $(\bm U, \bm V)$ of objects of $\mRV^{\db}[k]$ such that there exist isomorphic blowups $\bm U^{\flat}$, $\bm V^{\flat}$. Let $\mispdb[*] = \bigoplus_{k} \mispdb[k]$.
\end{defn}

We will just write $\mispdb$ for all these sets when there is no danger of confusion. By Corollary~\ref{blowup:equi:class}, they may be regarded as binary relations on isomorphism classes.

\begin{lem}\label{isp:congruence:vol}
$\mispdb[k]$ is a semigroup congruence relation and $\mispdb[*]$ is a semiring congruence relation.
\end{lem}
\begin{proof}
This is an easy consequence of Lemma~\ref{blowup:equi:class:coa} (the proof of \cite[Lemma~6.8]{Yin:special:trans}, which uses \cite[Lemma~6.5]{Yin:special:trans}, contains more details).
\end{proof}

Let  $T$ be a special covariant bijection on $\bb L (U,f)$. Denote the object $(\mathbb{L} (U,f)^\flat_{\RV})_{\leq k} \in \RV^{\db}[k]$ by $(U_{T})_{\leq k}$ and  the object $(U_{T}, \omega_T)_{\leq k} \in \mRV^{\db}[k]$ by $\bm U_{T}$, where $\omega_T : U_{T} \fun \RV$ is the volume form induced by $\omega$.

\begin{lem}\label{special:to:blowup:coa}
The object $\bm U_T$ is isomorphic to a blowup of $\bm U$ of the same length as $T$.
\end{lem}
\begin{proof}
By induction on the length $\lh (T)$ of $T$ and Lemma~\ref{blowup:equi:class:coa}, this is immediately reduced to the case $\lh (T) = 1$. Without loss of generality, we may assume that the locus of $T$ is indeed $\mathbb{L} (U,f)$. Then it is clear how to construct an (elementary) blowup of $\bm U$ as desired, using the aperture provided by $T$.
\end{proof}

\begin{lem}\label{blowup:same:RV:coa}
Let $\bm U^{\flat}$ be a blowup of $\bm U$. Then $\mL \bm U^{\flat}$ is isomorphic to $\mL \bm U$ in $\mVF^{\diamond}[k]$.
\end{lem}
\begin{proof}
By induction on the length $l$ of $\bm U^{\flat}$, this is immediately reduced to the case $l=1$. We may assume that $\bm U^{\flat}$ is an elementary blowup in the first coordinate. Let $(\tau, t)$ be the aperture of $\bm U^{\flat}$. Fix a definable point $c \in t^\sharp$. For $u \in U$ and $a \in  f(u)_{\tilde 1}^\sharp$, every $\RV$-disc in $f(u)_{1}^\sharp$ contains an $a$-algebraic point $b_a$; moreover, by Lemma~\ref{RVlift}, these points may be chosen uniformly via a partially differentiable finitary function. Thus, there is a continuous (fiberwise) additive translation, with respect to the points $b_a$ and $b_a - c b_a$, between the two sets in question, which is a $\mVF^{\diamond}[k]$-morphism by Lemmas~\ref{conti:covar} and \ref{conv:unif}.
\end{proof}

\begin{lem}\label{kernel:dim:1:coa}
Let $[\bm A] = [\bm B]$ in  $\gsk \mVF^{\diamond}[1]$ and $\bm U, \bm V \in \mRV^{\db}[1]$ be standard contractions of $\bm A$, $\bm B$, respectively. Then $([\bm U], [\bm V]) \in \mispdb$.
\end{lem}
\begin{proof}
This is immediate by Lemmas~\ref{simul:special:dim:1} (applied to $\mL \bm U$, $\mL \bm V$) and \ref{special:to:blowup:coa}.
\end{proof}

\begin{lem}\label{isp:VF:fiberwise:contract}
Let $A'$, $A''$ be definable sets with $A'_{\VF} = A''_{\VF} = A \sub \VF^n$ and $\omega'$, $\omega''$ definable functions from $A'$, $A''$ into $\Gamma$, respectively. Write $\bm A' = (A', \omega')$ and $\bm A'' = (A'', \omega'')$. Suppose that $\bm A'$, $\bm A''$ are proper invariant and there is a $k \in \N$ such that for  every $a \in A$,
\begin{equation}\label{condi:fiber}
([(A'_a, \omega')]_{\leq k}, [(A''_a, \omega'')]_{\leq k}) \in \mispdb.
\end{equation}
Let $\hat T_{\sigma}$, $\hat R_{\sigma}$ be standard contractions of $\bm A'$, $\bm A''$, respectively. Then
\[
([\hat T_{\sigma}(\bm A')]_{\leq n+k}, [\hat R_{\sigma}(\bm A'')]_{\leq n+k}) \in \mispdb.
\]
\end{lem}

Note that condition (\ref{condi:fiber}) makes sense only  over the substructure $\mathbb{S} \la a \ra$.

\begin{proof}
By induction on $n$, this is immediately reduced to the case $n=1$. So assume $A \sub \VF$. Using Proposition~\ref{simplex:with:hole:rvproduct} and FMT, we can construct a special covariant bijection $F : A \fun A^\sharp$, as in the proof of \cite[Lemma~6.14]{Yin:int:acvf}, such that for all $\RV$-polydisc $\gp \sub A^{\sharp}$ and all $a_1, a_2 \in F^{-1}(\gp)$,
\[
\omega' \rest A'_{a_1} = \omega' \rest A'_{a_2} \dand \omega'' \rest A''_{a_1} = \omega'' \rest A''_{a_2}.
\]
Therefore, using Lemma~\ref{kernel:dim:1:coa} in place of \cite[Corollary~6.11]{Yin:int:acvf}, that proof goes through here with virtually no changes, and the additional computations involving $\jcb_{\Gamma}$ are all straightforward. (Also consult the proof of \cite[Lemma~5.36]{Yin:tcon:I}, which is in a different environment but is formally the same and is  written down with more details.)
\end{proof}

\begin{cor}\label{contraction:same:perm:isp}
Let $\bm A', \bm A'' \in \mVF^{\diamond}[n]$ and suppose that there is a morphism $F : \bm A' \fun \bm A''$ that is relatively unary in the $i$th $\VF$-coordinate. Then for any permutation $\sigma$ of $[n]$ with $\sigma(1) = i$ and any standard contractions $\hat T_{\sigma}$, $\hat R_{\sigma}$ of $\bm A'$, $\bm A''$,
\[
([\hat T_{\sigma}(\bm A')]_{\leq n}, [\hat R_{\sigma}(\bm A'')]_{\leq n}) \in \mispdb.
\]
\end{cor}
\begin{proof}
Note that, since $F$ is differentiable outside a definable subset of $\VF$-dimension less than $n$, it may not induce a morphism $\bm A'_a \fun \bm A''_a$ for every $a \in \pr_{\tilde i}(A)$. Nevertheless, for all sufficiently large $\delta \in \Gamma^+$, both $A'$ and $\hat T_{\sigma(1)}(A')$ are $\delta$-invariant. This means that for all $a \in \pr_{\tilde i}(A)$ there is a $b \in \pr_{\tilde i}(A)$ such that they are contained in the same open polydisc of radius $\delta$, and $F$ indeed induces a morphism $\bm A'_b \fun \bm A''_b$, and $\hat T_{\sigma(1)}(A')_a = \hat T_{\sigma(1)}(A')_b$. Thus, the claim follows immediately from Lemmas~\ref{kernel:dim:1:coa} and \ref{isp:VF:fiberwise:contract}.
\end{proof}

\subsection{$2$-cells}

\begin{defn}[$\vv$-affine and $\rv$-affine]\label{rvaffine}
Let $\ga$ be an open disc and $f : \ga \fun \VF$ an injection.
We say that $f$ is \memph{$\vv$-affine} if there is a (necessarily unique) $\gamma \in \Gamma$, called the \memph{shift} of $f$, such that, for all $a, a' \in \ga$,
\[
\vv(f(a) - f(a')) = \gamma + \vv(a - a').
\]
We say that $f$ is \memph{$\rv$-affine} if there is a (necessarily unique) $t \in \RV$, called the \memph{slope} of $f$, such that, for all $a, a' \in \ga$,
\[
\rv(f(a) - f(a')) = t \rv(a - a').
\]
\end{defn}

\begin{defn}\label{defn:balance}
Let $A \sub \VF^2$ be a definable set such that $\ga_1 \coloneqq \pr_1(A)$ and $\ga_2 \coloneqq \pr_2(A)$ are both open discs. Let $f : \ga_1 \fun \ga_2$ be a definable bijection that has the disc-to-disc property.  We say that $f$ is \memph{balanced in $A$} if $f$ is actually $\rv$-affine and there are $t_1, t_2 \in \RV_\infty$, called the \memph{paradigms} of $f$, such that, for every $a \in \ga_1$,
\[
A_a = t_2^\sharp + f(a) \dand f^{-1}(A_a) = a - t_1^\sharp.
\]
\end{defn}

\begin{rem}\label{2cell:linear}
Suppose that $f$ is balanced in $A$ with paradigms $t_1$, $t_2$.
If one of the paradigms is $\infty$ then the other one must be $\infty$. In this case $A$ is just the (graph of the) bijection $f$ itself.

Assume that $t_1$, $t_2$ are not $\infty$. Let $\gB_1$, $\gB_2$ be the sets of closed subdiscs of $\ga_1$, $\ga_2$ of radii $\vrv(t_1)$, $\vrv(t_2)$, respectively. Let $a_1 \in \gb_1 \in \gB_1$ and $\go_1$ be the maximal open subdisc of $\gb_1$ containing $a_1$. Let $\gb_2 \in \gB_2$ be the smallest closed disc containing the open disc $\go_2 \coloneqq A_{a_1}$. Then, for all $a_2 \in \go_2$, we have
\[
\go_2 = t_2^\sharp + f(\go_1) = A_{a_1} \dand A_{a_2} = f^{-1}(\go_2) + t_1^\sharp = \go_1.
\]
This internal symmetry of $A$ is illustrated by the following diagram:
\[
\bfig
  \dtriangle(0,0)|amb|/.``<-/<600,250>[\go_1`f^{-1}(\go_2)`\go_2; \pm t_1^\sharp`\times`f^{-1}]
  \ptriangle(600,0)|amb|/->``./<600,250>[\go_1`f(\go_1)`\go_2; f`` \pm t_2^\sharp]
 \efig
\]
Since $f$ is $\rv$-affine, we see that its slope must be $-t_2/t_1$.

Let $\tor(\code {\gb_1})$, $\tor(\code {\gb_1})$ denote the sets of the maximal open subdiscs of $\gb_1$, $\gb_1$. These may be viewed as $\K$-torsors and as such  are equipped with much of the structure of $\K$. Then the set $A \cap (\gb_1 \times \gb_2)$ may be thought of as the ``line'' in $\tor(\code {\gb_1}) \times \tor(\code {\gb_2})$ given by the equation
\[
x_2 = - \tfrac{t_2}{t_1}(x_1 - \code{\go_1}) + (\code{\go_2} - t_2).
\]
Thus, by Lemma~\ref{simul:special:dim:1}, the obvious bijection between $A_1 \times t_2^\sharp$ and $t_1^\sharp \times A_2$ is the lift of an $\RV^{\db}[2]$-morphism modulo special covariant bijections; see Lemma~\ref{2:unit:contracted} below for details.
\end{rem}

\begin{defn}\label{def:units}
We say that a set $A$ is a \memph{$1$-cell} if it is either an open disc contained in a single $\RV$-disc or a point in $\VF$. We say
that $A$ is a \memph{$2$-cell} if
\begin{itemize}
 \item $A$ is a subset of $\VF^2$ contained in a single $\RV$-polydisc and $A_1$ is a $1$-cell,
 \item there is a function $\epsilon : A_1  \fun \VF$ and a $t \in \RV$ such that, for every $a \in A_1$, $A_a = t^\sharp + \epsilon(a)$,
 \item one of the following three possibilities occurs:
  \begin{itemize}
   \item $\epsilon$ is constant,
   \item $\epsilon$ is injective, has the disc-to-disc property, and $\rad(\epsilon(A_1)) \geq \vrv(t)$,
   \item $\epsilon$ is balanced in $A$.
  \end{itemize}
\end{itemize}
The function $\epsilon$ is called the \memph{positioning function} of $A$ and the element $t$ the \memph{paradigm} of $A$.

More generally, a set $A$ with exactly one $\VF$-coordinate is a \memph{$1$-cell} if, for each $t \in A_{\RV}$, $A_t$ is a $1$-cell in the above sense; the parameterized version of the notion of a $2$-cell is formulated in the same way.
\end{defn}

\begin{lem}\label{decom:into:2:units}
Let  $A \sub \VF^2$ be a proper invariant set. Then there is a standard contraction $\hat T_\sigma$ of $A$ such that $A_s$ is a $2$-cell for every $s \in \hat T_\sigma(A)$.
\end{lem}
\begin{proof}
This is a variation of  \cite[Lemma~4.8]{Yin:int:acvf}. The proof of the latter proceeds by constructing a positioning function $\epsilon_{(t,s)}$ in each $\VF$-fiber, which heavily relies on the use of standard contractions in the preceding auxiliary results, namely \cite[Lemmas~4.2, 4.4]{Yin:int:acvf}.  It still goes through here, provided that FMT is applied at suitable places to cut out finitely many exceptional open discs in the first coordinate at which the desired properties of $\epsilon_{(t,s)}$ do not hold, as we have done above, say, in the proof of Proposition~\ref{simplex:with:hole:rvproduct}.
\end{proof}

We also remark that Lemma~\ref{decom:into:2:units} holds fiberwise for proper invariant sets $A \sub \VF^n$ with $n \geq 2$, that is, there is a standard contraction $\hat T_\sigma$ of $A$ such that for every $(a, s) \in \hat T_{\sigma(2)}(A)$, $\hat T_{\sigma(2)}^{-1}(a, s)$ is of the form $a \times C$, where $C$ is a $2$-cell. This follows from Lemma~\ref{decom:into:2:units} and the argument in the proof of  Lemma~\ref{exit:stand} (Lemma~\ref{decom:into:2:units} serves as the base case and hence, in the inductive step, we can assume that $\hat T_{\sigma(2)}$ is already as desired fiberwise).

For the next two lemmas, let $12$, $21$ denote the permutations of $[2]$ and $\bm A = (A, \omega) \in \mVF^\diamond[2]$.

\begin{lem}\label{2:unit:contracted}
Suppose that $A \sub \VF^2$ is a $2$-cell and $\omega$ is constant. Then there are standard contractions $\hat T_{12}$, $\hat R_{21}$ of $\bm A$ such that $[\hat T_{12}(\bm A)]_{\leq 2} = [\hat R_{21}(\bm A)]_{\leq 2}$.
\end{lem}
\begin{proof}
All we need to do is to check that the maps constructed in the proof of  \cite[Lemma~5.7]{Yin:int:acvf} are indeed $\mRV^{\db}[2]$-morphisms. By inspection of that proof, we see that there are two cases: $A$ is a product of two open discs or the positioning function $\epsilon$ in question is balanced in $A$ with nonzero paradigms $t_1$, $t_2$. The first case is obvious since we can simply use the identity map. In the second case,  a morphism between the standard contractions can be easily constructed using  $\epsilon$  (we could also cite Lemma~\ref{simul:special:dim:1}, but the situation here is much simpler), and  the requirement on $\jcb_{\Gamma}$ is satisfied since the slope of the $\rv$-affine function $\epsilon$ is $-t_2/t_1$ (see the last paragraph of Remark~\ref{2cell:linear} for further explanation).
\end{proof}

\begin{lem}\label{subset:partitioned:2:unit:contracted}
There are a morphism $\bm A \fun \bm A^*$, relatively unary in both coordinates, and two standard contractions $\hat T_{12}$, $\hat R_{21}$ of $\bm A^*$ such that $[\hat T_{12}(\bm A^*)]_{\leq 2} = [\hat R_{21}(\bm A^*)]_{\leq 2}$.
\end{lem}
\begin{proof}
This follows from  Lemmas~\ref{decom:into:2:units}, \ref{2:unit:contracted}, and compactness (see the proof of \cite[Corollary~5.8]{Yin:int:acvf} for a bit more details).
\end{proof}

\begin{lem}\label{contraction:perm:pair:isp}
Let $\bm A = (A, \omega) \in \mVF^\diamond[n]$. Let $i, j \in [n]$ with $i \neq j$ and $\sigma_1$, $\sigma_2$ be  permutations of $[n]$ such that
\[
\sigma_1(1) = \sigma_2(2) = i, \quad \sigma_1(2) = \sigma_2(1) = j, \quad \sigma_1
\rest \set{3, \ldots, n} = \sigma_2 \rest \set{3, \ldots, n}.
\]
Then, for any standard contractions $\hat T_{\sigma_1}$, $\hat T_{\sigma_2}$ of $\bm A$,
\[
([\hat T_{\sigma_1}(\bm A)]_{\leq n}, [\hat T_{\sigma_2}(\bm A)]_{\leq n}) \in \mispdb.
\]
\end{lem}
\begin{proof}
Let $ij$, $ji$ denote the permutations of $\{i, j\}$ and $E = [n] \mi \{i, j\}$. By compactness and
Lemma~\ref{isp:VF:fiberwise:contract}, it is enough to show
that for  all $a \in A_E$ and all standard
contractions $\hat T_{ij}$, $\hat T_{ji}$ of $\bm A_a = (A_a, \omega \rest A_a)$,
\[
([\hat T_{ij}(\bm A_a)]_{\leq 2}, [\hat T_{ji}(\bm A_a)]_{\leq 2})
\in \mispdb.
\]
By Corollary~\ref{contraction:same:perm:isp} and Lemma~\ref{isp:VF:fiberwise:contract}, it is enough to find a morphism $\bm A_a \fun \bm B$, relatively unary in both coordinates, and  standard contractions $\hat R_{ij}$, $\hat R_{ji}$ of $\bm B$ such that $[\hat R_{ij}(\bm B)]_{\leq 2} = [\hat R_{ji}(\bm B)]_{\leq 2}$. This is just Lemma~\ref{subset:partitioned:2:unit:contracted}.
\end{proof}

\subsection{The kernel of $\mL$ and the bounded integral}

The following proposition is the culmination of the preceding technicalities, it identifies the
congruence relation  $\mispdb$ with that induced by
$\mL$.

\begin{prop}\label{kernel:L}
For $\bm U, \bm V \in \mRV^{\db}[k]$,
\[
[\mu\bb L \bm U] = [\mu\bb L \bm V] \text{ in }\gsk \mVF^\diamond[k] \quad \text{if and only if} \quad ([\bm U], [\bm V]) \in \mispdb.
\]
\end{prop}
\begin{proof}
The ``if'' direction simply follows from Lemma~\ref{blowup:same:RV:coa} and
Corollary~\ref{L:sur:c}.

For the ``only if'' direction, we proceed by induction on $k$. The base case $k = 1$ is of course Lemma~\ref{kernel:dim:1:coa}. For the inductive step, let
\[
\mu\bb L \bm U = \bm B_1 \to^{G_1} \bm B_2 \cdots \bm B_l \to^{G_l} \bm B_{l+1} = \mu\bb L \bm V
\]
be a sequence of relatively unary $\mVF^\diamond[k]$-morphisms, which exist by definition. For each $j \leq l - 2$, we can choose five standard contractions
\[
[\bm U_j]_{\leq k}, \quad [\bm U_{j+1}]_{\leq k}, \quad [\bm U'_{j+1}]_{\leq k}, \quad [\bm U''_{j+1}]_{\leq k}, \quad [\bm U_{j+2}]_{\leq k}
\]
of $\bm B_j$, $\bm B_{j+1}$, $\bm B_{j+1}$, $\bm B_{j+1}$, $\bm B_{j+2}$ with the permutations $\sigma_{j}$, $\sigma_{j+1}$, $\sigma'_{j+1}$, $\sigma''_{j+1}$, $\sigma_{j+2}$ of $[k]$, respectively, such that
\begin{itemize}
  \item $\sigma_{j+1}(1)$ and $\sigma_{j+1}(2)$ are the $\VF$-coordinates targeted by $G_{j}$ and $G_{j+1}$, respectively,
  \item $\sigma''_{j+1}(1)$ and $\sigma''_{j+1}(2)$ are the $\VF$-coordinates targeted by $G_{j+1}$ and $G_{j+2}$, respectively,
  \item $\sigma_{j} = \sigma_{j+1}$, $\sigma''_{j+1} =  \sigma_{j+2}$,  and $\sigma'_{j+1}(1) = \sigma''_{j+1}(1)$,
  \item the relation between $\sigma_{j+1}$ and $\sigma'_{j+1}$ is as described in Lemma~\ref{contraction:perm:pair:isp}.
\end{itemize}
By Corollary~\ref{contraction:same:perm:isp} and Lemma~\ref{contraction:perm:pair:isp}, all the adjacent pairs of these standard contractions are indeed $\mispdb$-congruent, except $([\bm U'_{j+1}]_{\leq k}, [\bm U''_{j+1}]_{\leq k})$.  Since we can choose $[\bm U'_{j+1}]_{\leq k}$, $[\bm U''_{j+1}]_{\leq k}$ so that they start with the same contraction in the first targeted $\VF$-coordinate of $\bm B_{j+1}$, the resulting sets from this step are the same. Thus, applying  the inductive hypothesis in each fiber over the just contracted coordinate, we see that this last pair is also $\mispdb$-congruent. This completes the ``only if'' direction.
\end{proof}

\begin{rem}\label{thekernel}
Proposition~\ref{kernel:L} shows that  the kernel  of $\mL$ in $\gsk \mRV^{\db}[*]$ is generated by the pairs $([1], [\RV^{\ccirc}(t)_\gamma])$, where $\gamma \in \Gamma$ and $t \in \gamma^\sharp$ are definable (Notation~\ref{punc:RVd}), and hence  the corresponding ideal of the graded ring $\ggk \mRV^{\db}[*]$  is $(\bm P_\Gamma)$ (Notation~\ref{pgammabis}).
\end{rem}

\begin{thm}\label{main:prop}
For each $k \geq 0$ there is a canonical isomorphism of Grothendieck semigroups
\[
  \int^{\diamond}_+ : \gsk \mVF^\diamond[k] \fun \gsk  \mRV^{\db}[k] /  \mispdb
\]
such that $\int_{+}^{\diamond} [\bm A] = [\bm U]/  \mispdb$ if and only if $[\bm A] = [{\mL}\bm U]$. Taking the direct sum of the groupifications
yields a canonical isomorphism of graded rings
\[
  \int^{\diamond} : \ggk \mVF^\diamond[*] \fun \ggk  \mRV^{\db}[*] /  (\bm P_\Gamma).
\]
\end{thm}
\begin{proof}
This is immediate by Corollary~\ref{L:sur:c} and Proposition~\ref{kernel:L}.
\end{proof}

\begin{cor}\label{db:to:novol}
The diagram (\ref{diag-diamond-interpol}) commutes.
\end{cor}
\begin{proof}
Note that, in (\ref{diag-diamond-interpol}),  the first horizontal arrow  is induced by the obvious forgetful functor (it exists
because the morphisms in $\mgVF^\diamond[*]$ are bijections instead of essential bijections, see Remark~\ref{dia:cat:well}), the second horizontal arrow is induced by the subcategory relation, and the reason that the bottom row exists is given in Notation~\ref{pgammabis}. That the upper half of the diagram commutes is then simply a consequence of the fact that image of $[\bm A]$ under any of the three integrals  can be represented by a standard contraction of $\bm A$. The rest of it is just a combination of (\ref{edb:eb:com}) and Proposition~\ref{prop:eu:retr:k}.
\end{proof}

\begin{rem}\label{eg:eb:factor}
We may define the motivic volume operator in Notation~\ref{mov:vol} using $\bb E_g$ instead of $\bb E_b$ (of course $[\A]$ will have to be inverted, see Remark~\ref{eb:no:eg}).
For proper invariant objects, this only changes the image by a factor. To see this, first note that  Theorem~\ref{main:prop} (or rather Corollary~\ref{L:sur:c}) shows that every proper invariant object $A \in \VF_*$ with $\dim_{\VF}(A) = n$ is isomorphic to an object of the form $\bb L \bm U$ with $(\bm U, 0) \in  \mgRV^{\db}[n]$. Therefore, in light of the isomorphism ${\mu}\Psi$ in (\ref{bdd:to:2bdd}) and the defining conditions of $\bb E_g$, $\bb E_b$ in \S~\ref{subs:uni:retr:novol}, we have
\begin{equation}\label{off:fac:A}
\bb E_g \Big (\int [A] \Big ) = \bb E_b \Big (\int [A] \Big )[\A]^{-n}.
\end{equation}
\end{rem}

\section{Motivic Milnor fiber}\label{sec:spec:pui}

We begin with a very brief discussion on specialization to henselian subfields. Let $\mathbb{M}$ be a $\VF$-generated  substructure. Recall that $\mathbb{S}$ is a part of the language and hence all other substructures contain it. If $X \sub \VF^n  \times \RV^m$ is a definable (and hence quantifier-free definable) set then the \memph{trace} of $X$ in $\mathbb{M}$, denoted by $X(\mathbb{M})$, is the set of $\mathbb{M}$-rational points of $X$, that is,
\[
X(\mathbb{M}) = X \cap (\VF(\mathbb{M})^n  \times \RV(\mathbb{M})^m).
\]
Such a trace is also called a \memph{constructible set} in $\mathbb{M}$ since it is indeed quantifier-free definable in $\mathbb{M}$. Note that, however, if $f : X \fun \Gamma$ is a definable function then the image $f(X(\mathbb{M}))$  is not necessarily a set in $\Gamma(\mathbb{M})$, but rather a set in the divisible hull $\Q \otimes \Gamma(\mathbb{M})$  of $\Gamma(\mathbb{M})$. For instance, if $\mathbb{M} = \C \dpar t$ then $\Gamma(\mathbb{M}) = \Z$ and hence $\gamma \in \Gamma$ is definable if and only if $\gamma \in \Q \otimes \Gamma(\mathbb{M}) = \Q$. On the other hand, if $X$ is a set in $\Gamma$ and $f$ is a piecewise $\mgl_k(\Z)$-transformation on $X$ with constant $\Gamma(\mathbb{M})$-terms  then $f(X(\mathbb{M}))$ is of course a set in $\Gamma(\mathbb{M})$; this is the situation in the $\Gamma$-categories (see Remark~\ref{GLZ:char}).


If $\mathbb{M}$ is definably closed and $\Gamma(\mathbb{M})$ is nontrivial, or equivalently (see Lemma~\ref{cut:to:hensel:substru}), the valued field $(\VF(\mathbb{M}), \OO(\mathbb{M}))$  is henselian, then $\mathbb{M}$ is \memph{functionally closed}, that is,  for any definable set $X$ and any definable function $f$ on $X$, the image $f(X(\mathbb{M}))$ is a set that is definable in $\mathbb{M}$ (which then, ex post facto, is  constructible in $\bb M$). This is all we need to deduce the results below.

\subsection{Piecewise retraction to $\RES$}\label{piece:res}

As in \S~\ref{jp:fubini}, from here on, we work in the field $\puC$ of complex Puiseux series (with an implicit reduced cross-section) and take  $\mathbb{S} = \C \dpar{ t }$. The value group $\Gamma$ is identified with $\Q$. We also consider the family of henselian subfields $\C \dpar{ t^{1/m} }$,  $m \in \Z^+$, regarded as substructures. The value group $\Gamma(\C \dpar{ t^{1/m} })$ is identified with $m^{-1} \Z$.

\begin{nota}\label{loc:rings}
For any ring $R$, let $R[T^{\Q}]$ denote the ring of Puiseux polynomials over $R$, that is, the group ring of $\Q$ over $R$.

Let $\gmv[[\A]^{-1}][T,T^{-1}]$ be the obvious subring of $\gmv[[\A]^{-1}][T^{\Q}]$. The canonical image of $\gmv [T^{\Q}]$ in  $\gmv[[\A]^{-1}][T^{\Q}]$ is just denoted as such. The assignment $T \efun [\A]$ determines a ring homomorphism
\begin{equation}\label{pui:laurent}
  \bm \eta : \gmv[[\A]^{-1}][T,T^{-1}] \fun \gmv[[\A]^{-1}].
\end{equation}
\end{nota}


Recall Notations~\ref{comp:form} and \ref{bipo:gdv}. Let $\bm U = (U, f, \omega) \in \mgRV^{\db}[k]$. In particular, if $f=(f_1,\dots,f_n)$, we have introduced the shorthand 
\[
\omega_{f}: x\in U \efun \omega(x)+\sum_{i=1}^n\vrv(f_i(x)).
\]
The image $\omega_f(U_\gamma)$ is finite for every $\gamma \in \vrv(U)$, because the $\K$-sort and the $\Gamma$-sort are orthogonal to each other. It follows that there is a bipolar twistoid decomposition $(U_i)_i$ of $U$ such that each restriction $\omega_f \rest U_i$ $\vrv$-contracts to a function $\sigma_i : I_i = \vrv(U_i) \fun \Gamma$. Write $I_{i,m}$ for the set $I_i(m^{-1}\Z)$ of  $m^{-1}\Z$-rational points of $I_i$. We assign to $\bm U$ the expression
\begin{equation}\label{hm:ass}
  h_m(\bm U) = \sum_i [\tbk(U_i)]^{\hat \mu} \sum_{\gamma \in I_{i,m}} T^{-m \sigma_i(\gamma)},
\end{equation}
which is  a finite sum and hence  belongs to $\gmv[[\A]^{-1}][T^{\Q}]$.

If $\bm U$ is already of the form $(V, g, \omega') \times (I, \sigma) \in \mgRES[i] \times \mG^{\db}[j]$, where $\vrv(g)$ is a singleton, $\vrv(V) \in m^{-1}\Z$, and $\omega'_g$ is constant, then clearly
\[
h_m(\bm U) = \Theta([V])T^{-m\omega'_g} \cdot [\G_m]^{j} \sum_{\gamma \in I_{m}} T^{-m (\sigma(\gamma) + \Sigma \gamma )};
\]
note that if $\vrv(V) \notin m^{-1}\Z$ then simply  $h_m(\bm U) = 0$. In light of the isomorphism ${\mu} \Psi^{\db}$ in Remark \ref{tensor:bd}, we may as well take this as the definition of $h_m$, which is essentially how it is done in \cite[\S~8.2]{hru:loe:lef}. Computationally, though,  (\ref{hm:ass}) is more effective.

The ``piecewise retraction'' formula (\ref{hm:ass}) is indeed quite similar to the ``uniform retraction'' formula (\ref{twist:C}). The difference is just that the coefficient $\chi_b(I_i)$ in the latter is replaced in the former by a formal ``counting'' sum $\sum_{\gamma \in I_{i,m}} T^{-m \sigma_i(\gamma)}$ that also takes the volume form into account. As mentioned in Remark~\ref{bipo:desc}, that the ``uniform retraction'' formula does not depend on the choice of the bipolar twistoid decomposition is simply a consequence of the construction of $\bb E_b$ itself. The same holds for the ``piecewise retraction'' formula, but we do need an argument.

\begin{lem}\label{muhat:section}
The  assignment (\ref{hm:ass}) does not depend on the choice of the bipolar twistoid decomposition and is invariant on isomorphism classes.
\end{lem}
\begin{proof}
Let $\bm V = (V, g, \pi) \in \mgRV^{\db}[k]$  and $f : U \fun V$ be a $\mgRV^{\db}[k]$-morphism. Let $D = (U_i)_i$, $E = (V_i)_i$ be bipolar twistoid decompositions of $U$, $V$ satisfying the condition above. We need to show that $h_m(\bm U)$, which depends on $D$, and $h_m(\bm V)$, which depends on $E$, are  equal.
This is clear if $\bm U = \bm V$,  and $E$ is trivial (so $U$ is already a bipolar twistoid and $\omega_f$ already $\vrv$-contracts to a function on $\vrv(U)$), and $D$ is $\Gamma$-cohesive or $\vrv(U_i) = \vrv(U_j)$ for all $i$, $j$. The case that $\bm U = \bm V$ and $D$ is a refinement of $E$  follows easily from this since there is a refinement $(U_{ij})_{ij}$ of $D$ such that $(U_{ij})_{j}$ is a $\Gamma$-cohesive twistoid decomposition of $U_i$  and $\vrv(U_{ij}) = \vrv(U_{i'j})$ for all $i$, $i'$. Finally,  if $\omega_f$, $\pi_g$ both $\vrv$-contract to a function and $f$ $\vrv$-contracts to a bijection then, by Lemma~\ref{coh:decom}, we may assume that $U$, $V$ are already bipolar twistoids. In that case the desired equality follows because $\C \dpar{ t^{1/m} }$ is functionally closed.

For the general case, we first remark that the image $\vrv(f(U_\gamma))$ is finite for every $\gamma \in \vrv(U)$. It then follows from Lemmas~\ref{shift:K:csn:def} and~\ref{coh:decom} that there is a twistoid decomposition $(f_i)_i$ of $f$ such that every $f_i$ is a $\mgRV^{\db}[k]$-morphism as in the last special case considered above and, moreover, is compatible with $D$, $E$ in the obvious sense, in other words, the domains and ranges of these $f_i$ induce bipolar refinements of $D$, $E$.  The result follows.
\end{proof}

Thus, $h_m$ may be viewed as a map on $\gsk \mgRV^{\db}[k]$. It is now routine to check that  we have in effect constructed a ring homomorphism
\[
\bm h_m : \ggk \mgRV^{\db}[*] \fun  \ggk^{\hat \mu} \var_{\C}[[\A]^{-1}][T^{\Q}].
\]

Let $\ggk^\natural_m \mgRV^{\db}[*]$ denote the subring $(\bm h_m)^{-1}(\gmv[[\A]^{-1}][T, T^{-1}])$ of $\ggk \mgRV^{\db}[*]$.

\begin{lem}\label{PG:van}
The homomorphism
\[
\bm \eta \circ \bm h_m : \ggk^\natural_m \mgRV^{\db}[*] \fun \gmv[[\A]^{-1}]
\]
vanishes on $(\bm P_{\Gamma})$.
\end{lem}
Note that the ideal $(\bm P_\Gamma)$ of $\ggk \mgRV^{\db}[*]$ in Notation~\ref{pgammabis} is now generated by the elements $\bm P_\gamma$ with $\gamma \in \Z$ (because the point $t_\gamma$ there needs to be definable, which is possible only if $\gamma \in \Z$ in the current setting).

\begin{proof}
For $\gamma \in \Z$, the image of $[\RV^{\circ \circ} \mi \RV^{\circ \circ}_{\gamma}] + [\{t_\gamma\}]$ under $\bm h_m$ in $\gmv[[\A]^{-1}][T, T^{-1}]$ is
\[
 ([\A] - 1) \sum_{i = 1}^{m\gamma} T^{-i} + T^{-m\gamma},
\]
which, after passing to $\gmv[[\A]^{-1}]$ via $\bm \eta$, becomes $1 = \bm \eta (\bm h_m([1]))$.
\end{proof}

\begin{rem}\label{dia:to:db:dag}
If $\bm U = (U, f, l) \in \mgRV^{\db}[*]$ with $l \in m^{-1}\Z$ (constant volume form) then the exponents in (\ref{hm:ass}) are all integers and hence $[\bm U] \in \ggk^\natural_m \mgRV^{\db}[*]$. Actually we shall only need the case $l = 0$.

The ring $\bigcap_{m \in \Z^+} \ggk^\natural_m \mgRV^{\db}[*]$ is denoted by $\ggk^\natural \mgRV^{\db}[*]$.

If $\bm A = (A, l) \in \mgVF^\diamond[*]$ with $l$  constant then $\int^\diamond [\bm A]$ may be expressed as $[(U, f, l)]/ (\bm P_{\Gamma})$, and hence if $l \in \Z$ then $\int^\diamond [\bm A]$ belongs to $\ggk^\natural \mgRV^{\db}[*]/ (\bm P_{\Gamma})$. In that case, by Lemma~\ref{PG:van}, for every $m \in \Z^+$, the expression $(\bm \eta \circ \bm h_m \circ \int^\diamond)([\bm A])$ designates a unique element  in $\gmv[[\A]^{-1}]$.
\end{rem}

Denote by $\RES_{m}$ the full subcategory of $\RES$ such that $U \in \RES_{m}$ if and only if every $\gamma \in \vrv(U)$ is a tuple in $m^{-1} \Z$, or equivalently, $U \in \RES_{m}$ if and only if the action on $U$ of the kernel of the canonical projection $\hat \mu \fun \mu_m$ is trivial.

Let $\beta = (\beta_1, \ldots, \beta_n) \in (m^{-1} \Z)^n$ and $A \sub \OO^n \times \RV^l$ be a proper $\beta$-invariant definable set. Then there is a set
\[
A[m;\beta] \sub \prod_i \C[t^{1/m}] / t^{\beta_i + 1/m} \times \RV(\C \dpar{t^{1/m}})^l
\]
such that, for every $t \in \RV(\C \dpar{t^{1/m}})^l$, the $\VF$-fiber $A_t(\C \dpar{t^{1/m}})$ is, under the obvious quotient map, the preimage of the fiber $A[m;\beta]_t$.  We can think of $\C[t^{1/m}] / t^{\beta_i + 1/m}$ as the object $\prod_{0 \leq \gamma \leq \beta_i} \gamma^\sharp$ of $\RES_{m}$. Then, since every element in the value group $\Q$ is definable and $\vrv(A[m;\beta])$ is doubly bounded,  it follows that  $A[m;\beta]$ is indeed a finite disjoint union of definable sets in $\RV$ (no extra parameters other than those in $\mathbb{S} = \C \dpar{ t }$ are needed) and hence, as such, is an object of $\RES_{m}$ (see \cite[\S~4.2]{hru:loe:lef} for more details).

\begin{lem}[{\cite[Lemma~4.2.1]{hru:loe:lef}}]
Let $\beta' \in (m^{-1} \Z)^n$ with $\beta_i \leq \beta'_i$ for all $i$. Then
\begin{equation}\label{arc:com}
  [A[m;\beta']] = [A[m;\beta]][\A]^{m \Sigma (\beta' - \beta)} \in \ggk \RES_{m}.
\end{equation}
\end{lem}

Thus the element $[A[m;\beta]][\A]^{- m\sum \beta }$ in $\ggk \RES_{m}[[\A]^{-1}]$ does not depend on $\beta$; we denote it by $\tilde A[m]$.

The localization of $\Theta$ at $[\A]$ is still denoted by $\Theta$.

\begin{lem}\label{inv:dir:com}
Suppose that $(A, 0) \in \mgVF^\diamond[n]$. Then, in $\gmv[[\A]^{-1}]$,
\[
(\bm \eta \circ \bm h_m) \Big( \int^\diamond [(A, 0)] \Big) = \Theta(\tilde A[m]).
\]
\end{lem}
\begin{proof}
By Lemma~\ref{L:measure:surjective}, there exists a special covariant bijection $F : (A, 0) \fun (\bb L \bm U, 0)$, where $\bm U \in \RV^{\db}[*]$. Observe that we can choose apertures $(\gamma_i, t_i)$ with $\gamma_i \in m^{-1} \Z$ so that $F$ induces a definable bijection between $A[m;\beta]$ and  $\bb L \bm U[m;\beta]$ for all sufficiently large $\beta$; it is for this purpose that each $\gamma_i$ needs to lie in $m^{-1} \Z$, due to the second clause of (\ref{centri}) and the mandatory regularizations in Definition~\ref{defn:special:bijection}. Thus, we may assume that $A$ is already of the form $\bb L \bm U$. Let $D$ be a $\Gamma$-cohesive bipolar twistoid decomposition of $\bm U$. Since both sides respect finite disjoint union of proper $\beta$-invariant definable sets, we may further assume that $D$ is actually trivial or even $\vrv(\bm U)$ is a singleton. Then the equality follows from a simple computation.
\end{proof}

\subsubsection{Bounded version}
\label{sec:hm:bounded}
Although not needed, we can  extend the construction above to the category $\mgRV^{\bdd}[*]$; this category is formulated as $\mgRV^{\db}[*]$, but with doubly bounded objects replaced by bounded ones.

Let $\gmv[[\A]^{-1}][T^{\Q}]_{\loc}$ be the localization of $\gmv[[\A]^{-1}][T^{\Q}]$ by the multiplicative family generated by the elements $1 - T^{-i}$, $i \in \Z^+$. The ring $\gmv[[\A]^{-1}][T, T^{-1}]_{\loc}$ is the localization of $\gmv[[\A]^{-1}][T, T^{-1}]$ by the same multiplicative family. Similarly one defines the ring $\gmv[[\A]^{-1}]_{\loc}$ with respect to the multiplicative family generated by the elements $1 - [\A]^{-i}$, $i \in \Z^+$. As in (\ref{pui:laurent}), the assignment the $T \efun [\A]$ determines a ring homomorphism
\begin{equation}\label{pui:laurent:loc}
  \bm \eta_{\loc} : \gmv[[\A]^{-1}][T, T^{-1}]_{\loc} \fun \gmv[[\A]^{-1}]_{\loc}.
\end{equation}

We assign to each $\bm U= (U, f, \omega) \in \mgRV^{\bdd}[k]$ the same expression as in (\ref{hm:ass}), which is now a formal Laurent series in $T^{-1/l}$ for some integer $l > 0$, because each $\sigma_i$ is piecewise $\Q$-linear  and its graph is bounded. By \cite[Lemma~8.2.1]{hru:loe:lef}, if $l =1$ then the series belongs to $\gmv[[\A]^{-1}][T, T^{-1}]_{\loc}$ and hence,  in general, it belongs to $\gmv[[\A]^{-1}][T^{\Q}]_{\loc}$ (consider any isomorphism of $\puC$ that fixes $\C$ and sends $1/l$ to $1$).  So the assignment  yields a ring homomorphism $\bm h_{\loc,m}$ and consequently a subring $\ggk^\natural_m \mgRV^{\bdd}[*]$.

\begin{lem}
\label{lem:hm:bounded}
The homomorphism
\[
\bm \eta_{\loc} \circ \bm h_{\loc,m} : \ggk^\natural_m \mgRV^{\bdd}[*] \fun \gmv[[\A]^{-1}]_{\loc}
\]
 vanishes on $(\bm P)$.
\end{lem}
\begin{proof}
The image of $[\RV^{\ccirc}]$ in $\gmv[[\A]^{-1}][T, T^{-1}]_{\loc}$ is
\[
 ([\A] - 1) \sum_{i > 0} T^{-i} = ([\A] - 1) / (T-1)
\]
and hence, after passing to $\gmv[[\A]^{-1}]_{\loc}$ via $\bm \eta_{\loc}$, it becomes $1 = \bm \eta_{\loc}( \bm h_{\loc,m}([1]))$.
\end{proof}

\subsection{Constructing motivic Milnor fiber}\label{sec:milnor}

Finally, we proceed to replicate the construction in \cite[\S~8]{hru:loe:lef} so to recover the motivic zeta function with coefficients in $\gmv[[\A]^{-1}]$ (this is the purpose of Lemma~\ref{inv:dir:com}) and the corresponding motivic Milnor fiber.

Let $X$ be a smooth connected complex algebraic variety of dimension $d$ and $f : X \fun \A^1$ a nonconstant regular function. Suppose that $0 \in X(\C) \cap f^{-1}(0)$ is singular. For $m \geq 1$, we consider the set of truncated arcs
\begin{equation}\label{f:m:arc}
\mdl X_{m} = \set{ \varphi \in X(\C[t] / (t^{m+1})) \given f(\varphi) = t^m \mod t^{m+1}},
\end{equation}
which may be viewed in a natural way as the set of closed points of an algebraic variety over $\C$ that carries a natural $\mu_{m}$-action, and the so-called \memph{nonarchimedean Milnor fiber}
\[
\mdl X_f = \set{ x \in X(\MM) \given \rv(f(x)) = \rv(t)}.
\]
Note that $\mdl X_f$ may be constructed in an affine neighborhood of $0$ and hence is indeed a (quantifier-free) definable set. Moreover, it is $\beta$-invariant for every $\beta \geq 1$. Therefore, $\mdl X_f$ may be regarded as an object of $\mgVF^\diamond[*]$ equipped with the constant volume form $0$. The  \memph{motivic zeta function} of $f$ is then the power series
\begin{equation}\label{def:zeta}
Z_f(T) = \sum_{m \in \Z^+} \Big(\bm \eta \circ \bm h_m \circ \int^{\diamond} \Big)([\mdl X_f]) T^m \in \gmv[[\A]^{-1}] \dbra T.
\end{equation}
By Remark~\ref{dia:to:db:dag} and Lemma~\ref{inv:dir:com}, we have, in $\gmv[[\A]^{-1}]$,
\begin{equation}\label{z:coeff:change}
 \Big(\bm \eta \circ \bm h_m \circ \int^{\diamond} \Big)([\mdl X_f]) = \Theta(\tilde{\mdl X_f}[m]) = [\mdl X_{m}][\A]^{-md},
\end{equation}
where the second equality follows if  we change the variable $t$ to $t^{1/m}$ in (\ref{f:m:arc}) (see the discussion in \cite[\S~6.2]{hru:loe:lef} for  detail). This shows that the coefficients of $Z_f(T)$ may also be written as $[\mdl X_{m}][\A]^{-md}$, which is how they are usually defined.

\begin{rem}
One of the key steps of the main construction in \cite[\S~8]{hru:loe:lef} is to show (\ref{z:coeff:change}), see \cite[Proposition~8.3.1]{hru:loe:lef}, but with $\int^\diamond$  replaced by $\int$. This, and a few other statements in \cite[\S~8]{hru:loe:lef}, rely on a tensor decomposition of the ring $\ggk \mathrm{vol}\RV[*]$, as claimed in \cite[Proposition~10.10(2)]{hrushovski:kazhdan:integration:vf}. We have pointed out above that, unfortunately, this result does not hold.
\end{rem}

\begin{nota}\label{dag:loc}
Let $\gmv[[\A]^{-1}][T]_{\dag}$ be the localization of $\gmv[[\A]^{-1}][T]$ by the multiplicative family generated by the elements $1 - [\A]^{a}T^b$, where $a \in \Z$ and $b \in \Z^+$; it may be regarded as a subring of  $\gmv[[\A]^{-1}] \dbra T$.
\end{nota}

It is known that $Z_f(T)$ belongs to $\gmv[[\A]^{-1}][T]_{\dag}$ and, letting ``$T$ go to infinity'' as described in \cite[\S~8.4]{hru:loe:lef}, we get a limit
\[
\mathscr S_f \coloneqq - \lim_{T \limplies \infty} Z_f(T) \in \gmv[[\A]^{-1}],
\]
which is understood as the \memph{motivic Milnor fiber} attached to $f$.

\begin{rem}
Recall the homomorphism $\bb E^{\diamond}$ from (\ref{edb:eb:com}).
By Lemma~\ref{coh:decom} and Notation~\ref{bipo:gdv}, we can construct the homomorphism
\[
\Theta \circ \bb E^{\diamond} : \ggk \mgRV^{\db}[*] \fun \gmv
\]
simply using the expression in (\ref{twist:C}) with respect to  bipolar twistoid decompositions.
\end{rem}

\begin{thm}\label{direct:mil}
$\mathscr S_f = (\Theta \circ \bb E^{\diamond} \circ \int^{\diamond})([\mdl X_f])$.
\end{thm}
\begin{proof}
Let $[\bm U] = [(U, f, l)] \in \ggk^\natural \mgRV^{\db}[*]$ with $l \in \Z$  and consider the zeta function
\begin{equation}\label{U:zeta}
Z([\bm U])(T) = \sum_{m \in \Z^+} (\bm \eta \circ \bm h_m)([\bm U]) T^m \in \gmv[[\A]^{-1}]\dbra T.
\end{equation}
If $[\bm U] / (\bm P_\Gamma) = \int^{\diamond}([\mdl X_f])$ then this is $Z_f(T)$. Thus it is enough to show that  $Z([\bm U])(T)$ belongs to $\gmv[[\A]^{-1}][T]_{\dag}$, and $\lim_{T \limplies \infty} Z([\bm U])(T)$ exists and equals $- (\Theta \circ \bb E^{\diamond})([\bm U])$.

Without loss of generality, we may assume that $U$ is already a bipolar twistoid and the function $l_f: U \fun \Gamma$ (Notation~\ref{comp:form}) $\vrv$-contracts to a function $\sigma : I = \vrv(U) \fun \Gamma$.  Set $v = [\tbk(U)]^{\hat \mu}$. Then
\[
(\Theta \circ \bb E^{\diamond})([\bm U]) = \chi(I) v.
\]
Let $Z(v)(T) = \sum_{m \in \Z^+} v T^m$. This is an element in $\gmv[[\A]^{-1}][T]_{\dag}$ and $\lim_{T \limplies \infty}Z(v)(T) = -v$. Write  $I_{m} = I(m^{-1}\Z)$ and let
\[
Z(\sigma)(T) = \sum_{m \in \Z^+} \sum_{\gamma \in I_{m}} [\A]^{-m \sigma(\gamma)} T^m.
\]
By Remark~\ref{GLZ:char}, $\sigma$ is $\Z$-linear. So, by \cite[Proposition~8.5.2]{hru:loe:lef}, $Z(\sigma)(T)$ is in $\gmv[[\A]^{-1}][T]_{\dag}$ and $\lim_{T \limplies \infty}Z(\sigma)(T) = -\chi(I)$. Since $Z([\bm U])(T)$ is the Hadamard product of $Z(v)(T)$ and $Z(\sigma)(T)$, by \cite[Lemma~8.4.1]{hru:loe:lef}, $\lim_{T \limplies \infty}Z([\bm U])(T) = - \chi(I) v$.
\end{proof}

\begin{rem}
The construction above no longer needs to go through this additional localization process ``$\loc$'' employed throughout \cite[\S~8]{hru:loe:lef}.
\end{rem}

\begin{rem}
Using $\bm \eta_{\loc} \circ \bm h_{\loc,m}$ instead of $\bm \eta \circ \bm h_m$, we can define the zeta function $Z([\bm U])(T)$ for every $[\bm U] \in \ggk^\natural \mgRV^{\bdd}[*]$, and the proof of Theorem~\ref{direct:mil} shows that $\lim_{T \limplies \infty} Z([\bm U])(T)$ exists and belongs to $\gmv[[\A]^{-1}]_{\loc}$. This and  Lemma \ref{lem:hm:bounded} gives the right half of the diagram (\ref{diag-diamond-zeta-function}). However,  we are not  claiming  that $\lim_{T \limplies \infty} Z([\bm U])(T)$ equals $(\Theta \circ \bb E_b)( [\bm U])$, which is reflected in the fact that there is no commutative relation between the left and right columns in (\ref{diag-diamond-zeta-function}). In fact, they are clearly not equal, as seen for example by considering $\bm U = \RV^{\circ\circ}$ (then the first value is $-1$ and the second one is $0$).  Such a failure explains the need to introduce an interpolant, such as the composite map $- \lim \circ Z \circ \int^\diamond$, in (\ref{diag-diamond-zeta-function}), for which Theorem \ref{direct:mil} may be formulated and proven.
\end{rem}

\begin{rem}\label{rem-comp-milnor-fib}
The expression $(\Theta \circ \bb E^{\diamond})([\bm U])$, as an element in $\gmv[[\A]^{-1}]$, does not actually involve  $[\A]^{-1}$. Still, since the coefficients of $Z([\bm U])(T)$ does involve $[\A]^{-1}$ and it is a fact that the natural homomorphism $\gmv \fun \gmv[[\A]^{-1}]$ is not injective, we cannot really take  the motivic Milnor fiber $\mathscr S_f$ of $f$ in $\gmv$, at least not if $\mathscr S_f$ is viewed as something obtained through  $Z_f(T)$. It is this point of view that forces us to work with an integral whose target only involves doubly bounded sets in $\RV$, namely $\int^{\diamond}$, instead of $\int$, so as to facilitate the computation of the coefficients of $Z_f(T)$, and consequently with the nonarchimedean Milnor fiber $\mdl X_f$, which is proper invariant, instead of, perhaps, the more obvious definable set
\[
\mdl X_t = \set{x \in X(\MM) \given f(x) = t},
\]
which is not proper invariant. This set $\mdl X_t$ may be understood as containing the $\puC$-rational points of  the analytic Milnor fiber introduced in \cite{nicaise:sebag:motivic:serre} (see Remark~\ref{rem-comp-milnor-fib-bis} below for more) and does play a role in \cite{hru:loe:lef} (but not in this paper).

On the other hand, in light of Theorem~\ref{direct:mil}, we can forego the zeta function point of view and recover $\mathscr S_f$ directly as $(\Theta \circ \bb E^{\diamond} \circ \int^{\diamond})([\mdl X_f])$. In that case there is truly no need to invert $[\A]$. In fact, we can also recover $\mathscr S_f$  directly as  $\Vol([\mdl X_f])$ (Notation~\ref{mov:vol}), the result is the same because the left half of the diagram (\ref{diag-diamond-interpol}) commutes (Corollary~\ref{db:to:novol}).
\end{rem}

\begin{rem}\label{rem-comp-milnor-fib-bis}
For any $t'\in \rv(t)^\sharp$, there exists an immediate automorphism $\sigma$ of $\puC$ over $\C$ with $\sigma(t') = t$, where ``immediate'' just means that $\sigma$ fixes $\RV(\puC)$ pointwise. Therefore, $\int[\mdl X_{t'}]=\int[\mdl X_t]$. Since $\mdl X_f = \bigcup_{t'\in \rv(t)^\sharp} \mdl X_{t'}$, it follows from compactness  that
\begin{equation*}
\int [\mdl X_f] = \int [\rv(t)^\sharp] \int [\mdl X_t] =  [1]\int [\mdl X_t].
\end{equation*}
Since $\bb E_b([1]) = 1 $, this and Remark \ref{rem-comp-milnor-fib} show that
\[
\mathscr S_f  = \Vol([\mdl X_f]) = \Vol([\mdl X_t]).
\]

This is first shown in \cite{nicaise_ps_tropical}, but without taking the $\hat \mu$-actions into account, and also in \cite{Nic:Pay:trop:fub, forey_virtual_2017}. The arguments there all rely on resolution of singularities as well as other algebro-geometric machineries.
\end{rem}

\providecommand{\bysame}{\leavevmode\hbox to3em{\hrulefill}\thinspace}
\providecommand{\MR}{\relax\ifhmode\unskip\space\fi MR }
\providecommand{\MRhref}[2]{%
  \href{http://www.ams.org/mathscinet-getitem?mr=#1}{#2}
}
\providecommand{\href}[2]{#2}

\end{document}